\pdfoutput=1
\documentclass{siamart0516}
\usepackage{mathtools,amssymb,amsfonts,url,xspace,graphicx}
\usepackage{blkarray}
\usepackage{microtype}
\usepackage{tikz}
\usepackage{textgreek}
\usepackage{subcaption}
\hypersetup{
pdftitle={The space of circular planar electrical networks},pdfauthor={Richard W. Kenyon and David B. Wilson},bookmarks=true,bookmarksopen=true,bookmarksopenlevel=2,unicode=true
}
\usepackage[all]{hypcap}
\begin{document}
\title{The space of circular planar electrical networks}
\author{
Richard W. Kenyon\thanks{Brown University.  Research partially supported by Microsoft Research and partially supported by NSF grant DMS-1208191 and Simons Foundation grant 327929.}
\and
David B.\! Wilson\thanks{Microsoft Research.}
}
\headers{The space of circular planar electrical networks}{Richard W. Kenyon and David B.\! Wilson}
\maketitle

\newcommand{\qedhere}{}

\newtheorem{conjecture}{Conjecture}
\newcommand{\Z}{\mathbb{Z}}
\newcommand{\Q}{\mathbb{Q}}
\newcommand{\R}{\mathbb{R}}
\renewcommand{\C}{\mathbb{C}}
\newcommand{\N}{\mathbb{N}}
\renewcommand{\G}{\mathcal{G}}
\newcommand{\BP}{\text{BP}}
\newcommand{\ZZ}{\mathcal{Z}}
\newcommand{\Pf}{\mathrm{Pf}}
\newcommand{\No}{\mathcal{N}}
\newcommand{\eps}{\epsilon}
\newcommand{\ep}{\varepsilon}
\newcommand{\old}[1]{}
\renewcommand{\t}{{\text t}}
\newcommand{\sym}{\text{sym}}
\newcommand{\asym}{\text{asym}}
\newcommand{\be}{\begin{equation}}
\newcommand{\ee}{\end{equation}}
\newcommand{\w}{{\mathrm w}}
\renewcommand{\b}{{\mathrm b}}
\newcommand{\sign}{\operatorname{sign}}
\newcommand{\note}[1]{{\color{red} #1}}
\newcommand{\Pu}{\ddddot\Pr}
\newcommand{\Pfd}{\operatorname{Pfd}}
\newcommand{\IG}{\operatorname{IG}}
\newcommand{\CM}{\mathrm{CM}}
\newcommand{\AD}{\mathrm{AD}}
\newcommand{\TAD}{\mathrm{TAD}}
\newcommand{\td}[1]{{\color{red}{\textbf {[#1]}}}}
\newcommand{\red}[1]{{\color{red} #1}}
\newcommand{\green}[1]{{\color{black!50!green} #1}}
\newcommand{\blue}[1]{{\color{blue} #1}}

\newcommand{\minor}[2]{\begin{tikzpicture}[scale=0.6,baseline=0cm-2.5pt]
    \foreach \x in {1,...,#1} { \coordinate (\x) at ({1.*cos(\x*360/#1)},{1.*sin(\x*360/#1)});}
    \draw[fill=yellow,draw=none] (0,0) circle(1.0);
    \draw [thick] \foreach \x/\y in {#2} {(\x)--(\y)};
    \foreach \x in {1,...,#1} {\node [circle,fill=orange!60!yellow,inner sep=.3pt] at (\x) {$\scriptstyle\x$};}
\end{tikzpicture}}

\newcommand{\minorsize}[3]{\begin{tikzpicture}[scale=#3,baseline=0cm-2.5pt]
    \foreach \x in {1,...,#1} { \coordinate (\x) at ({1.*cos(\x*360/#1)},{1.*sin(\x*360/#1)});}
    \draw[fill=yellow,draw=none] (0,0) circle(1.0);
    \draw [thick] \foreach \x/\y in {#2} {(\x)--(\y)};
    \foreach \x in {1,...,#1} {\node [circle,fill=orange!60!yellow,inner sep=.3pt] at (\x) {\scalebox{0.9}{$\scriptstyle\x$}};}
    \foreach \x/\y in {#2} {
      \node [circle,fill=green!50!yellow,inner sep=.2pt,draw] at (\x) {\scalebox{0.9}{$\scriptstyle\x$}};
      \node [circle,fill=red!0!white,inner sep=.2pt,draw] at (\y) {\scalebox{0.9}{$\scriptstyle\y$}};
     };
\end{tikzpicture}}
\newcommand{\minors}[2]{\minorsize{#1}{#2}{0.45}}
\newcommand{\minorsold}[2]{\begin{tikzpicture}[scale=0.45,baseline=0cm-2.5pt]
    \foreach \x in {1,...,#1} { \coordinate (\x) at ({1.*cos(\x*360/#1)},{1.*sin(\x*360/#1)});}
    \draw[fill=yellow,draw=none] (0,0) circle(1.0);
    \draw [thick] \foreach \x/\y in {#2} {(\x)--(\y)};
    \foreach \x in {1,...,#1} {\node [circle,fill=orange!60!yellow,inner sep=.3pt] at (\x) {\scalebox{0.9}{$\scriptstyle\x$}};}
    \foreach \x/\y in {#2} {
      \node [circle,fill=green!50!yellow,inner sep=.2pt,draw] at (\x) {\scalebox{0.9}{$\scriptstyle\x$}};
      \node [circle,fill=red!0!white,inner sep=.2pt,draw] at (\y) {\scalebox{0.9}{$\scriptstyle\y$}};
     };
\end{tikzpicture}}

\newcommand{\minorat}[3]{\begin{scope}[shift={#3},scale=0.35]
    \foreach \x in {1,...,#1} { \coordinate (\x) at ({1.*cos(\x*360/#1)},{1.*sin(\x*360/#1)});}
    \draw[fill=yellow,draw=none] (0,0) circle(1.0);
    \draw [thick] \foreach \x/\y in {#2} {(\x)--(\y)};
    \foreach \x in {1,...,#1} {\node [circle,fill=orange!60!yellow,inner sep=.2pt] at (\x) {\scalebox{0.7}{$\scriptstyle\x$}};}
    \foreach \x/\y in {#2} {
      \node [circle,fill=green!50!yellow,inner sep=.2pt,draw] at (\x) {\scalebox{0.7}{$\scriptstyle\x$}};
      \node [circle,fill=red!0!white,inner sep=.2pt,draw] at (\y) {\scalebox{0.7}{$\scriptstyle\y$}};
     };
\end{scope}}
\newcommand{\minorats}[3]{\begin{scope}[shift={#3},scale=0.35]
    \foreach \x in {1,...,#1} { \coordinate (\x) at ({1.*cos(\x*360/#1)},{1.*sin(\x*360/#1)});}
    \draw[fill=yellow,draw=none] (0,0) circle(1.0);
    \draw [thick] \foreach \x/\y in {#2} {(\x)--(\y)};
    \foreach \x in {1,...,#1} {\draw[fill=orange!60!yellow,draw=none] (\x) circle(0.2);}
    \foreach \x/\y in {#2} {
      \draw [fill=green!50!yellow,draw=none] (\x) circle(0.2);
      \draw [fill=white,draw=none] (\y) circle(0.2);
     };
\end{scope}}

\def\rcs $#1: #2 ${\expandafter\def\csname rcs#1\endcsname {#2}}
\rcs $Date: 2014/11/20 03:29:10 $

\maketitle
\begin{abstract}
  We discuss several parametrizations of the space of circular planar electrical networks.
  For any circular planar network we associate a canonical minimal
  network with the same response matrix, called a ``standard'' network.
  The conductances of edges in a standard network can be computed
  as a biratio of Pfaffians constructed from the response matrix.
  The conductances serve as coordinates that are compatible with the cell structure of
  circular planar networks in the sense that one conductance
  degenerates to $0$ or $\infty$ when moving from a cell to a boundary
  cell.

  We also show how to test if a network with $n$ nodes is well-connected
  by checking that $\binom{n}{2}$ minors of
  the $n\times n$ response matrix are positive;
  Colin de Verdi\`ere had previously
  shown that it was sufficient to check the positivity of
  exponentially many minors.  For standard networks with $m$ edges, positivity of
  the conductances can be tested by checking the positivity of $m+1$ Pfaffians.
\end{abstract}

\section{Introduction}

A \textbf{circular planar network} (CPN, or simply \textbf{network} in this
paper) is a finite graph $\G=(V,E)$ embedded in the plane with a
distinguished set of vertices $\No\subset V$, called nodes, on the outer
face, and a positive real-valued function $c:E\to\R_{>0}$ on the
edges.  The value~$c(e)$ is the \textit{conductance\/} of the edge.

On a network, the Laplacian operator $\Delta:\R^V\to\R^V$ is defined by
$\Delta(f)(v)=\sum_{v'} c_{v,v'}(f(v)-f(v'))$ where the sum is over
neighbors $v'$ of $v$.

Given a network and a function $u$ on $\No$, let $f$ be the unique
harmonic extension to $V$ of~$u$, and define $L(u)=-\Delta
f\big|_{\No}$.  The linear function $L:\R^{\No}\to\R^{\No}$ is the
\textbf{response matrix} of the network.  It is a symmetric,
negative semidefinite matrix with (if $\G$ is connected) kernel
consisting of the constant functions \cite{CdV}.

Circular planar networks arise in a number of different situations: in
electrical impedance tomography \cite{Borceaetal}, in the study of the
(positive part of the) orthogonal Grassmannian \cite{LP,ALT},
in statistical mechanics \cite{KW1}, in probability \cite{kw:annular}, and
in string theory (see e.g., \cite{huang-wen-xie,kim-lee}).

Their systematic study was first begun by Colin de Verdi\`ere \cite{CdV}
and Curtis, Ingerman, Mooers, and Morrow \cite{cmm,CIM},
who studied the space $\Omega_n$
of response matrices of all networks with $n$ nodes,
proving that it is a semialgebraic set whose interior is homeomorphic to a ball of dimension $n(n-1)/2$.
They showed that $\Omega_n$ was defined by inequalities $\det L_A^B\ge0$, where $A,B\subset\No$ run over
\textbf{noninterlaced} subsets of $\No$,
i.e., $A$ and $B$ are contained in disjoint intervals in the natural circular ordering of $\No$,
and $L_A^B$ is the minor of $L$ with rows~$A$ and columns~$B$.
(We use the term \textit{minor\/} to refer to 
both a submatrix and its determinant, when there is no chance of confusion.)

For minimal, well-connected networks (see definitions below)
it was shown that the edge conductances parametrize the interior $\Omega^+_n$ of $\Omega_n$.

The boundary of $\Omega_n$ has an interesting combinatorial structure.
Lam and Pylyavskyy showed that $\Omega_n$
has the structure of a cell complex \cite{LP}.
The cells are parametrized by equivalence
classes of minimal networks; two minimal networks are equivalent if they can be obtained from one another by
Y-$\Delta$ moves. The dimension of the cell is equal to the number of edges of the minimal network.
They conjectured that the cell structure is actually a regular CW complex, that is,
the closure of each cell is homeomorphic to a closed ball, and the
attaching maps from an $m$-cell to an $(m-1)$-cell arise by sending an appropriate conductance to $0$ or~$\infty$.
Recently Lam showed that the cell complex is Eulerian \cite{lam:eulerian}.

Our main goal is to understand the correspondence between response matrices and minimal networks.
We accomplish this by parameterizing in a canonical way each cell, and using it to give an explicit reconstruction map
(from the matrix to the conductances).
There are still some open problems pertaining to the determination of which cell a response matrix belongs to.

For each cell of dimension $m$, we define a canonical minimal network
in its corresponding network equivalence class, called a
\textbf{standard network}.  This definition is closely related to the
construction \cite{huang-wen}, however we found our formulation better
suited to our needs.
We associate to a standard minimal network a set of variables, the \textbf{tripod variables}
which can be computed in terms of Pfaffians and determinants involving
the response matrix.
The conductances are biratios of tripod variables at the adjacent vertices and faces.
The \textbf{reconstruction map} from the response matrix to the set of
conductances is given as an explicit rational function in these variables.  We previously
gave this explicit rational function in the case of well-connected
networks \cite{KW2}, and a recursive reconstruction procedure was
previously given in \cite{CIM} and studied further in
\cite{card-muranaka,russell} (see also \cite{johnson}).

We associate to a (standard or nonstandard) minimal network a set
of related variables called \textbf{$\mathbf B$ variables} at each face and vertex.  Like the tripod
variables, the
conductances are biratios of the $B$ variables at the adjacent
vertices and faces.  The $B$ variables transform under Y-$\Delta$
moves via the \textbf{cube recurrence} \cite{CS,GK}. 

The tripod variables and $B$ variables parameterize the same spaces but have different advantages:
$B$ variables can be defined for any network, but are harder to compute in terms of the response matrix;
tripod variables can be directly computed from the response matrix but are defined only for standard networks.

Another way to parametrize well-connected $n$-node networks is through a
collection of $\binom{n}{2}$ \textbf{central minors} of the $L$
matrix.  Positivity of the central minors implies positivity of all
noninterlaced minors.  This result is analogous to testing whether an
$n\times n$ matrix is totally positive by testing only $n^2$ minors
(see \cite{MR1745560}), but here there are $n$ nodes each of which can
index either a row or a column, while for the total positivity tests,
there are $n$ nodes that index rows and another $n$ nodes that index
columns.

Similar parametrizations (using minors of the $L$ matrix) seem to hold
for general minimal networks, but it is an open problem to find such
parametrizations in general; see \cite{ALT} for some work in this
direction.
\medskip

\noindent{\bf Acknowledgements.} We thank the referees for helpful comments and a careful review of the paper.

\section{Background}

Background in this section comes from \cite{CdV, curtis-morrow, kenyon:surfaces, lam-williams}.

\subsection{Dual network}

Given a circular planar network $\G$ with $n$ nodes, the \textbf{dual network}
is the network $\G^*$ with $n$ nodes, with a node between every two
adjacent nodes of $\G$, constructed from the dual graph of $\G$
embedded in the disk.  (That is, the vertices of $\G^*$ are the
regions of the disk which are bounded by edges of $\G$ or the boundary
of the disk, see Figure~\ref{strands-network}.)  The conductances of
an edge and its dual edge are reciprocals.

\subsection{Equivalence}

Two networks are \textbf{topologically equivalent} if they have the same number of nodes and one
can be obtained from the other by \textbf{electrical transformations}, see Figure~\ref{ETs}, disregarding
conductances.

\begin{figure}[htbp]
\includegraphics[width=\textwidth]{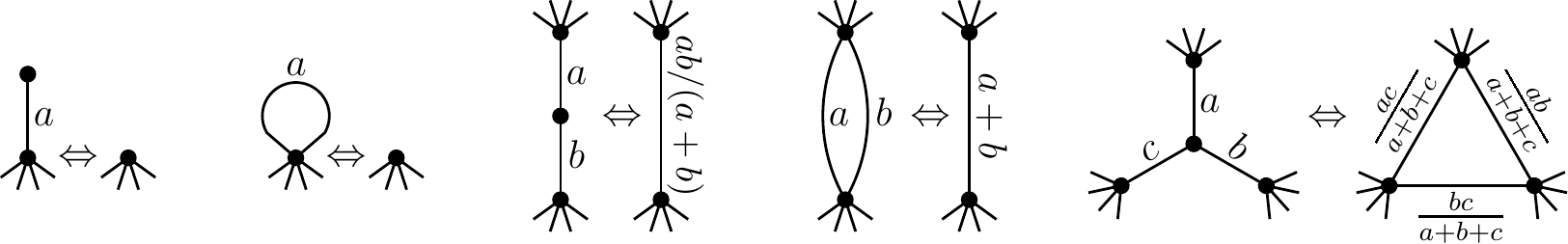}
\caption{\label{ETs}The electrical transformations:
removing a dead branch (a degree-$1$ non-node vertex),
removing a self-loop,
combining edges in series (when the central vertex is not a node),
combining edges in parallel,
and a Y-$\Delta$ transformation (when the central vertex is not a node).
This figure first appeared in \cite{KW2}.
}
\end{figure}

A network is \textbf{minimal} (also called \textbf{reduced}) if it has
the smallest number of edges in its topological equivalence class.
Two minimal equivalent networks can be obtained from one another using
only Y-$\Delta$ moves.

Two networks are \textbf{electrically equivalent\/} if they have the
same response matrix.  In \cite{CGV} it is shown that networks are
electrically equivalent if and only if they can be obtained from one
another by electrical transformations.

\subsection{Medial graph and strand matching}

The \textbf{medial graph} of a network is a degree-$4$ graph with a
vertex for every edge of $\G$, and an edge connecting two vertices if
the corresponding edges of $\G$ are consecutive around a face of $\G$
(see Figure~\ref{strands-network}).  For each node of $\G$ it is
customary to break the edge of the medial graph separating it from
$\infty$ into two half-edges, called \textbf{stubs}. In this way the
medial graph consists of $n$ \textbf{strands}, which are paths in the
medial graph which go straight (neither turning left nor right)
through each vertex; strands begin and end at stubs.

A circular planar network is minimal if there are no closed strands,
strands do not cross themselves, and two strands cross at most once
\cite{CGV}.

To each node $i$ of a minimal circular planar network there are two
stubs of the medial graph, one just to the left of $i$ (in the
circular order) and one just to the right of $i$.  We label these two
stubs $2i-1$ and $2i$.  There is a fixed-point free involution of the
stubs defined by following the strands from one end to the other.
This involution is called the \textbf{strand matching} $\pi=\pi(\G)$
of the network.

Note that Y-$\Delta$ moves do not change the strand matching~$\pi$.
Thus $\pi$ is a function only of the topological equivalence class of the
(minimal) network.

\begin{figure}[t]
\begin{center}
\includegraphics[width=3.5in]{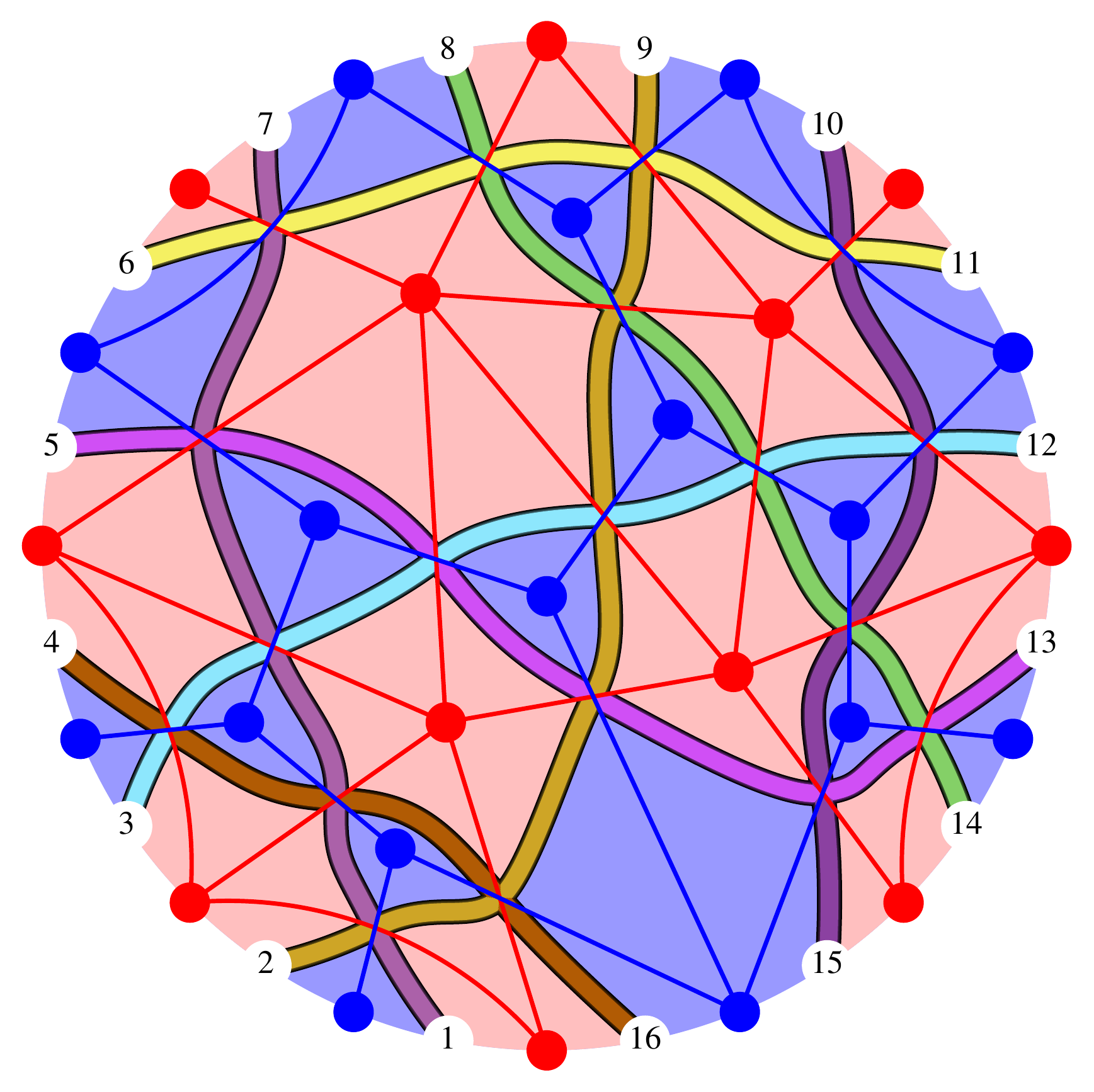}
\end{center}
\caption{ A strand diagram with 16 stubs.  The strands divide the disk
  into cells, which are alternately colored blue or red.  The blue
  cells are the vertices of the blue network, which has 8 nodes, and
  similarly the red cells define the red network, which is dual
  (within the disk) 
  to the blue network.  The vertices of the medial
  graph are where the blue and red network edges intersect, and the
  edges of the medial graph can be drawn where the strands are.
\label{strands-network}}
\end{figure}

Every fixed-point free involution on $\{1,\dots,2n\}$ is the strand matching of some
electrical network with $n$ nodes.  This can be seen by taking $2n$
points in generic position on the circle; join them in pairs using
chords according to the strand matching~$\pi$. The chords form the
medial graph of a network on $n$ nodes with strand matching~$\pi$.
For some strand matchings, some of the boundary nodes of the network
will be glued together or in different components --- the associated networks are called
cactus networks in \cite{lam}.

\subsection{Groves and partitions}
Given a network, a \textbf{grove}
is a set of edges with the property that it contains no cycles and every
component contains at least one node.
A grove is similar in concept to an \textit{essential spanning forest\/} on an infinite graph,
which is a set of edges containing no cycles for which every connected component is infinite,
in the sense that every tree reaches the ``boundary''.

Each grove partitions the nodes according to its connected components.
Associated to a partition $\tau$ of the nodes is its partition sum
\[Z_\tau=\sum_{T\in\tau} wt(T),\] where the sum is over groves having partition $\tau$,
and $wt(T)$ is the product of the conductances of edges in $T$.

In \cite{KW1} we showed how to compute $Z_\tau$ (appropriately
normalized) for any circular planar network, for any partition $\tau$,
as a polynomial of entries in $L$.  In \cite{kw:annular} we showed
that $Z_\tau$ (appropriately normalized) is a linear combination of
minors of $L$.

\subsection{Tripod and dual-tripod partitions}
\label{sec:tripod-pf}

A partition which contains a part of size $3$, with the 
remaining parts having size $2$ and
ordered in parallel in the three regions complementary to the
triple part is called a \textbf{tripod partition}.
We also allow the outermost part in any of the three regions to
be a singleton part rather than a doubleton part.  We also
consider partitions consisting of parallel doubleton parts,
possibly also with outermost singleton parts, to be (degenerate)
tripod partitions.  The bottom row of Figure~\ref{tripods} illustrates
a variety of tripod partitions.  The dual partitions are illustrated
in the top row, and are called \textbf{dual-tripod partitions}.
The boundary of the disk can partitioned into red, green, and blue
segments, such that such the parts of $\tau$ connect nodes of
different color.  Singleton nodes are split between these segments,
and can be regarded as having two colors.

\begin{figure}[t!]
\center{\includegraphics[width=\textwidth]{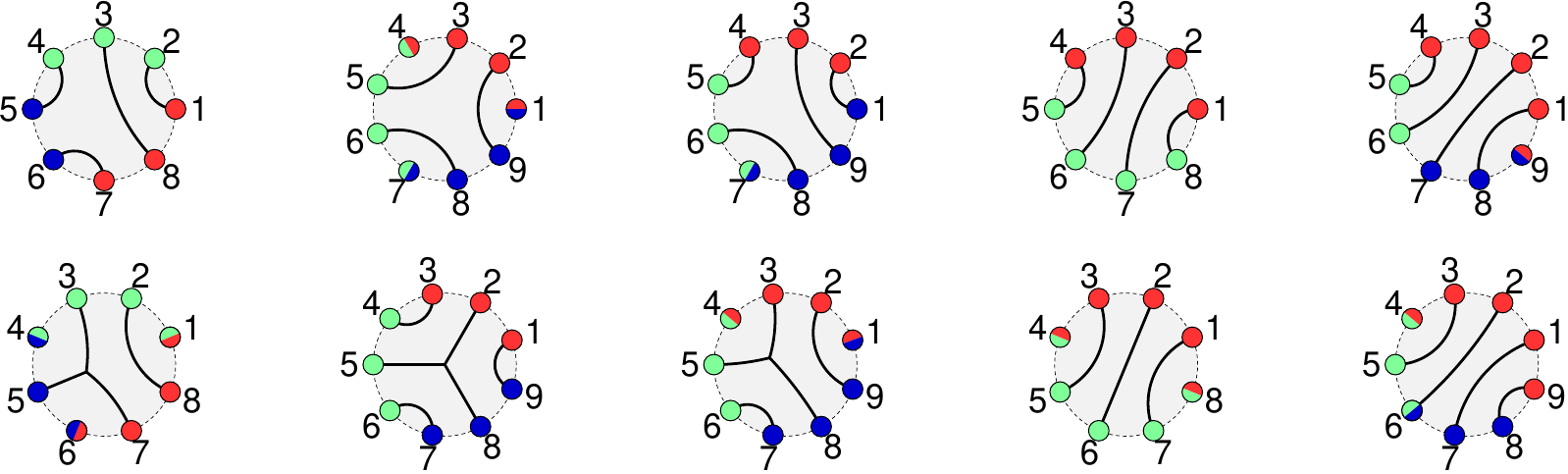}}
\caption{
Tripod partitions (bottom row) and dual-tripod partitions (top row).
There is another characterization of tripod and dual-tripod partitions, where the boundary is divided into three segments which are are colored red, green, and blue, and the parts of the partition connect nodes of different colors.
This figure first appeared in \cite{KW2}.
}
\label{tripods}
\end{figure}

In earlier work we showed how to compute grove partition
functions for dual-tripod and tripod partitions in terms of a Pfaffian of a
matrix constructed from the response matrix \cite{KW2}.
Recall that the Pfaffian (denoted $\Pf$) is defined for skew-symmetric matrices,
and its square is the determinant of the matrix.
We will show in Section~\ref{sec:tripod-variable} how to reconstruct
the edge conductances of a standard minimal network using the grove
partition functions of dual-tripod and tripod partitions.

\newcommand{\fm}{\phantom{-}}
\setcounter{MaxMatrixCols}{20}

The Pfaffian formula for dual-tripod partitions is best explained by an example (which we borrow from \cite{KW2}):
\[
\raisebox{-0.4\height}{$\displaystyle\frac{Z\left(\raisebox{-0.45\height}{\includegraphics[scale=0.85]{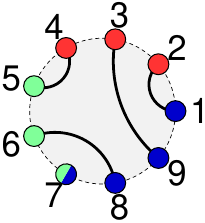}}\right)}{Z(uncrossing)}$} =
\Pf\footnotesize\begin{bmatrix}
 \red{0} & \red{0} & \red{0} &\fm L_{\red{2},\green{5}} &\fm L_{\red{2},\green{6}} &\fm L_{\red{2},\blue{8}} &\fm L_{\red{2},\blue{9}} &\fm L_{\red{2},\blue{1}} \\
 \red{0} & \red{0} & \red{0} &\fm L_{\red{3},\green{5}} &\fm L_{\red{3},\green{6}} &\fm L_{\red{3},\blue{8}} &\fm L_{\red{3},\blue{9}} &\fm L_{\red{3},\blue{1}} \\
 \red{0} & \red{0} & \red{0} &\fm L_{\red{4},\green{5}} &\fm L_{\red{4},\green{6}} &\fm L_{\red{4},\blue{8}} &\fm L_{\red{4},\blue{9}} &\fm L_{\red{4},\blue{1}} \\
 -L_{\green{5},\red{2}} & -L_{\green{5},\red{3}} & -L_{\green{5},\red{4}} & \green{0} & \green{0} &\fm L_{\green{5},\blue{8}} &\fm L_{\green{5},\blue{9}} &\fm L_{\green{5},\blue{1}} \\
 -L_{\green{6},\red{2}} & -L_{\green{6},\red{3}} & -L_{\green{6},\red{4}} & \green{0} & \green{0} &\fm L_{\green{6},\blue{8}} &\fm L_{\green{6},\blue{9}} &\fm L_{\green{6},\blue{1}} \\
 -L_{\blue{8},\red{2}} & -L_{\blue{8},\red{3}} & -L_{\blue{8},\red{4}} & -L_{\blue{8},\green{5}} & -L_{\blue{8},\green{6}} & \blue{0} & \blue{0} & \blue{0} \\
 -L_{\blue{9},\red{2}} & -L_{\blue{9},\red{3}} & -L_{\blue{9},\red{4}} & -L_{\blue{9},\green{5}} & -L_{\blue{9},\green{6}} & \blue{0} & \blue{0} & \blue{0} \\
 -L_{\blue{1},\red{2}} & -L_{\blue{1},\red{3}} & -L_{\blue{1},\red{4}} & -L_{\blue{1},\green{5}} & -L_{\blue{1},\green{6}} & \blue{0} & \blue{0} & \blue{0} \\
\end{bmatrix}
\]
Here $Z(uncrossing)$ denotes the partition function for groves in
which every node is in a separate tree.  The rows and columns are
indexed by the nodes (in cyclic order) which are red, green, and blue.
Singleton (bichromatic) nodes are omitted.  Above the diagonal, each
matrix entry is the response matrix, unless the row and column are of
the same color, in which case the matrix entry is $0$.

Since the Pfaffian is only defined for skew symmetric matrices, we
sometimes do not write the portion below the diagonal.

Tripod partitions have a similar but slightly different Pfaffian
formula, illustrated below.  In \cite{KW2} we gave a triple-sum of
Pfaffians formula for tripod partitions; this single-Pfaffian
formula for tripods is new, but follows from other formulas in \cite{KW2}.
\[
\raisebox{-0.4\height}{$\displaystyle\frac{Z\!\left(\!\raisebox{-0.45\height}{\includegraphics[scale=0.85]{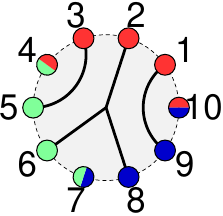}}\!\right)\!}{Z(uncrossing)}$} =
\Pf\footnotesize\begin{bmatrix}
 \red{0} & \red{0} & \red{0} & L_{\red{1},\green{5}} & L_{\red{1},\green{6}} & -L_{\red{1},\blue{8}} & -L_{\red{1},\blue{9}} & L_{\red{1},\green{5}}+L_{\red{1},\green{6}}+L_{\red{1},7} \\
 & \red{0} & \red{0} & L_{\red{2},\green{5}} & L_{\red{2},\green{6}} & -L_{\red{2},\blue{8}} & -L_{\red{2},\blue{9}} & L_{\red{2},\green{5}}+L_{\red{2},\green{6}}+L_{\red{2},7} \\
 & & \red{0} & L_{\red{3},\green{5}} & L_{\red{3},\green{6}} & -L_{\red{3},\blue{8}} & -L_{\red{3},\blue{9}} & L_{\red{3},\green{5}}+L_{\red{3},\green{6}}+L_{\red{3},7} \\
 & & & \green{0} & \green{0} & \fm L_{\green{5},\blue{8}} & \fm L_{\green{5},\blue{9}} & L_{\green{5},\blue{8}}+L_{\green{5},\blue{9}}+L_{\green{5},10} \\
 & & & & \green{0} & \fm L_{\green{6},\blue{8}} & \fm L_{\green{6},\blue{9}} & L_{\green{6},\blue{8}}+L_{\green{6},\blue{9}}+L_{\green{6},10} \\
 & & & & & \blue{0} & \blue{0} & L_{\blue{8},\red{1}}+L_{\blue{8},\red{2}}+L_{\blue{8},\red{3}}+L_{\blue{8},4} \\
 & & & & & & \blue{0} & L_{\blue{9},\red{1}}+L_{\blue{9},\red{2}}+L_{\blue{9},\red{3}}+L_{\blue{9},4} \\
 & & & & & & & 0\\
\end{bmatrix}
\]
Except for the last row and column, the matrix is constructed in the
same manner as for the dual-tripod matrix, except that entries whose
row is red and column is blue have a minus sign.  The last column is
obtained by summing the positive entries in the row, and adding one
more term corresponding to the singleton node (if any) which has no colors in
common with the row node.

\begin{proposition}
  For a tripod partition $\tau$, $Z_\tau/Z_{\text{uncrossing}}$ is given by the Pfaffian of the above matrix.
\end{proposition}
\begin{proof}
In \cite[Appendix~B]{KW2} we gave a formula for dual tripod partition functions (without singleton parts),
normalized by the weighted sum of trees rather than $Z(\text{uncrossing})$,
using the Pfaffian of a matrix constructed
from the pairwise resistances.  This formula takes the form
\newcommand{\plt}{\phantom{+}t}
\[
\raisebox{-0.4\height}{$\displaystyle\frac{Z\!\left(\!\raisebox{-0.45\height}{\includegraphics[scale=0.85]{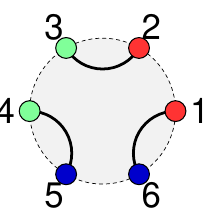}}\!\right)\!}{Z(tree)}$} =
[t]\,\Pf\footnotesize\begin{bmatrix}
\red{0}         &  \red{0}         &\plt-\frac12 R_{\red{1},\green{3}} &\plt-\frac12 R_{\red{1},\green{4}} &\plt-\frac12 R_{\red{1},\blue{5}} &\plt-\frac12 R_{\red{1},\blue{6}} \\[2pt]
&  \red{0}         &\plt-\frac12 R_{\red{2},\green{3}} &\plt-\frac12 R_{\red{2},\green{4}} &\plt-\frac12 R_{\red{2},\blue{5}} &\plt-\frac12 R_{\red{2},\blue{6}} \\[2pt]
&&  \green{0}         &  \green{0}         &\plt-\frac12 R_{\green{3},\blue{5}} &\plt-\frac12 R_{\green{3},\blue{6}} \\[2pt]
&&&  \green{0}         &\plt-\frac12 R_{\green{4},\blue{5}} &\plt-\frac12 R_{\green{4},\blue{6}} \\[2pt]
&&&&  \blue{0}         &  \blue{0}         \\
&&&&&  \blue{0}
\end{bmatrix}
\]
Here the Pfaffian is a polynomial in $t$, and we take the linear term.

Suppose that $\tau$ is a tripod partition with three singleton (bichromatic) nodes, so that the dual partition $\tau^*$ is a pairing of the dual nodes.
Since $Z_\tau/Z_{\text{uncrossing}} = Z^*_{\tau^*}/Z^*_{\text{tree}}$, where $Z^*_{\text{tree}}$ is the
weighted sum of trees in the dual network, we apply the above formula on the dual network.

Suppose that in the dual there are $r$ red nodes, $g$ green nodes, and $b$ blue nodes, so that in the primal there are $r-1$ fully red nodes, $g-1$ fully red nodes, and $b-1$ fully blue nodes, with the red/green node having index $r$, the green/blue node having index $r+g$, and the blue/red node having index $r+g+b$.  Let $n=r+g+b$.  We index the primal and dual nodes so that dual node $i$ is between primal nodes $i$ and $i+1\bmod n$.

We do a sequence of row and column operations on the above matrix which do not affect the Pfaffian.  We subtract each row from the previous row (and each column from the previous column).
Except along the borders of the blocks, and except for the last row and column, the new matrix entry at $i,j$ is either $0$ (if $i$ and $j$ have the same color), or else
$-\frac{1}{2}R^*_{i,j}+\frac{1}{2}R^*_{i+1,j}+\frac{1}{2}R^*_{i,j+1}-\frac{1}{2}R^*_{i+1,j+1}$.  But the dual resistance $R^*_{i,j}$ can be expressed in terms of the primal response matrix:
\[R^*_{i,j} = \sum_{\substack{i',j'\text{ such that}\\i'\leq i<j'\leq j\text{ cyclically}}} L_{i',j'},\]
so these non-border matrix entries simplify to $-L_{i,j}$.
After doing these row/column operations, the matrix becomes
\[\footnotesize
    \!\!\!\begin{blockarray}{ccccccc}
        & j<r\! & j=r& r<j<r{+}g & \!j=r{+}g\! & \cdots %r{+}g<j<r{+}g{+}b
 & j=r{+}g{+}b \\
      \begin{block}{r[cccccc]}\\[-4pt]
        i<r\!\!\!\!\!\! & 0 & \frac{1}{2}R^*_{i,r{+}1}{-}\frac{1}{2}R^*_{i{+}1,r{+}1} & -L_{i,j} & * & -L_{i,j} & \frac{1}{2}R^*_{i{+}1,n}{-}\frac{1}{2}R^*_{i,n} \\
        i=r\!\!\!\!\!\! &  & 0 & \frac{1}{2}R^*_{r,j{+}1}{-}\frac{1}{2}R^*_{r,j} & t-* & -L_{i,j} & \frac{1}{2}R^*_{i{+}1,n}{-}\frac{1}{2}R^*_{i,n} \\[-3pt]
     \vdots\quad %   r<i<r{+}g
 &  & & 0 & * & -L_{i,j} & \frac{1}{2}R^*_{i{+}1,n}{-}\frac{1}{2}R^*_{i,n} \\
%        i=
r{+}g\!\!\!\!\!\! &  &  & & 0 & * & t-*\\[-3pt]
        \vdots\quad %r{+}g<i<r{+}g{+}b\!\! 
&  & & & & 0 & 0 \\
%        i=
r{+}g{+}b\!\!\!\!\!\! &  & & & & & 0 \\[4pt]
      \end{block}
    \end{blockarray}\!
\]
Notice that the only matrix entries containing $t$ have row/column $r+g$ and either row/column $r$ or row/column $r+g+b$.  Since we are extracting the linear
term, we can express this as a sum of two Pfaffians, the first with row/columns $r+g$ and $r$ removed (and sign $(-1)^{g-1}$), and the other with row/columns $r+g$ and $r+g+b$ removed (and sign $(-1)^{b-1}$).
For the second Pfaffian, we move row/column $r$ to the last position (which introduces a sign of $(-1)^{g+b}$), and combine it with the first Pfaffian by simply adding the last row/columns of both matrices.  
The last column becomes
\[
\footnotesize
    \begin{blockarray}{ccccccc} \\
      \begin{block}{r[cccccc]}\\[-4pt]
        i<r &  L_{i,r{+}1}+\cdots+L_{i,n{-1}}\\
     r<i<r+g &  -L_{i,1}-\cdots-L_{i,r-1}-L_{i,n}\\
r+g<i<r+g+b & L_{i,r} \\
\text{last row} & 0 \\[4pt]
      \end{block}
    \end{blockarray}
\times (-1)^{g-1}
\]
Then we do additional row/column operations to absorb the factor of $(-1)^{g-1}$ and make the matrix take a more symmetric form with respect to the three color classes.

If the tripod partition $\tau$ does not contain three singleton nodes,
we may reduce this to the previous case by simply adjoining such nodes
to the network with no edges connecting them to other vertices.  Both
$Z_\tau$ and $Z_{\text{uncrossing}}$ remain unchanged, and the
response matrix entries for the new nodes are zero, so the same
formula holds.
\end{proof}

\section{Standard networks} \label{sec:standard}
The \textbf{standard well-connected network} on $n$ nodes (defined in \cite{CGV}) is the
network illustrated in Figure~\ref{stdnet}. We construct a network of
similar form for each strand matching~$\pi$; we call these
\textbf{standard networks}.

\begin{figure}[htbp]
\center{\includegraphics[height=1.1in]{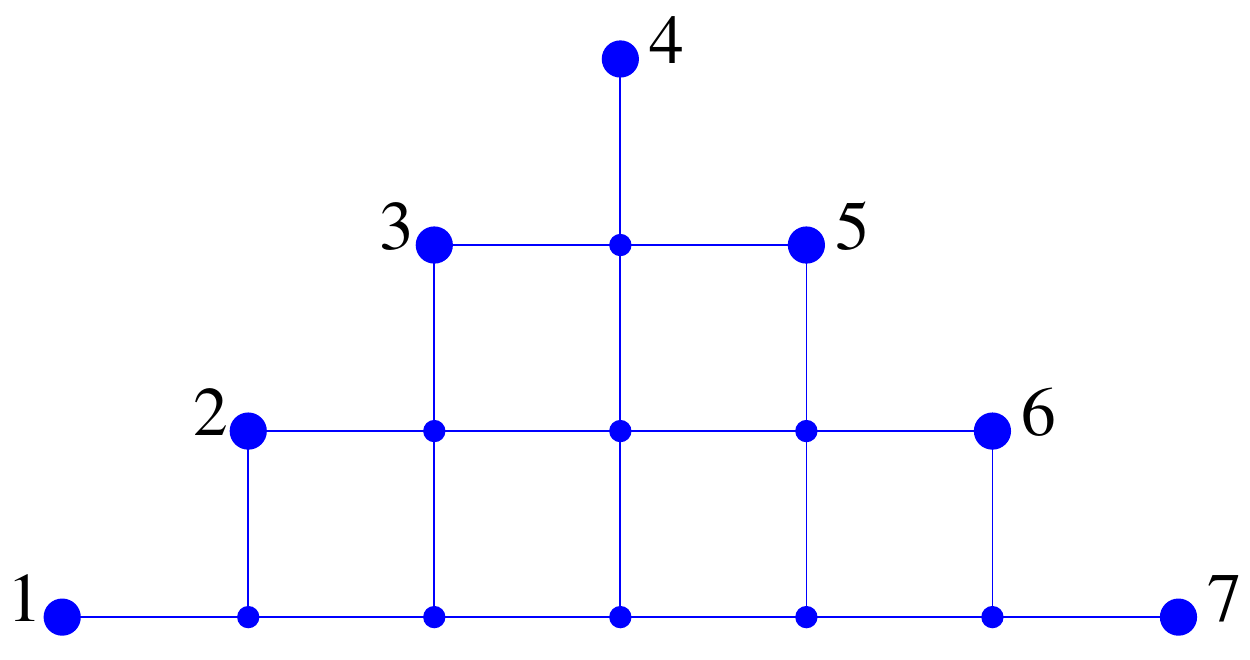}\hfill\includegraphics[height=1.1in]{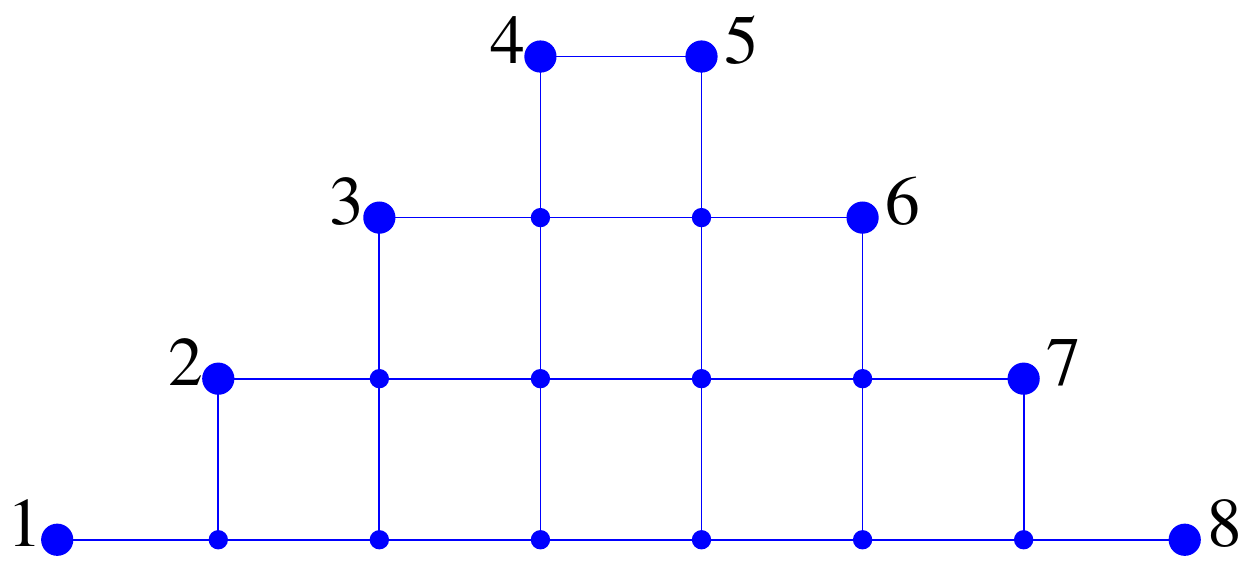}}
\caption{\label{stdnet}The standard well-connected networks on $7$ and $8$ nodes.}
\end{figure}

A \textbf{Dyck path} of order $n$
is the graph of a simple random walk on $\Z$ which starts at the origin, remains nonnegative,
and returns to the origin after $2n$ steps.
Given two Dyck paths $\lambda$ and $\mu$ of order $n$,
the domination partial order is defined by $\lambda\preceq\mu$ if at each horizontal position,
$\mu$ is at least as high as $\lambda$.  If $\lambda\preceq\mu$, then the region between them
is denoted $\lambda/\mu$ and is called a skew Young diagram, or simply a skew shape.
A \textbf{Dyck tiling} (see \cite{KW3,MR2927185}, where these were
called \textbf{cover-inclusive} Dyck tilings) is a tiling
of $\lambda/\mu$ with \textit{Dyck tiles\/} which are fattened Dyck paths, satisfying the following
``cover-inclusive'' constraint:
if two tiles are vertically adjacent (in the sense that some square of one lies in the same column
as a
square of the other, and one position above it), then the horizontal extent of the upper tile
is a subset of the horizontal extent of the lower tile.  See Figure~\ref{DyckTiling} for an
example of a Dyck tiling.
Dyck tilings have also been studied in \cite{kim,KMPW,kw:annular,fayers,fisher-nadeau}.
\begin{figure}[htbp]
\center{\includegraphics[width=5in]{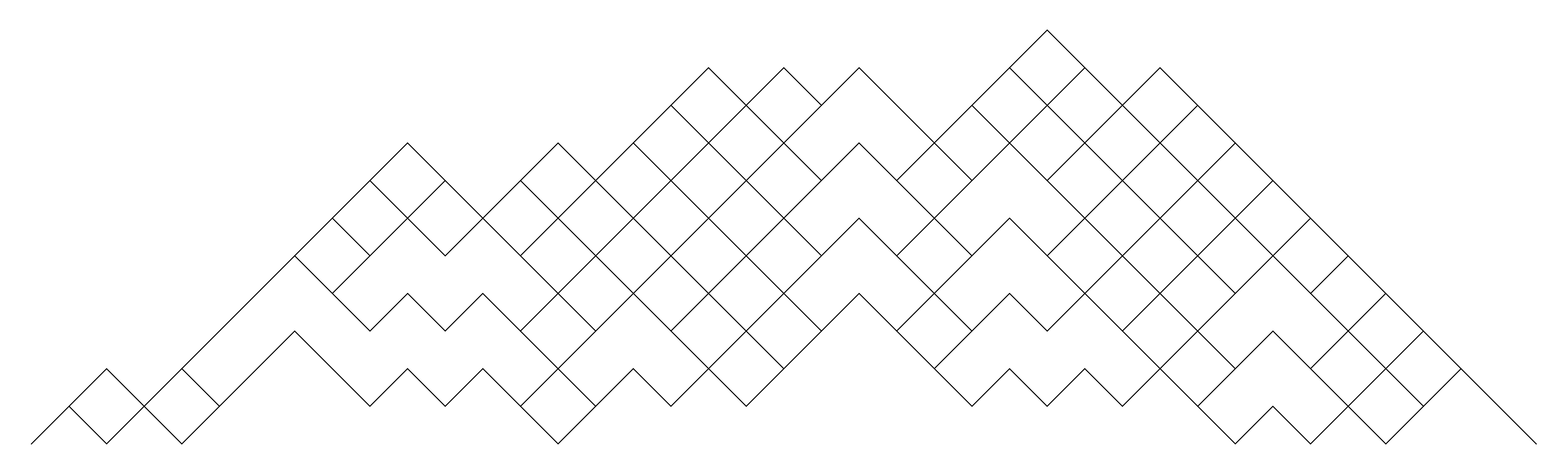}}
\caption{\label{DyckTiling}A Dyck tiling.}
\end{figure}

There is a bijection between Dyck tilings and perfect matchings of $\{1,\dots,2n\}$
\cite[Fig.~18]{kim} and \cite[Fig.~7]{KMPW}.  Here we need a slightly different bijection, which we now describe.
To the Dyck tiling we associate a strand diagram as follows (see Figure~\ref{strands}).
Along the upper Dyck path~$\mu$ of the tiling,
there is a strand starting or ending in the center of every edge.
Each Dyck tile contains two medial strands that cross once: the strand entering
at the lower edge adjacent to the left-most point of the tile and exiting at the upper edge adjacent to
the right-most point of the tile,
and the strand entering
at the upper edge adjacent to the left-most point of the tile and exiting at the lower edge adjacent to
the right-most point of the tile.
A tile may contain additional strands that pass through it horizontally, without crossing any
other strands within the tile.
From this strand diagram we can build the electrical network~$\G$ and its dual network~$\G^*$,
as illustrated in Figure~\ref{strands}.

In the special case where the lower path $\lambda$ is minimal (the zigzag path) and
the upper path $\mu$ is maximal, and every Dyck tile is a single box, then every
pair of strands cross, and the resulting network is the standard well-connected network.
The networks we obtain from other Dyck tilings are analogous to the standard networks
in the well-connected case, so we call them standard networks.

\begin{figure}[htbp]
\begin{center}
\raisebox{10pt}{(a)}\includegraphics[width=.71\textwidth]{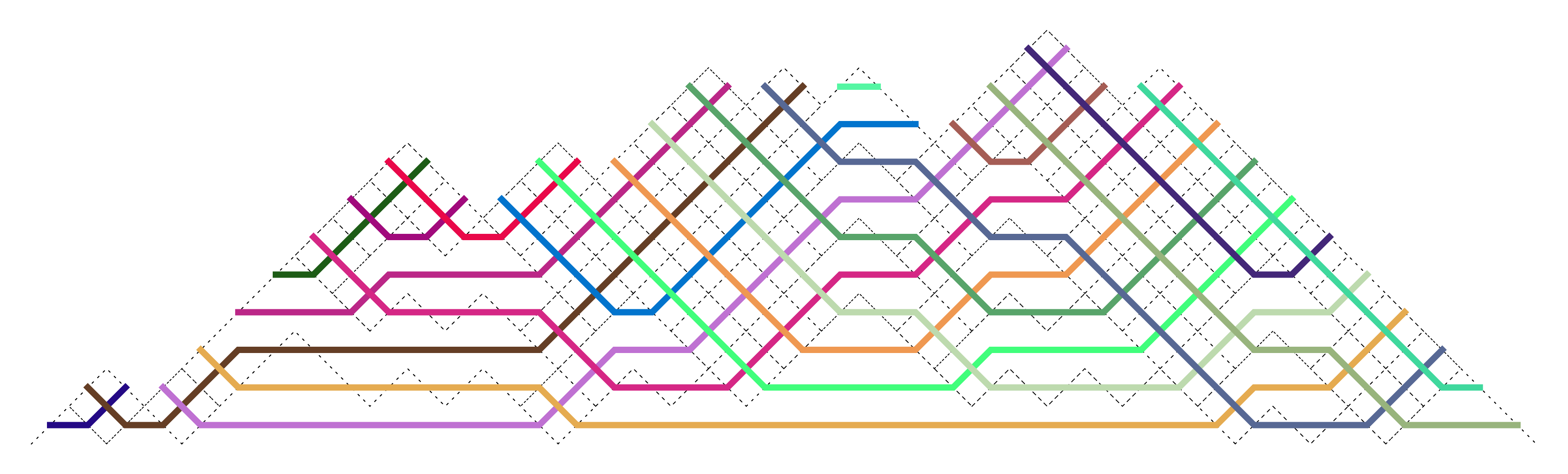}\\[-5pt]
\raisebox{10pt}{(b)}\includegraphics[width=.71\textwidth]{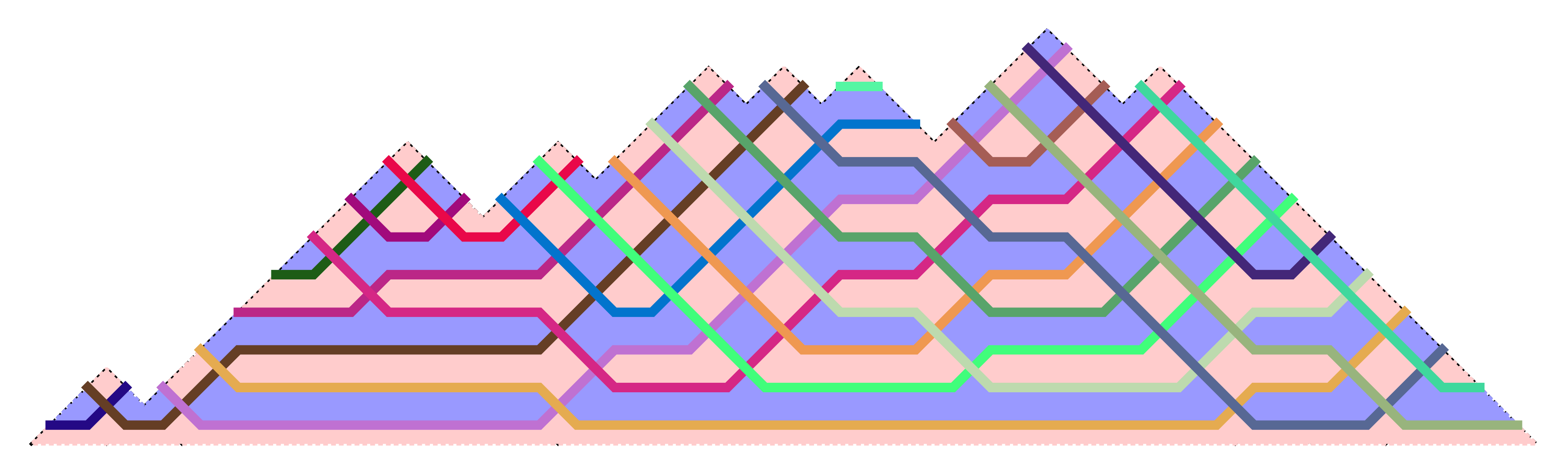}\\[5pt]
\raisebox{10pt}{(c)}\includegraphics[width=.71\textwidth]{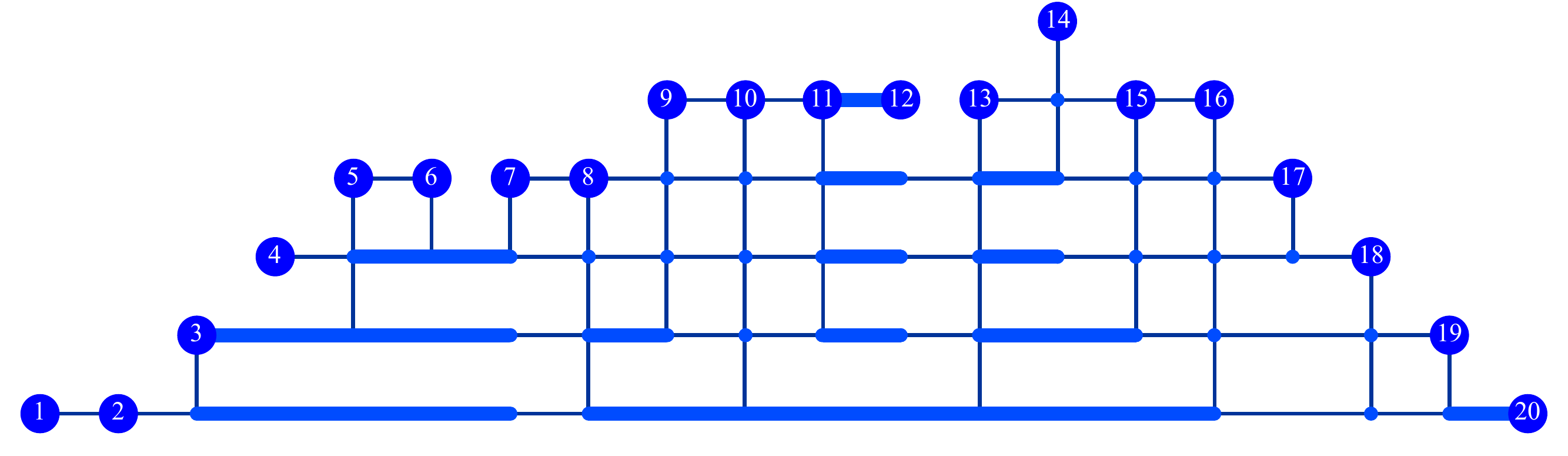}\\[5pt]
\raisebox{10pt}{(d)}\includegraphics[width=.71\textwidth]{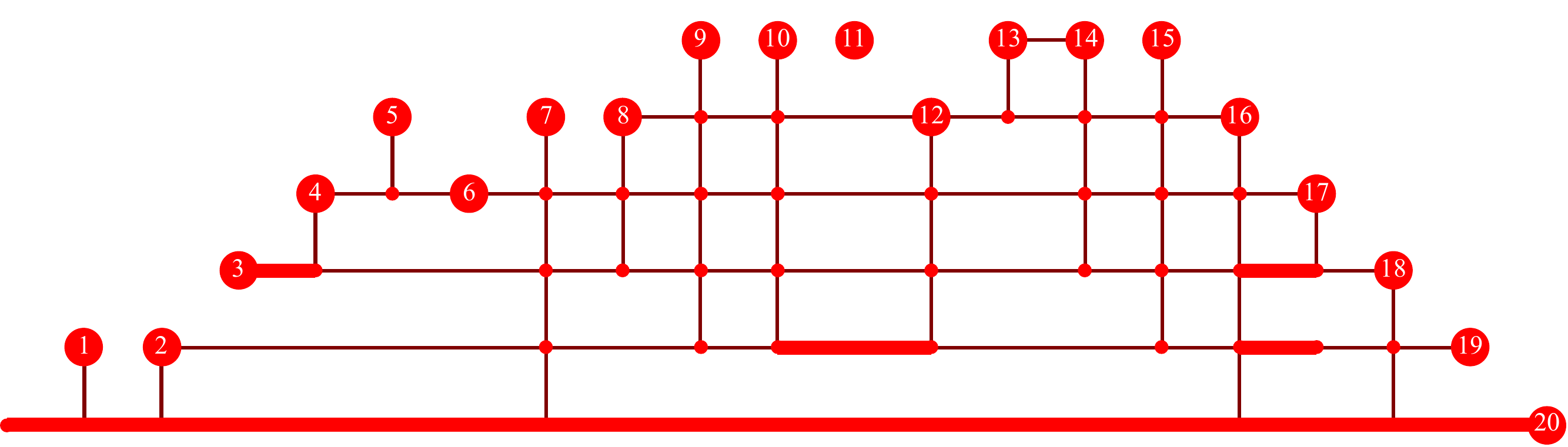}
\end{center}
\caption{(a) The strand diagram associated to a Dyck tiling of Figure \protect{\ref{DyckTiling}},
 (b) the strands in the region below the upper Dyck path (with tiles erased) together with the cells they bound colored red and blue, (c) the blue network, formed from the blue cells, (d) the red network formed from the red cells.  Some of the vertices in the networks are drawn in an extended fashion.  The blue and red networks are dual to one another.  In this example, red node 11 is disconnected, and blue nodes 11 and 12 are glued together.
\label{strands}}
\end{figure}

For every crossing in the strand diagram, there is either a horizontal
edge of $\G$ and a vertical dual edge of~$\G^*$, or else a vertical edge
of $\G$ and a horizontal dual edge of~$\G^*$.  By the \textit{horizontal
  conductance\/} of an edge (or crossing), we mean the conductance of
either the edge or its dual, whichever one is horizontal, and
similarly the \textit{vertical conductance\/} is the conductance of
whichever one is vertical (and is the reciprocal of the horizontal
conductance).

We now describe how to construct a Dyck tiling from a strand matching,
using an inductive procedure, which is
illustrated in Figure~\ref{strandtotiling}.  The base case is $n=0$,
in which the trivial strand matching, containing no strands or stubs,
corresponds to the trivial Dyck tiling, which has no tiles, and whose
upper and lower Dyck paths have length $2n=0$.  We can build up any
strand matching starting from the trivial matching by a sequence of
two types of elementary moves, and while doing this, we build up the
Dyck tiling and its associated rectilinear strand diagram starting
from the trivial tiling.
\begin{figure}[htbp]
\begin{center}
\includegraphics[width=3in]{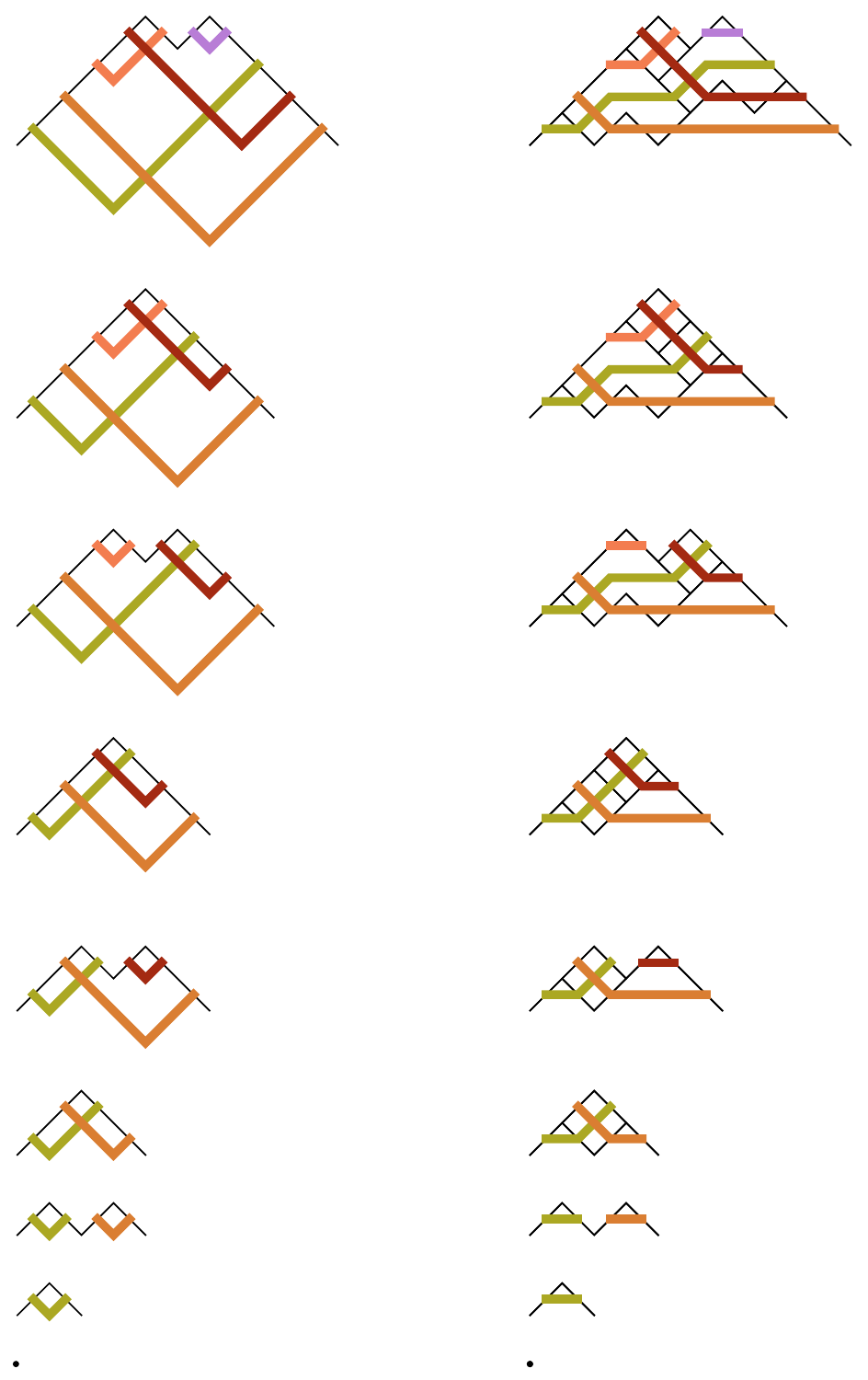}
\end{center}
\caption{
Going from a strand matching to a Dyck tiling with associated rectilinear strand diagram.
This construction is a modification of the recursive procedure illustrated in \cite[Fig.~7]{KMPW} (see also \cite[Fig.~18]{kim}).  We start in the upper left with the matching represented as a Dyck path and V-shaped strands (these are up to reflection and rotation the ``folded tableaux'' in \cite[Fig.~12]{huang-wen}).  This matching is deconstructed going down on the left by removing peaks in the Dyck path; each such peak corresponds to either matched neighbors $(i,i+1)$, or crossing neighbors $(a,i+1)$, $(i,b)$ where $a<i<i+1<b$.  For each deconstruction move going down on the left there is a corresponding construction move going up on the right: going up, if two neighbors are made to cross, then a box is added, and if a matched pair of neighbors is added, then the tiles overlapping that column are enlarged.
\label{strandtotiling}}
\end{figure}
The two moves to build up a strand matching are (1) introducing into the matching a new strand with adjacent stubs,
either at the beginning or end or between existing stubs,
and (2) crossing a pair of adjacent uncrossed strands, i.e., replacing $(a,i)$ and $(i+1,b)$ (where $a<i<i+1<b$)
with $(a,i+1)$ and $(i,b)$.
The operation of crossing a pair of adjacent endpoints corresponds to adding a square Dyck tile.
The operation of inserting a pair of paired endpoints corresponds to splitting a Dyck tiling
along a given column, in which case any tile crossing that column gets split into two halves which
are joined in a unique way (inserting a ``tent'') to make a larger tile.

\subsection{Tripod variables}
\label{sec:tripod-variable}

Given a standard minimal network, we define a collection of
tripod variables as illustrated in
Figure~\ref{dyck-tiling-tripod}.  Consider the strand diagram coming
from the Dyck tiling associated with the standard network.  For any
crossing $\chi$ in this strand diagram, the upward-left strand from
$\chi$ together with all the upward right strands from this strand
form a comb shape (Fig.~\ref{comb}).  If we resolve
each crossing in the comb vertically, and each crossing not in the
comb horizontally, as in Fig.~\ref{comb-tripod}, this
defines a partition $\tau_\chi$ of the nodes of the network (and a
dual partition $\tau^*_\chi$ of the dual nodes of the dual network).
We consider one further pair of partitions, the \textit{exterior
  partition\/}~$\tau_-$ and its dual $\tau^*_-$, which consist in
taking all the horizontal edges.

Let $\{Z_\tau\}$ be the collection of partition sums $Z_\tau$ as
$\tau$ ranges over the exterior partition~$\tau_-$ together with all
the partitions $\tau_\chi$ associated with crossings in the strand
diagram.
We refer to $\{Z_\tau\}$ as the \textbf{tripod variables} of the
network.  We similarly let $\{Z^*_{\tau^*}\}$ be the corresponding
tripod variables in the dual network.

\enlargethispage{12pt}
The partitions $\tau_\chi$ and $\tau_-$ are not themselves tripod or
dual-tripod partitions in general, but they are are closely related to
such partitions in an enlarged electrical network, justifying the name ``tripod variable''.
The enlarged electrical network
$\tilde\G_\chi$, shown in Fig.~\ref{comb-extended}, is formed by
extending the strands in the strand diagram, with a ``notch'' where
the uppermost tooth of the comb meets the spine of the comb.
(We do not define an enlarged network $\tilde\G_-$.)
In the enlarged network~$\tilde\G_\chi$, the comb at $\chi$ defines a
partition $\tilde\tau_\chi$ which \textit{is\/} a tripod or
dual-tripod partition (Fig.~\ref{comb-tripod-extended}).  The
all-horizontal partition $\tilde\tau_-$ in $\tilde\G_\chi$ is also a
tripod or dual-tripod partition.

\begin{lemma}\label{tau-unique}
  Let $\G$ be a standard minimal network, let $\chi$ be a crossing of
  $\G$'s strand diagram, and let $\tilde\G_\chi$ be the extended
  network of $\G$.  Within $\tilde\G_\chi$, there is a unique grove
  whose partition is $\tilde\tau_\chi$, and a unique grove whose
  partition is $\tilde\tau_-$.  Within $\G$, there is a unique grove
  whose partition is $\tau_\chi$, and a unique grove whose partition
  is $\tau_-$, which uniquely extend to the groves of type
  $\tilde\tau_\chi$ and $\tilde\tau_-$ within $\tilde\G_\chi$.
\end{lemma}
\newpage
\begin{proof}[Proof sketch]
  Suppose that $\tilde\tau_\chi$ is a tripod partition, since
  otherwise we may consider the dual network $\tilde\G_\chi^*$ and
  partition $\tilde\tau_\chi^*$.  Consider the grove illustrated in
  Fig.~\ref{comb-tripod-extended}.  In each of the three regions
  separated by the triple part, by considering the connections between
  the boundary nodes starting with the outermost pairings and
  progressing towards to the triple part, one may check that the
  connections are as far away from the triple part as possible.  Since
  this is true in each of the three regions, there is not ``room'' to
  make any changes to the grove while keeping the partition
  $\tilde\tau_\chi$.  The argument for $\tilde\tau_-$ is similar.

  If there were more than one grove of type $\tau_\chi$ or $\tau_-$
  within $\G$, we could take a grove of type $\tilde\tau_\chi$ or
  $\tilde\tau_-$ within $\tilde\G_\chi$ and rewire the portion within
  $\G$ to obtain a new grove of type $\tilde\tau_\chi$ or
  $\tilde\tau_-$, a contradiction.
\end{proof}

\begin{figure}[t]
\centering
\begin{subfigure}[t]{.48\textwidth}\includegraphics[width=\textwidth]{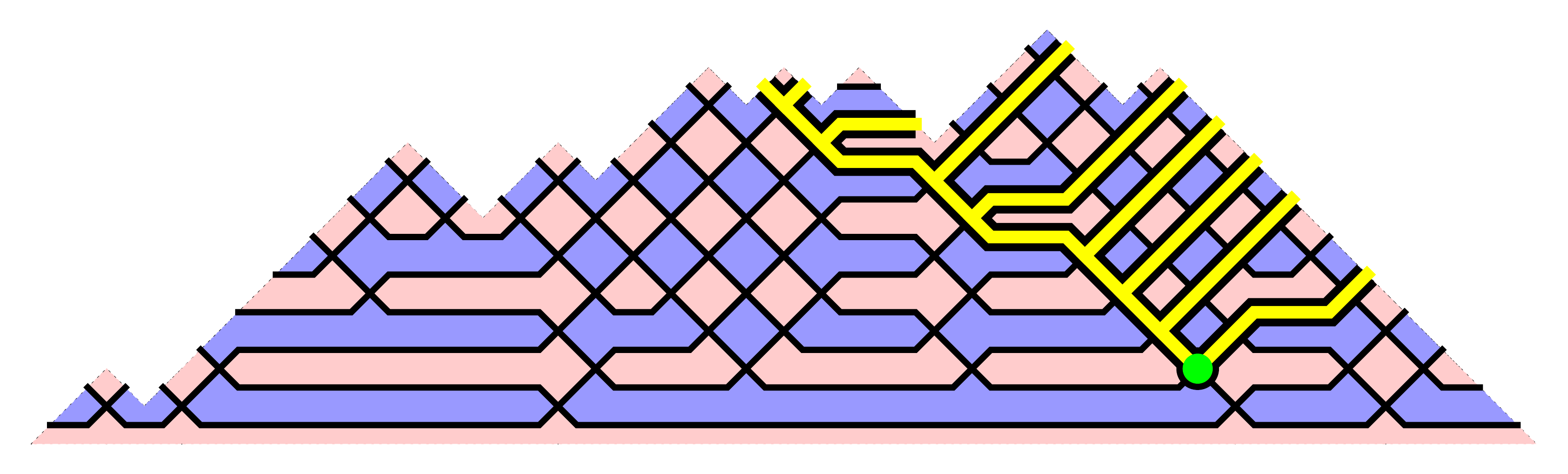}\caption{Strand diagram with ``comb'' (yellow) from one of its crossings $\chi$ (green).}\label{comb}\end{subfigure}\hfill
\begin{subfigure}[t]{.48\textwidth}\includegraphics[width=\textwidth]{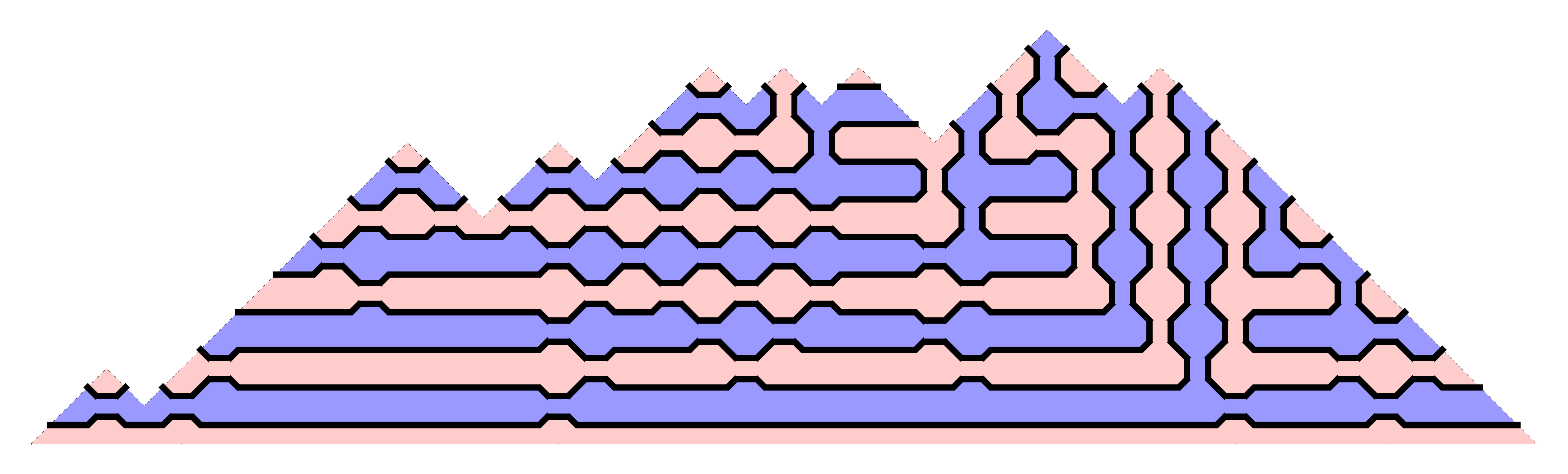}\caption{The unique grove of type $\tau_\chi$.}\label{comb-tripod}\end{subfigure}\hfill
\begin{subfigure}[t]{.48\textwidth}\includegraphics[width=\textwidth]{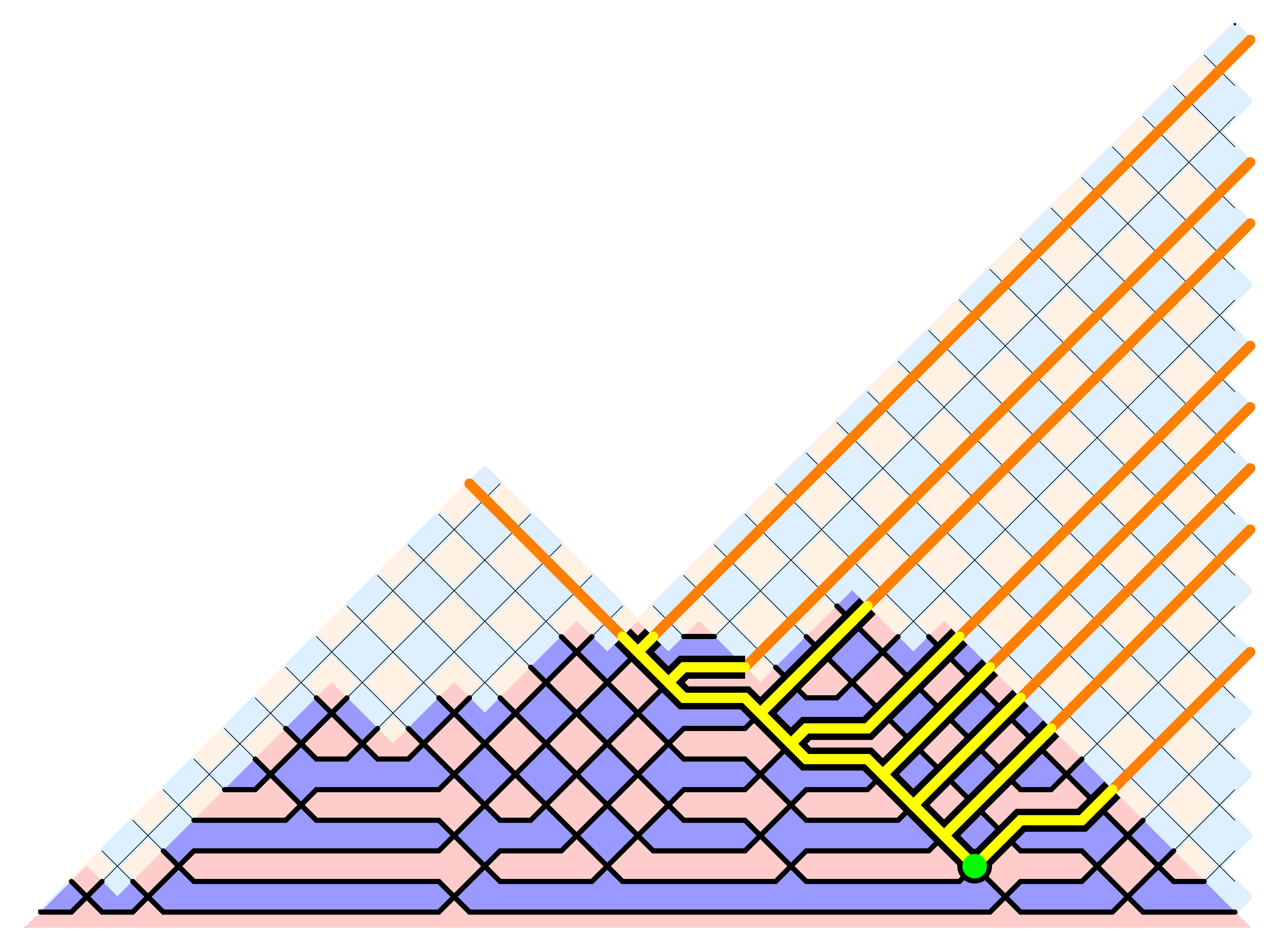}\caption{Strand diagram with comb at $\chi$ in \quad\quad\quad extended graph $\tilde\G_\chi$.}\label{comb-extended}\end{subfigure}
\begin{subfigure}[t]{.48\textwidth}\includegraphics[width=\textwidth]{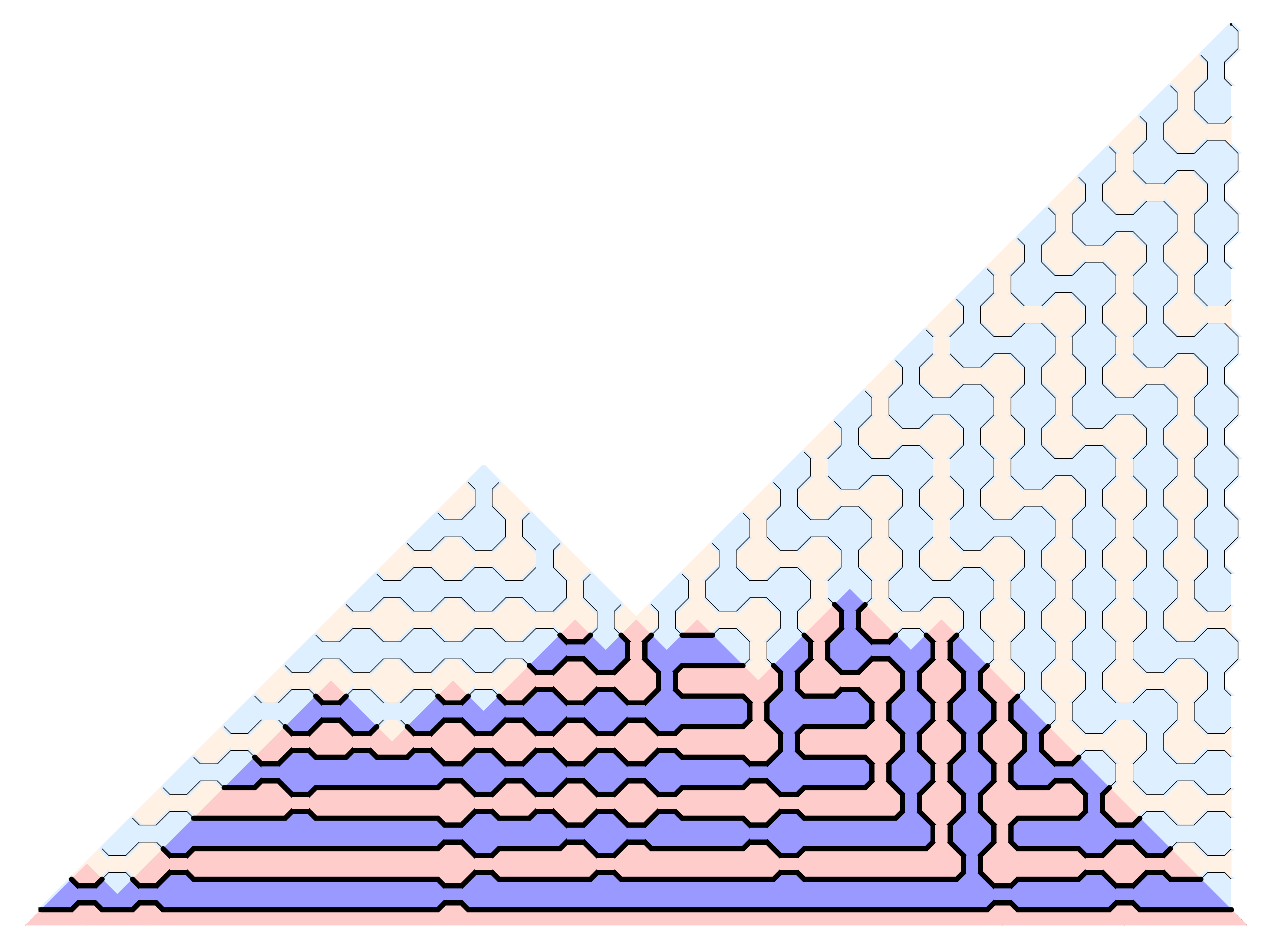}\caption{The unique grove of type $\tilde\tau_\chi$ in extended graph $\tilde\G_\chi$.  $\tau_\chi$ is a tripod or dual-tripod partition.}\label{comb-tripod-extended}\end{subfigure}
\caption{
In panel (a) is the rectilinear strand diagram coming from the Dyck tiling associated with a standard
network.  Also shown in panel (a) is a particular crossing of the strand diagram marked in green, together with the ``comb'' based at that crossing.  The spine of the comb is formed by following the up-left strand segment from the green crossing.  The teeth of the comb are formed by following the up-right strand segments from each crossing in the spine.
  In panel (b) is the associated grove and dual grove.  For each strand crossing in the comb, the vertical edge or dual edge is selected, and otherwise the horizontal edge or dual edge is selected.
Panel (c) shows the strand diagram of the enlarged network, which has a notch above the top tooth of the comb,
and panel (d) shows its associated grove and dual grove.
In this example, the blue grove in panel (d) is a
tripod, and the red grove is a dual tripod.
}
\label{dyck-tiling-tripod}
\end{figure}

\begin{theorem} \label{biratio-tripod}
The number of tripod variables is one more than the number of edges of $\G$.
Each tripod variable is a product of conductances.
The conductances are biratios of tripod variables:
The vertical conductance of a crossing $e$ is
\begin{equation}\label{c=ZZZZ}
c_e = \frac{Z_{\tau_e} Z_{\tau_f}}{Z_{\tau_a} Z_{\tau_b}} = \frac{Z^*_{\tau^*_e} Z^*_{\tau^*_f}}{Z^*_{\tau^*_a} Z^*_{\tau^*_b}}\,,
\end{equation}
where $a$ is the first crossing on the strand going up-and-left from crossing~$e$,
$b$ is the first crossing on the strand going up-and-right from~$e$,
and $f$ is the first crossing on the strand going up-and-left from crossing~$b$.
(If one or more of the crossings $a$, $b$, or $f$ does not exist, then
$\tau_a$, $\tau_b$, or $\tau_f$, respectively, is replaced with the exterior partition~$\tau_-$.)
\end{theorem}

\begin{proof}
  The first statement follows since each edge of $\G$ is associated
  with a crossing in the strand diagram; thus there is one tripod
  variable for each edge, as well as the exterior tripod variable.

  By construction, for each tripod partition $\tau$, there is only
  one grove of partition type $\tau$.  Hence $Z_\tau$ is the product of the
  conductances of the edges in this grove.

  Let $e$ be a crossing of the strand diagram, and let $a$ be the first crossing
  up and left from~$e$.  Comparing the combs based at $e$ and $a$, the comb based
  at $e$ has one more tooth, which starts at $e$.  Thus $Z_{\tau_e}/Z_{\tau_a}$ is
  the product of the vertical conductances of all crossings along
  the up-and-right strand segment starting at $e$.  (If there is no such crossing $a$,
  then the same statement holds when we replace $\tau_a$ with $\tau_-$.)
  The ratio $Z_{\tau_b}/Z_{\tau_f}$ has a similar interpretation, except starting
  at the crossing $b$ (if crossing $b$ exists).  Since crossing $b$ is the
  first crossing on the up-and-right strand segment from crossing $e$, the ratio
  $(Z_{\tau_e}/Z_{\tau_a})/(Z_{\tau_b}/Z_{\tau_f})$ is just the vertical conductance at $e$
  (with $\tau_x$ replaced by $\tau_-$ if crossing $x$ does not exist).

  The formula in terms of the dual groves follows because
  \[Z^*_{\tau^*} = Z_\tau / (\text{product of all edge weights})\]
  for each $\tau$.
\end{proof}

There is one monomial relation between either the tripod variables or
the dual-tripod variables: Consider any lower-most crossing $\chi$ of
the strand diagram.  The edge associated with $\chi$ is horizontal for
either the network or the dual network; for convenience say that it is
horizontal for the network.  The conductance of the associated edge
does not appear as a factor in the tripod variable $Z_{\tau_\chi}$.
Since each edge weight is the biratio of tripod variables, each tripod
variable is a Laurent monomial in the set of tripod variables.
Because $\chi$ is a lowermost crossing, its tripod variable appears in
the biratio of only the associated edge.  Thus $Z_{\tau_\chi}$ is a
Laurent monomial in the \textit{other\/} tripod variables, giving a
nontrivial monomial relation.

For the purpose of reconstruction, we compute \textit{normalized\/} tripod variables
using the response matrix $L$, as explained below.
Combined with Theorem~\ref{biratio-tripod}, this will give an explicit expression for each
conductance as a rational function of $L$.

We first need to compute the response matrix for the
enlarged graph in terms of the original response matrix
$L$.  The enlargement of the graph $\G$ to $\tilde\G_\chi$ can be expressed
as a sequence of steps which either (1) adjoin an edge connecting
adjacent nodes, (2) internalize a node, i.e., demoting to a regular
(internal) vertex, or (3) insert a new node between a pair of adjacent
nodes, with no edges connecting it to the rest of the network.
Let $\tilde\G$ denote an intermediate graph in this sequence of steps,
and let $\tilde L$ denote its response matrix.
Adjoining an edge of conductance $c$ between nodes $i$ and $i+1$
increases $\tilde L_{i,i+1}=\tilde L_{i+1,i}$ by $c$ and decreases
$\tilde L_{i,i}$ and $\tilde L_{i+1,i+1}$ by $c$.
Inserting a new node
inserts an all-zero row and column into $\tilde L$.  When internalizing
a node, the response matrix is updated by Schur reduction \cite{zhang}.

With $e$, $a$, $b$, and $f$ denoting the four crossings from Theorem~\ref{biratio-tripod},
the graph extensions $\tilde\G_a=\tilde\G_e$ are the same, as are $\tilde\G_f=\tilde\G_b$.
From Lemma~\ref{tau-unique}, we see that for any choice of the edge conductances in
the graph extensions, we have
\begin{equation}
\frac{Z_{\tau_e}}{Z_{\tau_a}} = 
\frac{Z_{\tilde\tau_e}(\tilde\G_e)/Z_{unc}(\tilde\G_e)}{Z_{\tilde\tau_a}(\tilde\G_a)/Z_{unc}(\tilde\G_a)}
\quad\quad\text{and}\quad\quad
\frac{Z_{\tau_f}}{Z_{\tau_b}} = 
\frac{Z_{\tilde\tau_f}(\tilde\G_f)/Z_{unc}(\tilde\G_f)}{Z_{\tilde\tau_b}(\tilde\G_b)/Z_{unc}(\tilde\G_b)}\,,
\end{equation}
since the weights coming from the new edges in the graph extension cancel.
If one or more of the crossings $a$, $b$, or $f$ do not exist, so that $\tau_-$ is used in equation~\eqref{c=ZZZZ},
we choose the corresponding graph extensions for $\tau_-$ so that
these above equalities hold.
Consequently, we can compute these above ratios of tripod variables
using a ratio of the Pfaffians as described in Section~\ref{sec:tripod-pf},
since the $\tilde Z_{unc}$ terms cancel.
These ratios of tripod variables are sufficient to compute the edge conductances using \eqref{c=ZZZZ}.

When doing computations, it is easier to consider the limiting case when the new edges in the
graph extensions have infinite conductance.

\subsection{Reconstruction example} \label{sec:reconstruction}

It is instructive to work out the reconstruction problem for a small example.  Consider the following circular-planar network with 5 nodes, 2 internal vertices, and 7 edges together with the response matrix $L$:
\[
\raisebox{-0.45\height}{\includegraphics[width=1.5in]{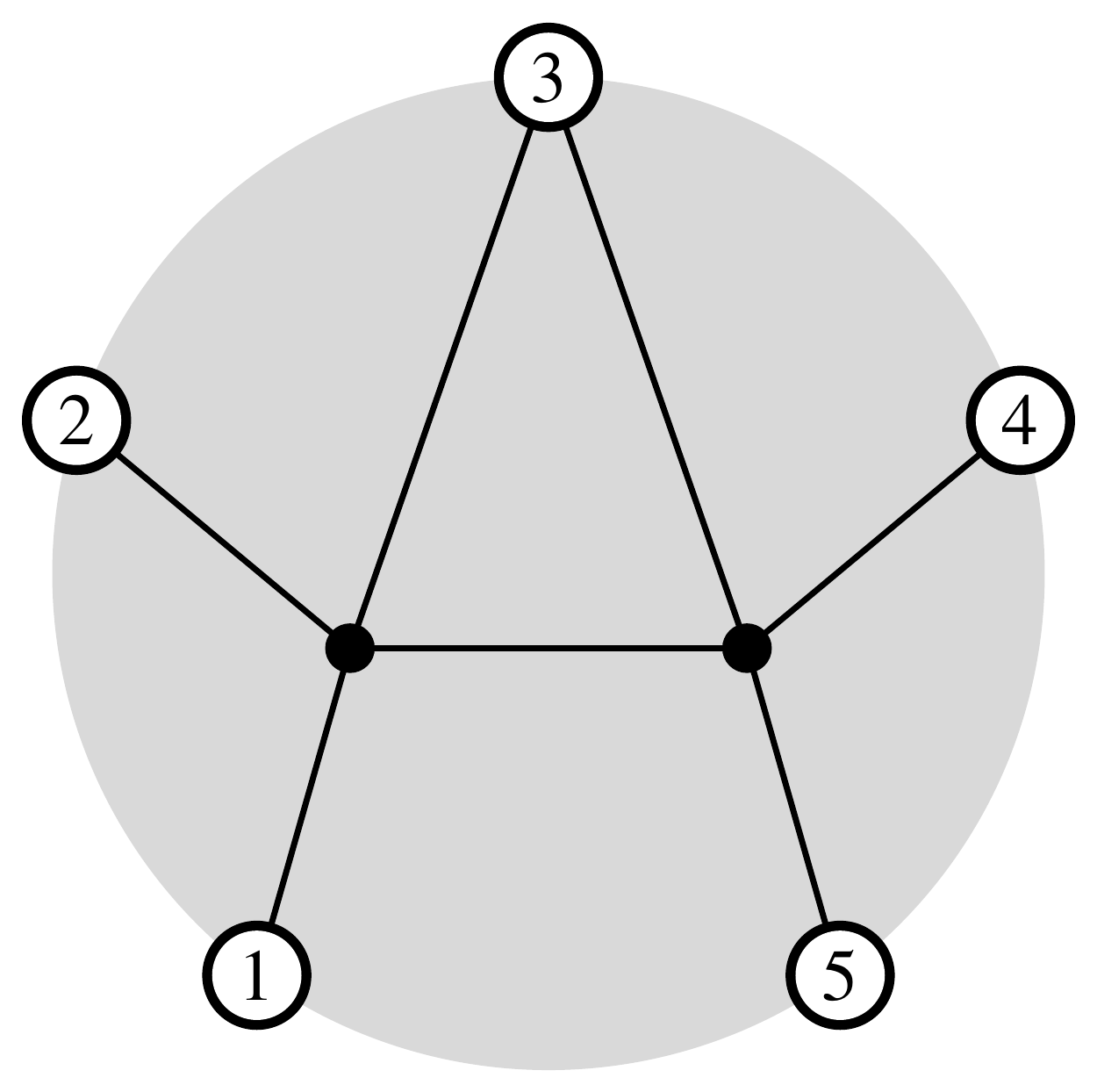}}\quad\quad\quad\quad 
L=\begin{bmatrix}
 -100 & 40 & 45 & 10 & 5 \\
 40 & -88 & 36 & 8 & 4 \\
 45 & 36 & -99 & 12 & 6 \\
 10 & 8 & 12 & -40 & 10 \\
 5 & 4 & 6 & 10 & -25 \\
\end{bmatrix}
\]
The first step is to draw the strand diagram:
\[
\includegraphics[width=1.5in]{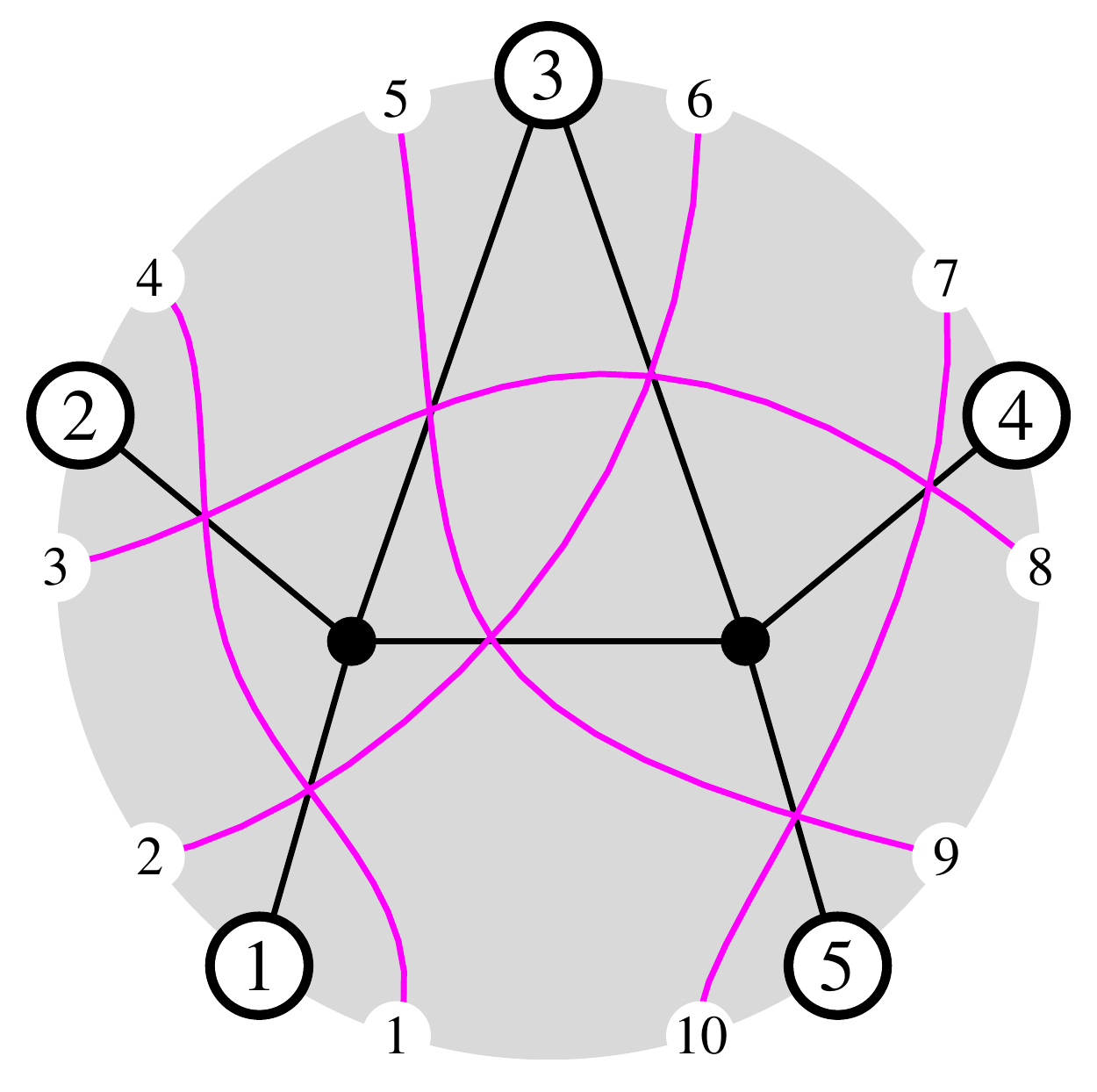}
\]
We see that the above network is minimal and has strand matching 
\[\{\{1,4\},\{2,6\},\{3,8\},\{5,9\},\{7,10\}\}.\]
The associated Dyck tiling and standard minimal network are
\[
\includegraphics[width=1.4in]{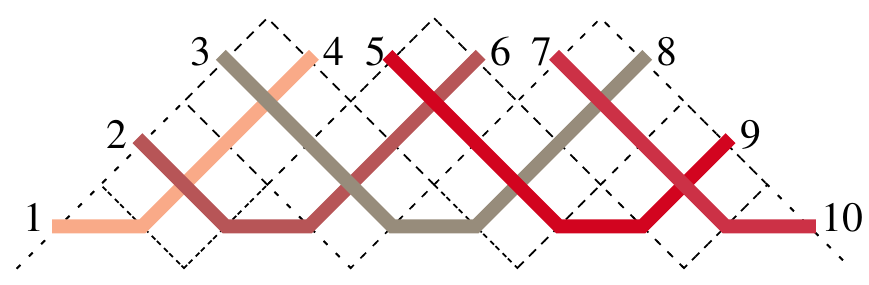}\quad\quad
\includegraphics[width=1.4in]{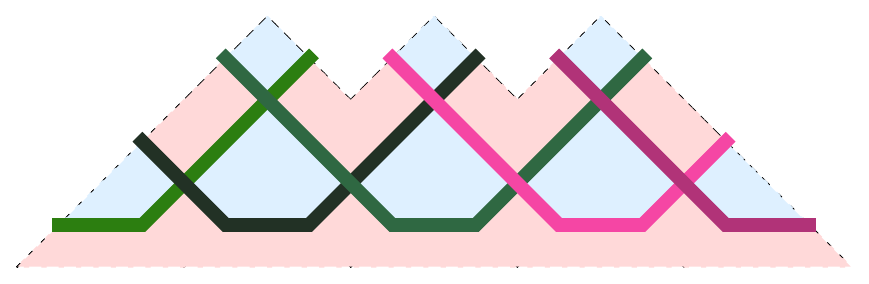}\quad\quad
\includegraphics[width=1.4in]{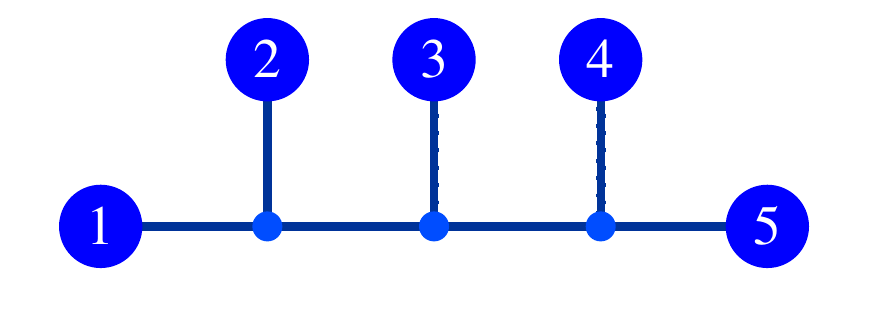}
\]
This is a small example for which all the Dyck tiles are single-box tiles,
but a more complicated example might look like Fig.~\ref{strands}.
Next we carry out the reconstruction for the standard minimal network.

\[
\raisebox{-.4\height}{\includegraphics[width=1.2in]{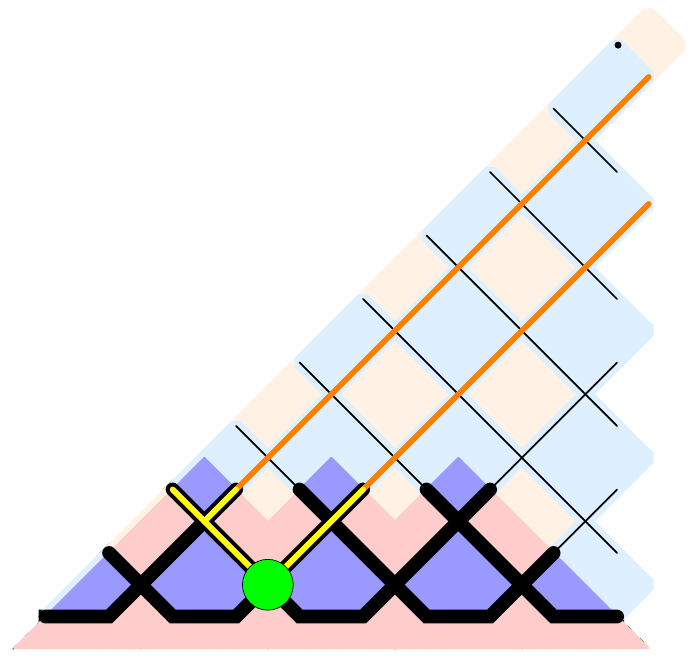}}\quad\quad
\raisebox{-.4\height}{\includegraphics[width=1.2in]{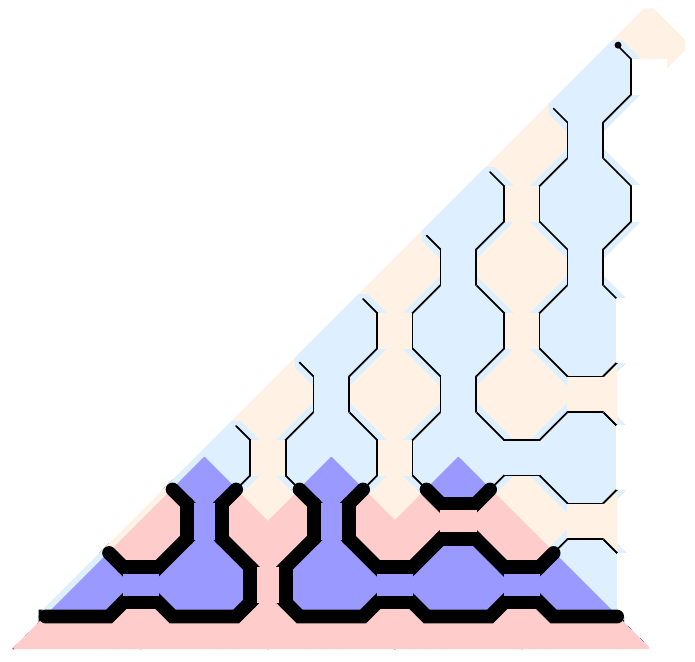}}\quad\quad
\Pf\begin{bmatrix}
  0 & 40 & 45 & 5 \\
  & 0 & 0 & 4 \\
  & & 0 & 6 \\
  & & & 0 \\
\end{bmatrix}
=60
\]

For the next two crossings we glue nodes $3$ and $4$ together to form node $3\&4$,
with the new response matrix being
\[\footnotesize\begin{bmatrix}
-100 & 40 & 55 & 5 \\
 40 & -88 & 44 & 4 \\
 55 & 44 & -115 & 16 \\
 5 & 4 & 16 & -25 \\
\end{bmatrix}.\]
Using the dual tripod Pfaffian formula with $R=\{2\}$ and $B=\{5\}$ with singleton nodes $1$ and $3\&4$, we find
\[
\raisebox{-.4\height}{\includegraphics[width=1.2in]{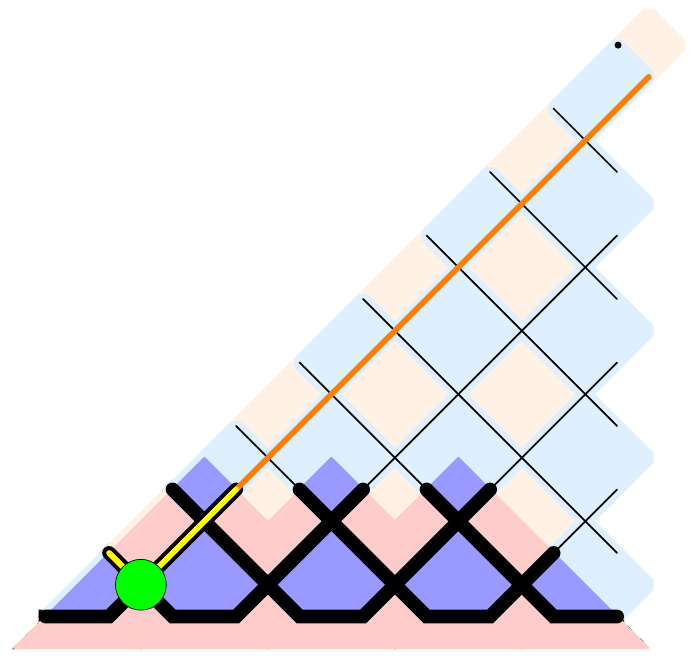}}\quad\quad
\raisebox{-.4\height}{\includegraphics[width=1.2in]{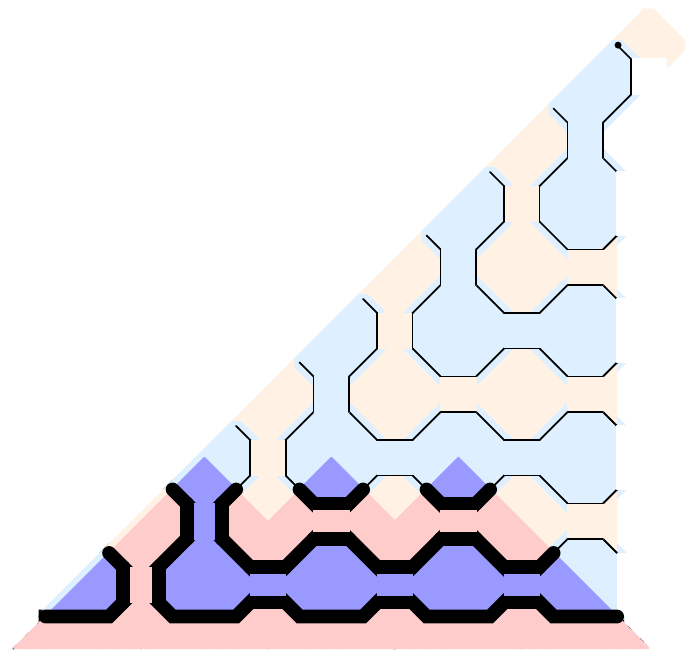}}\quad\quad
\Pf\begin{bmatrix}
0&4\\&0\\
\end{bmatrix}
=4.
\]
Using the tripod Pfaffian formula (with one singleton node $3\&4$) we obtain
\[
\raisebox{-.4\height}{\includegraphics[width=1.2in]{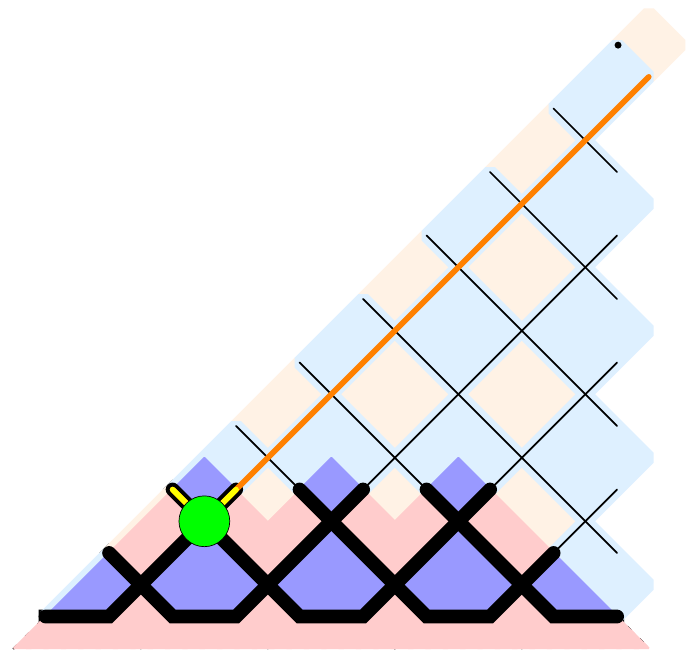}}\quad\quad
\raisebox{-.4\height}{\includegraphics[width=1.2in]{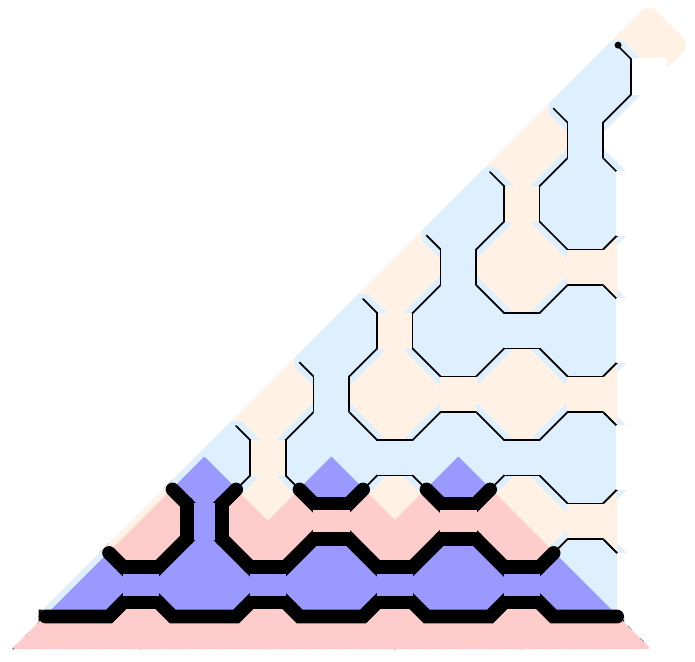}}\quad\quad
\Pf\begin{bmatrix}
 0 & 40 & -5 & 95 \\
 -40 & 0 & 4 & 4 \\
 5 & -4 & 0 & 5 \\
 -95 & -4 & -5 & 0 \\
\end{bmatrix}
= 600
\]

The next two crossings use the original response matrix
\[
\raisebox{-.4\height}{\includegraphics[width=1.2in]{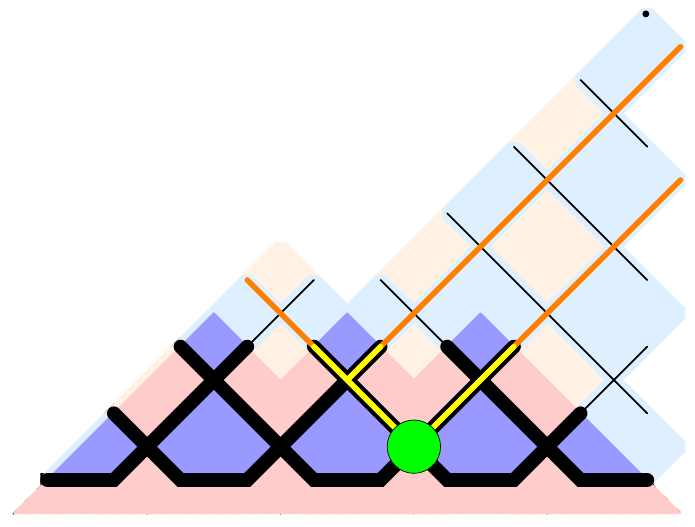}}\quad\quad
\raisebox{-.4\height}{\includegraphics[width=1.2in]{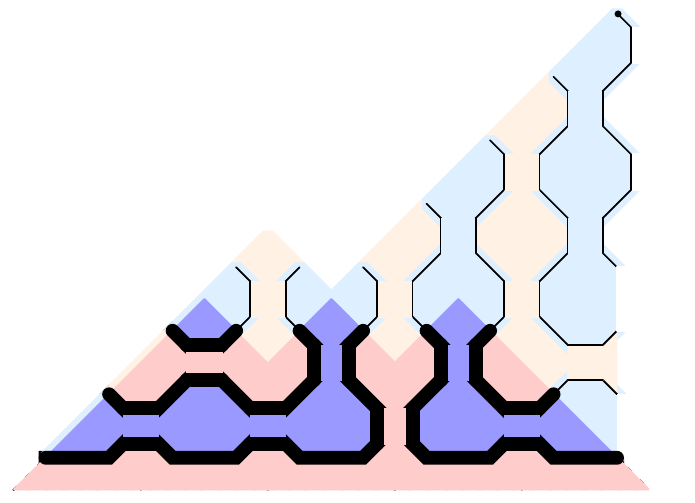}}\quad\quad
\Pf\begin{bmatrix}
 0 & 45 & 10 & 5 \\
 -45 & 0 & 0 & 6 \\
 -10 & 0 & 0 & 10 \\
 -5 & -6 & -10 & 0 \\
\end{bmatrix}
=390
\]

\[
\raisebox{-.4\height}{\includegraphics[width=1.2in]{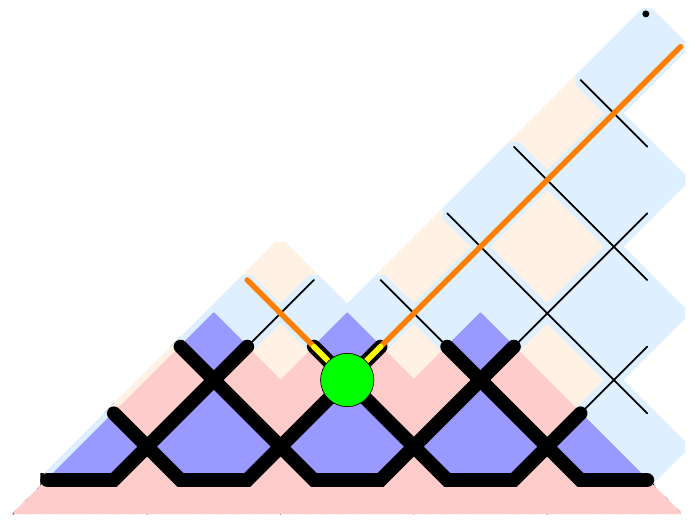}}\quad\quad
\raisebox{-.4\height}{\includegraphics[width=1.2in]{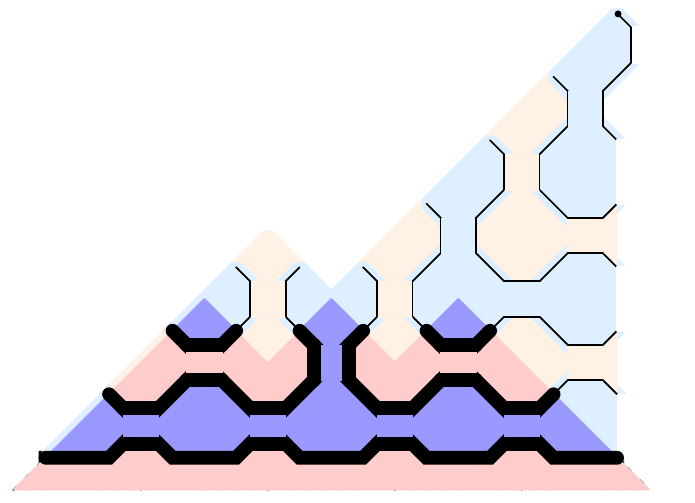}}\quad\quad
\Pf\begin{bmatrix}
 0 & 45 & -5 & 55 \\
 -45 & 0 & 6 & 6 \\
 5 & -6 & 0 & 9 \\
 -55 & -6 & -9 & 0 \\
\end{bmatrix}
=765
\]

For the next two crossings we glue nodes 2 and 3 together to obtain reponse matrix
\[\footnotesize\begin{bmatrix}
-100 & 85 & 10 & 5 \\
 85 & -115 & 20 & 10 \\
 10 & 20 & -40 & 10 \\
 5 & 10 & 10 & -25 \\
\end{bmatrix}\]

\[
\raisebox{-.4\height}{\includegraphics[width=1.2in]{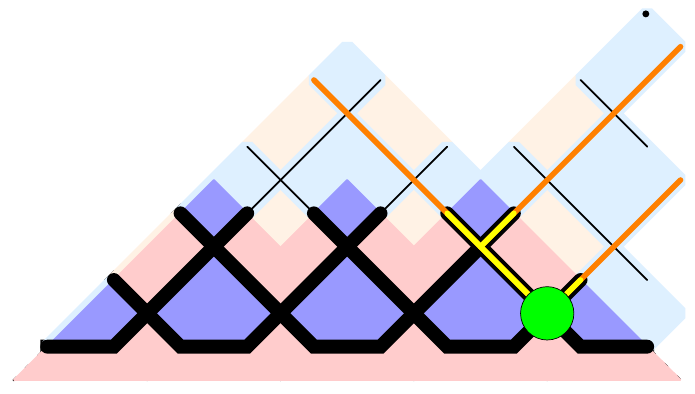}}\quad\quad
\raisebox{-.4\height}{\includegraphics[width=1.2in]{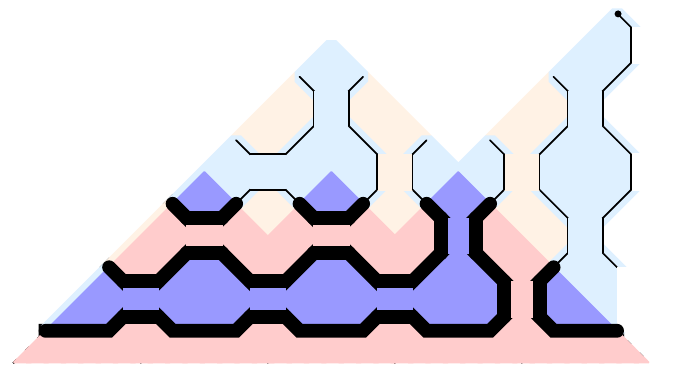}}\quad\quad
\Pf\begin{bmatrix}
0 & 10 \\
& 0\\
\end{bmatrix}
=10
\]

\[
\raisebox{-.4\height}{\includegraphics[width=1.2in]{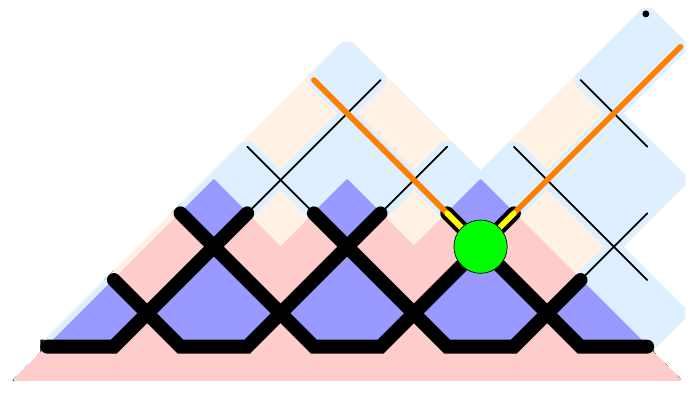}}\quad\quad
\raisebox{-.4\height}{\includegraphics[width=1.2in]{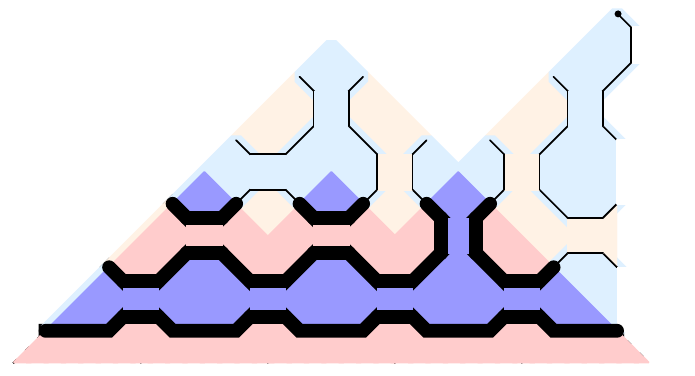}}\quad\quad
\Pf\begin{bmatrix}
0 & 10 & -5 & 10 \\
 -10 & 0 & 10 & 10 \\
 5 & -10 & 0 & 15 \\
 -10 & -10 & -15 & 0 \\
\end{bmatrix}
=300
\]

For each of the exterior partitions we get $\Pf\footnotesize\begin{bmatrix}0&5\\&0\end{bmatrix}=5$.

To summarize, the normalized tripod variables are
\[
\includegraphics[width=1.5in]{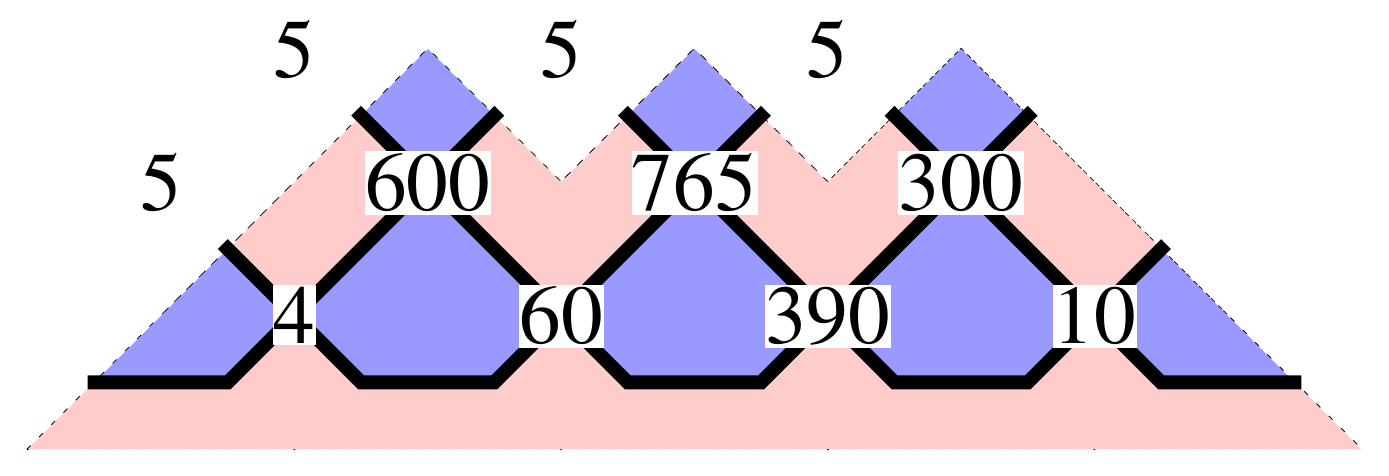}
\]
the ratios of the tripod variables (along the up-and-left strand) are
\[
\includegraphics[width=1.5in]{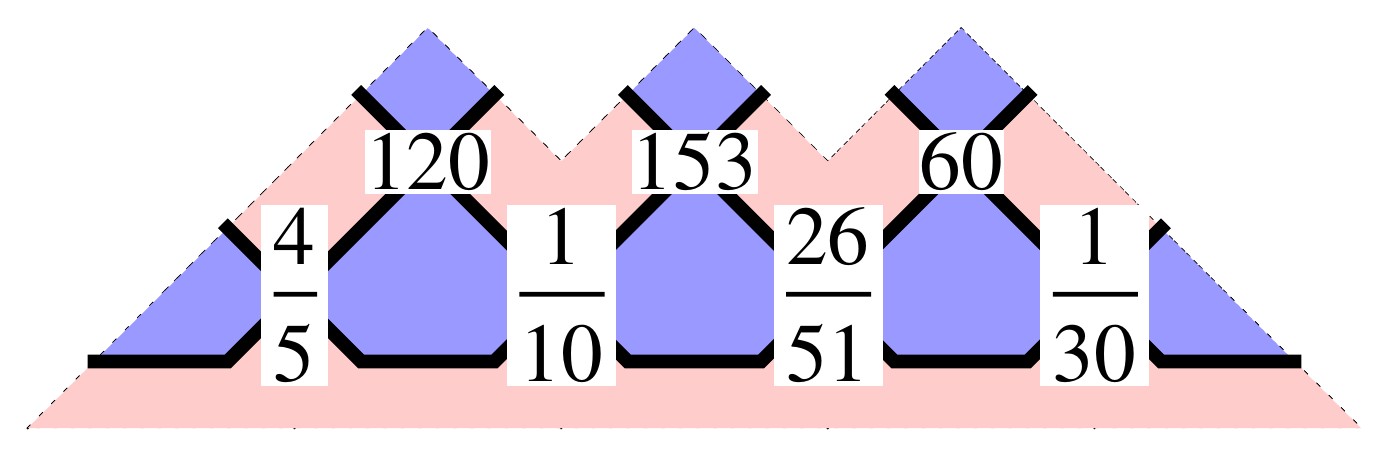}
\]
and the biratios, which equal the vertical conductances, are
\[
\includegraphics[width=1.5in]{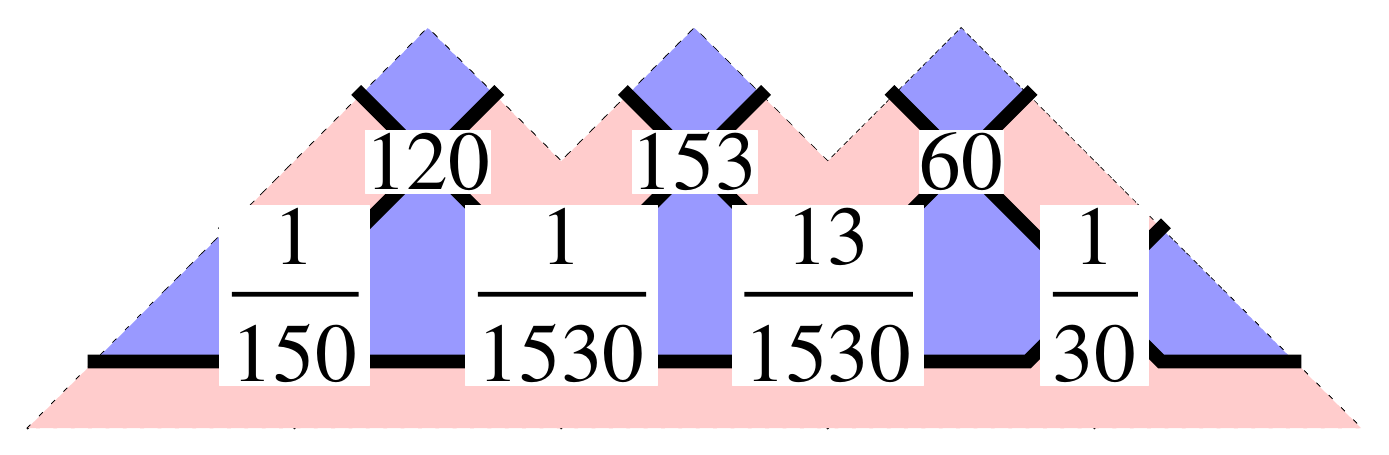}
\]
These are the conductances for the vertical edges and the reciprocal conductances for the horizontal edges,
so the primal network's conductance are
\[
\includegraphics[width=1.5in]{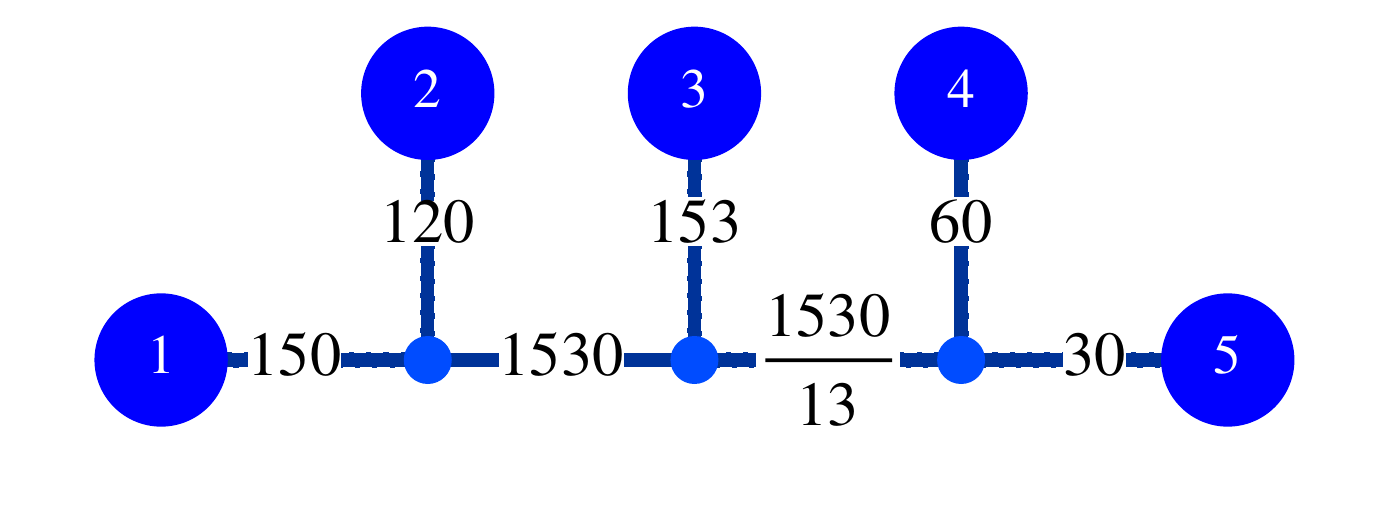}
\]
Upon doing a $\Delta$-Y transformation, we obtain the conductances of the original network:
\[
\includegraphics[width=1.5in]{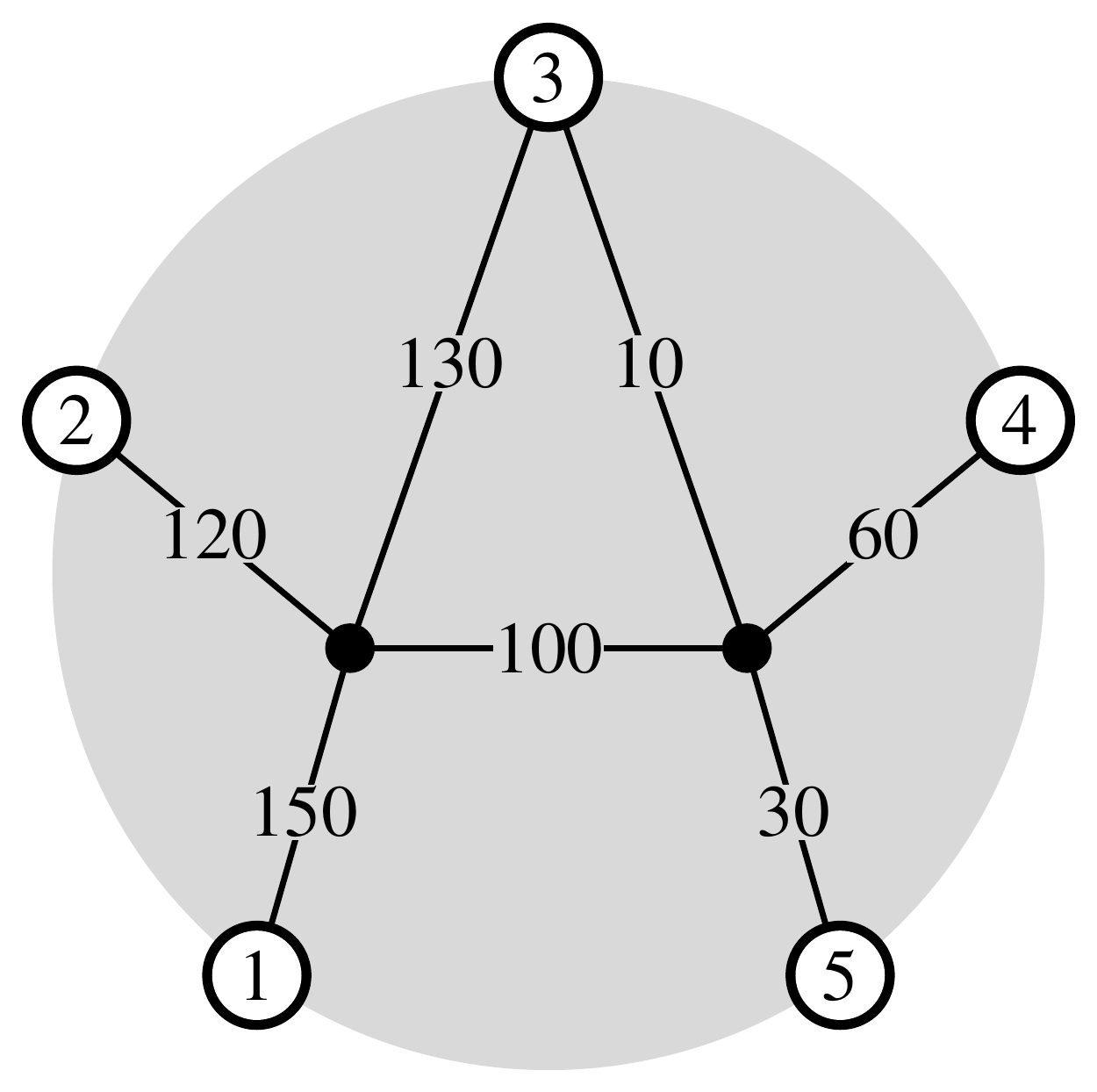}
\]

\subsection{B variables}
Above we defined the tripod variables of a standard network.  Here we associate a different set of variables,
called \textbf{B variables}, to the vertices and faces of an arbitrary minimal network,
or equivalently, to the cells of its associated strand diagram.
The defining characteristic of $B$ variables is that each conductance $c=c(e)$ satisfies
\be\label{BBBB}c=\frac{B_v B_{v'}}{B_f B_{f'}},\ee
where $v$ and $v'$ are the endpoints of edge $e$ and $f$ and $f'$ are the adjacent faces.
This biratio formula is very similar to the biratio formula~\eqref{c=ZZZZ} for the tripod variables,
but the tripod variables do not comprise a set of $B$~variables because
there are different tripod variables along different portions of the outer face,
and there can only be one $B$~variable per face.  The $B$~variables are not
uniquely defined by \eqref{BBBB}, since we can multiply the $B$ variables on the cells to one
side of strand without affecting \eqref{BBBB}, but these are the only degrees of freedom.
We show here how to construct a valid set of $B$ variables from a minimal strand diagram.

\begin{figure}[b!]
\begin{subfigure}[c]{.48\textwidth}\centerline{\includegraphics[width=.7\textwidth]{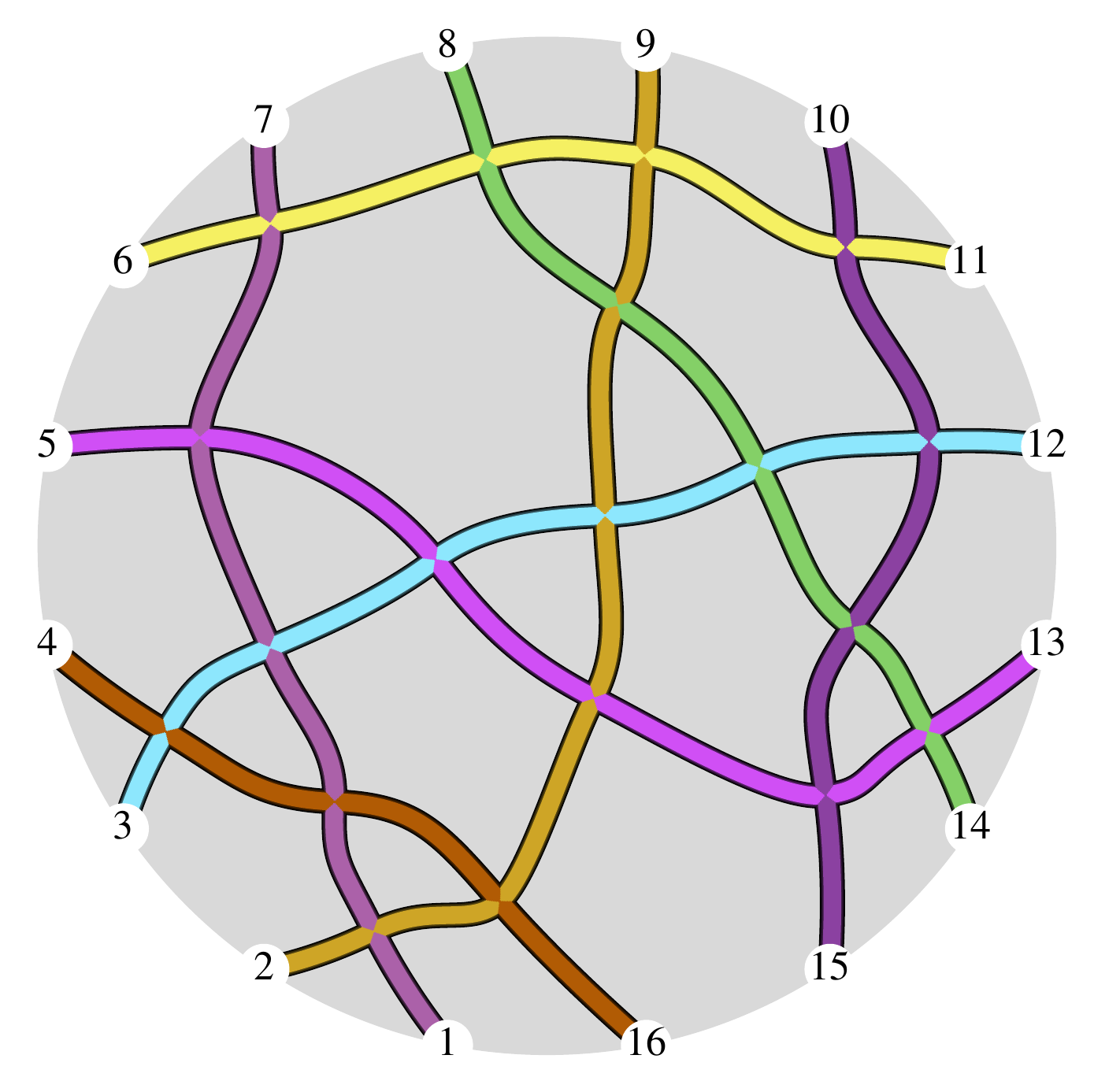}}\caption{Minimal strand diagram (from Fig.~\ref{strands-network})\label{fig:minimal-nonmonotone}}\end{subfigure}
\hfill
\begin{subfigure}[c]{.48\textwidth}\includegraphics[width=\textwidth]{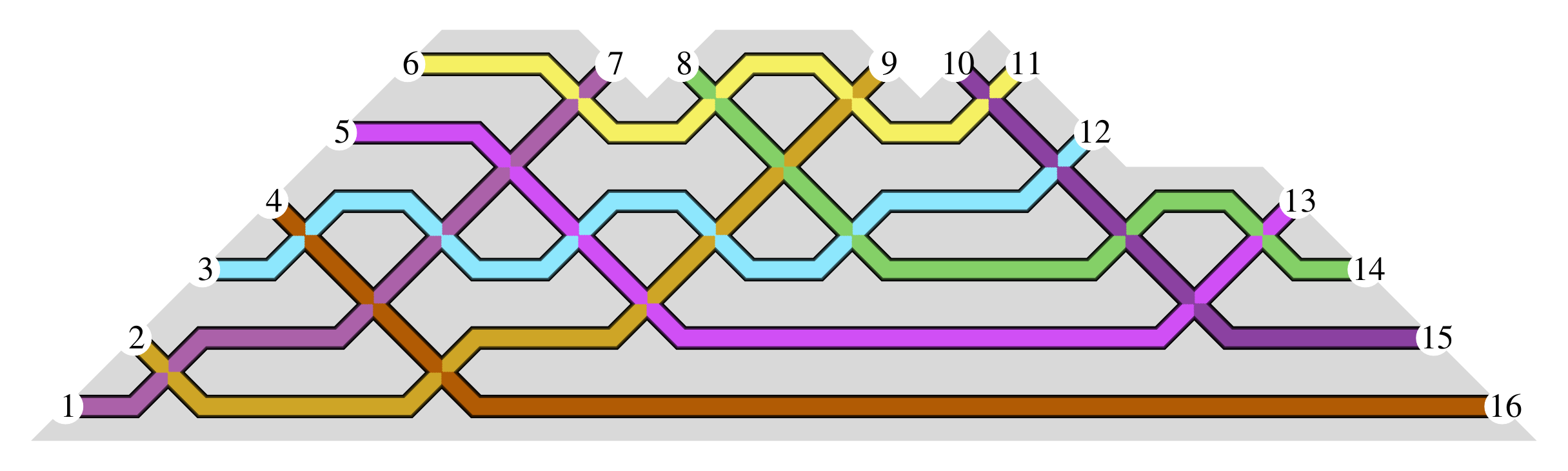}\caption{Strand diagram drawn rectilinearly with each strand going from left to right monotonically ($x$-coordinate only increases along path)\label{fig:minimal-monotone}}\end{subfigure}
\begin{subfigure}[t]{.48\textwidth}\includegraphics[width=\textwidth]{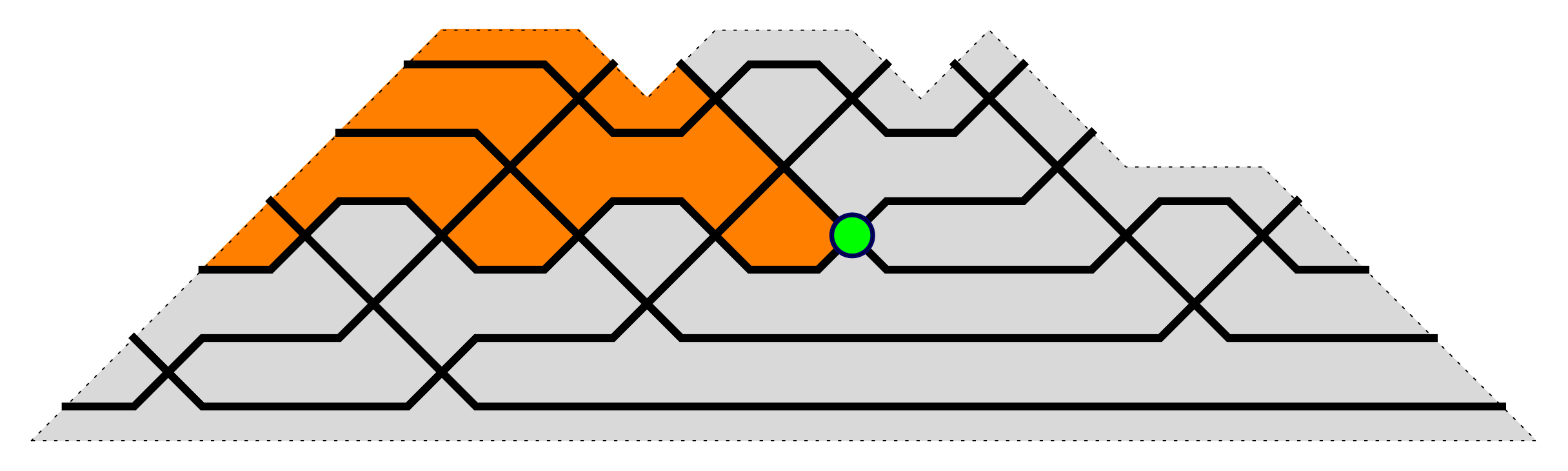}\caption{Region to the left of a crossing\label{fig:left-region}}\end{subfigure}
\hfill
\begin{subfigure}[t]{.48\textwidth}\includegraphics[width=\textwidth]{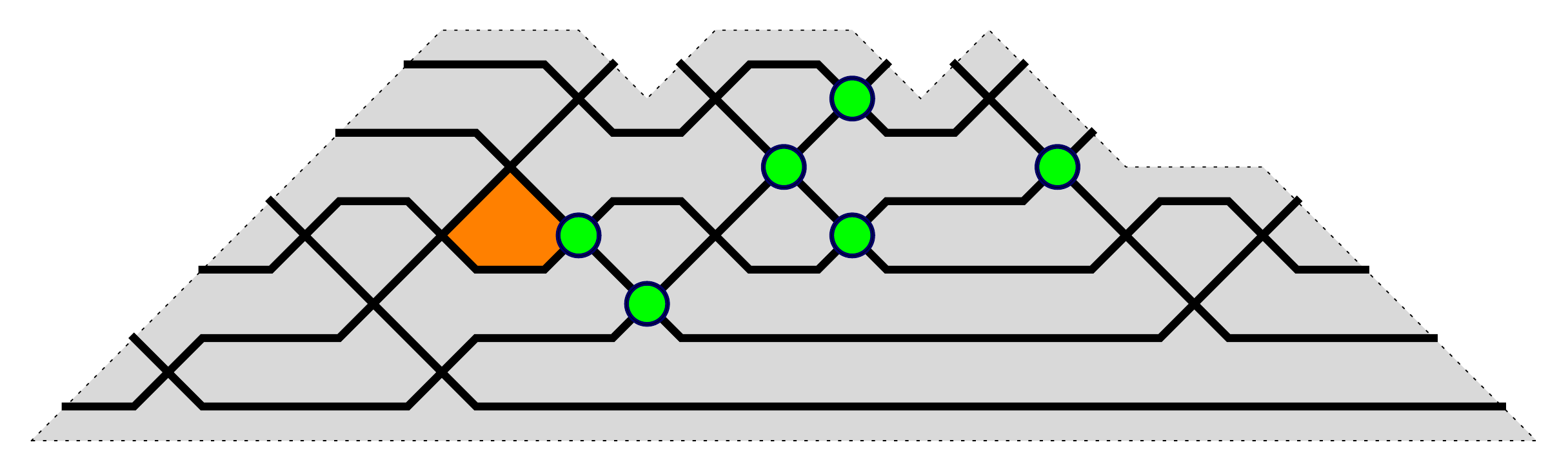}\caption{A cell together with the crossings to which it is to the left}\label{fig:right-crossings}\end{subfigure}
\caption{Construction of the $B$-variables for minimal networks.
There is one $B$-variable for each cell in the strand diagram, which is the product of the horizontal edge conductances
to which it is to the left, when drawn as above.}
\label{dyck-tiling-B}
\end{figure}

We can draw a minimal strand diagram so that each strand $(i,j)$ with $i<j$ goes from $i$ to $j$ monotonically left to right, i.e., without backtracking left, as shown in Figure~\ref{fig:minimal-monotone}.  One way to construct such an embedding is to start with the strand diagram associated to the Dyck tiling of the standard minimal network (Figure~\ref{strands}), in which each strand goes monotonically from left to right, and then perform a sequence of $\Delta$-Y transformations to recover the original minimal network, while preserving the left-to-right monotonicity of the strands, as indicated in Figure~\ref{BYD}.

Consider a crossing between two strands $(i_1,j_1)$ and $(i_2,j_2)$ where $i_1<i_2<j_1<j_2$.  These two strands define four regions, and we call the region between nodes $i_1$ and $i_2$ the \textit{left region}.  In the left-to-right monotone embedding of the strand diagram, the ``left region'' lies entirely to the left of the crossing (Figure~\ref{fig:left-region}), but as a formal definition, the ``left region'' makes sense regardless of how the strands are embedded in the plane, e.g., we could find the left region of a crossing directly from Figure~\ref{fig:minimal-nonmonotone} without the embedding in Figure~\ref{fig:minimal-monotone}.  We define the \textit{horizontal conductance\/} of a crossing to be the conductance of the associated edge in the primal network or dual network, according to whether or not the edge's endpoints lie in the left and right regions of the crossing.  (This is consistent with our earlier definition for standard minimal networks.)

\begin{figure}[t]
\center{\includegraphics[width=.3\textwidth]{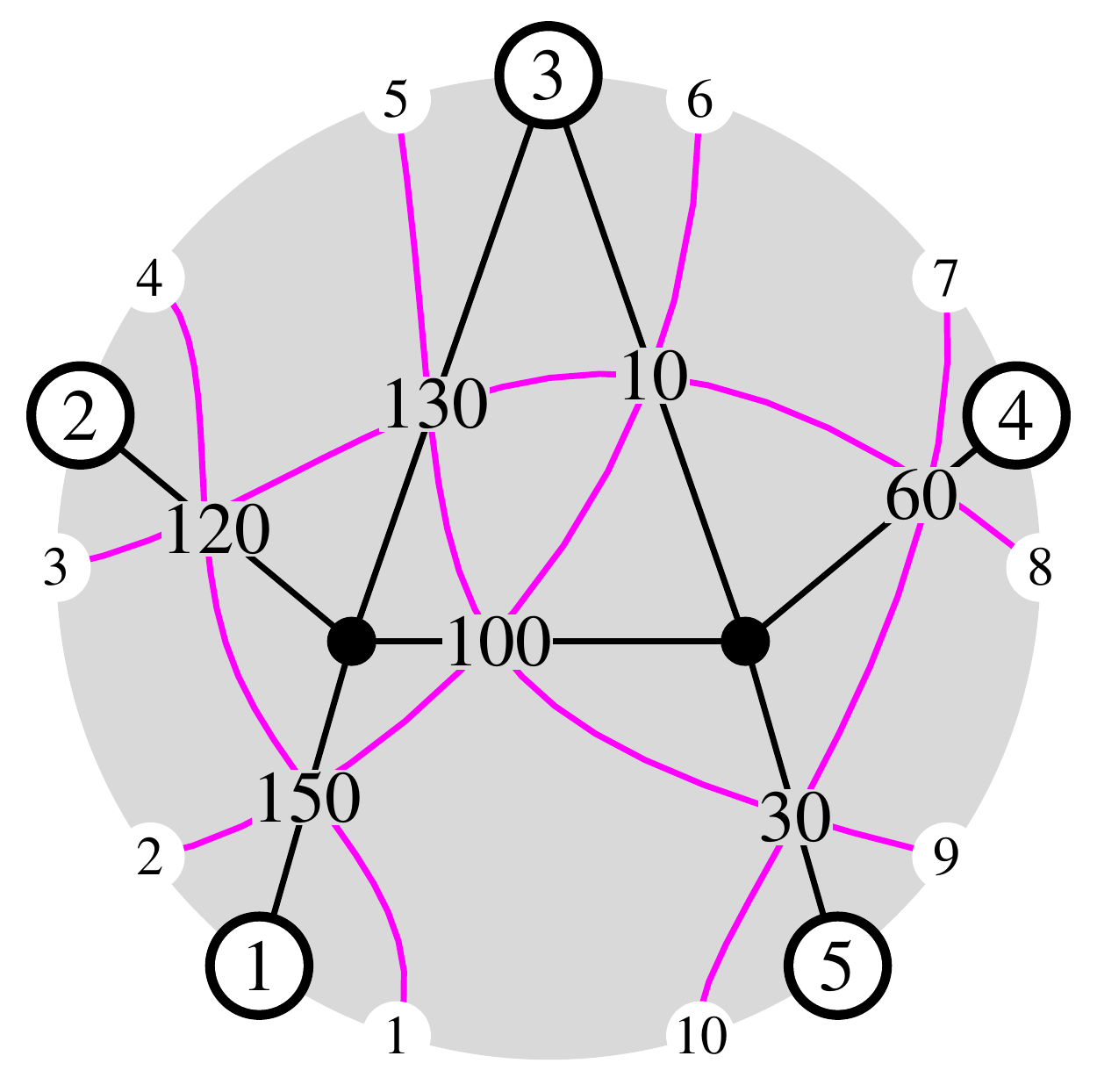}\hfill\includegraphics[width=.3\textwidth]{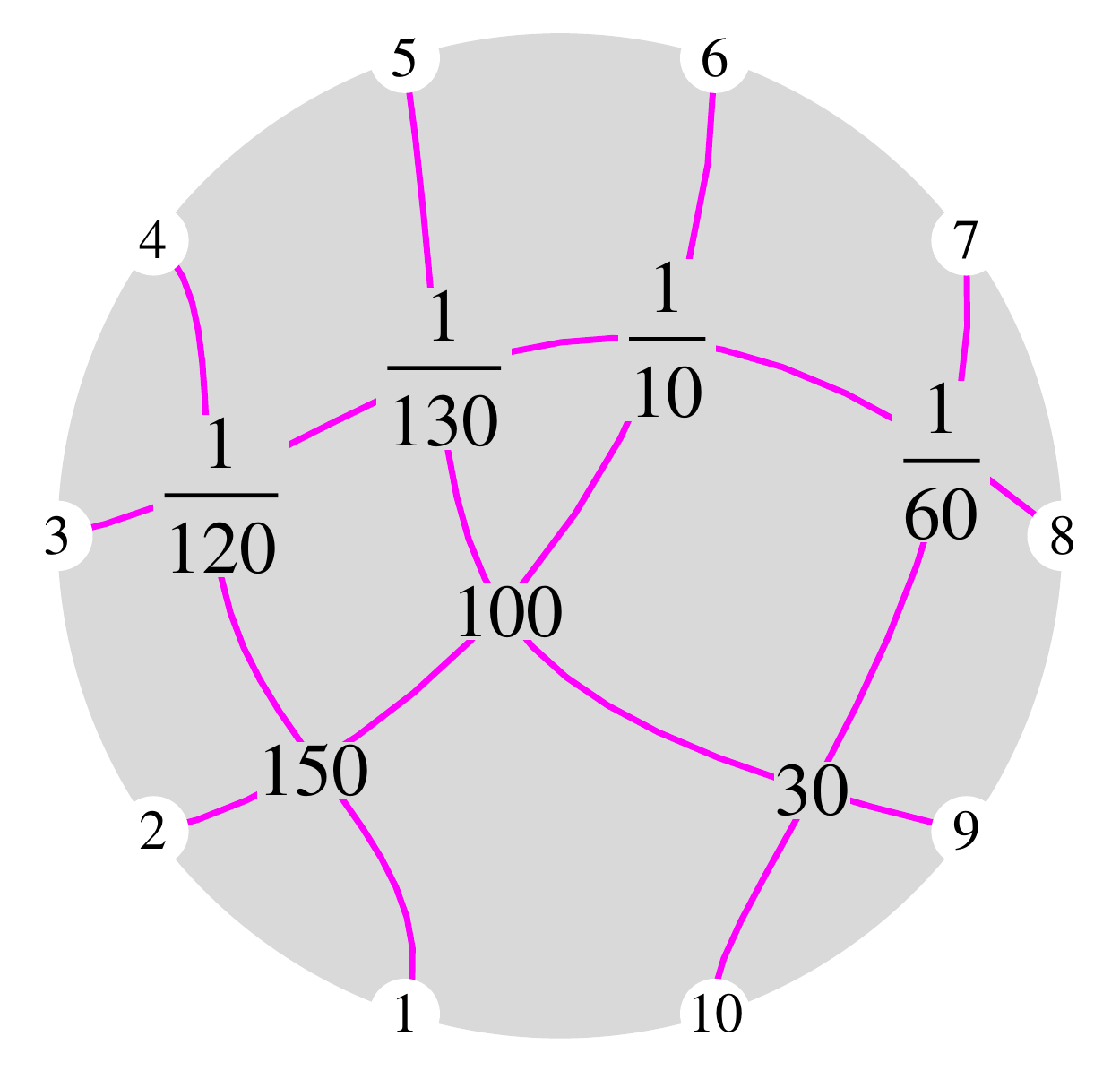}\hfill\includegraphics[width=.3\textwidth]{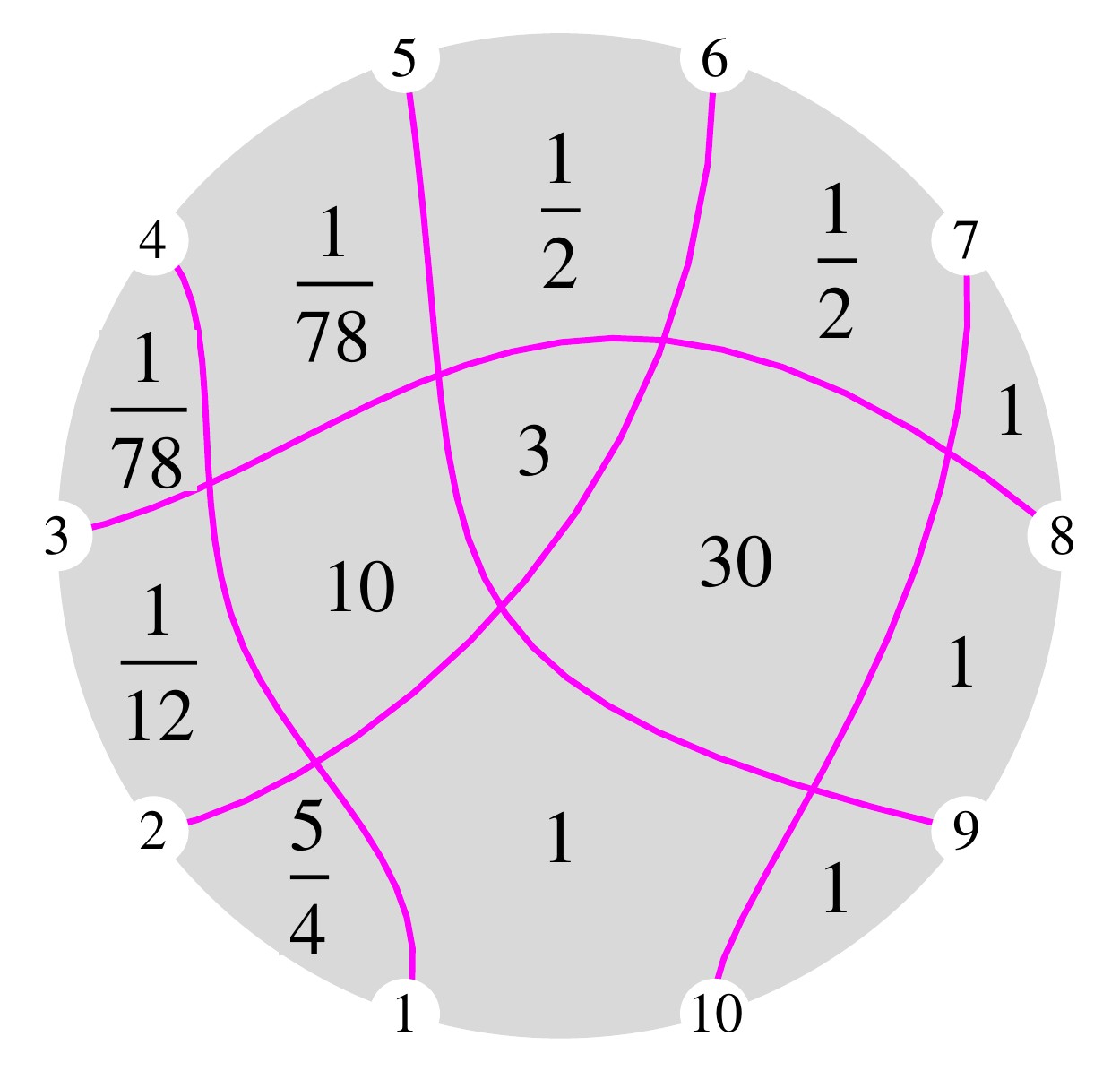}}
\caption{Construction of $B$ variables.  At left is the network, conductances, and strand diagram from the reconstruction example in Section~\ref{sec:reconstruction}.  In the middle is the strand diagram with horizontal conductances.  At right is the strand diagram with $B$ variables.\label{fig:b-variables}}
\vspace{-20pt}
\end{figure}
\enlargethispage{20pt}
\begin{figure}[b!]
\center{\hfill\includegraphics[width=1.55in]{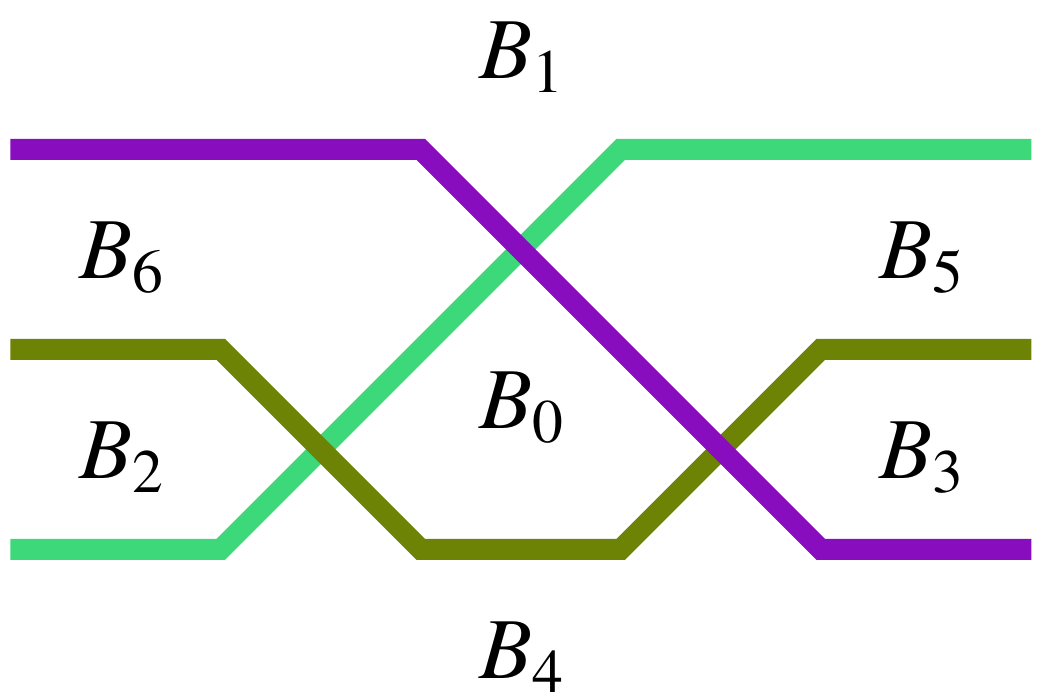}\hfill\includegraphics[width=1.55in]{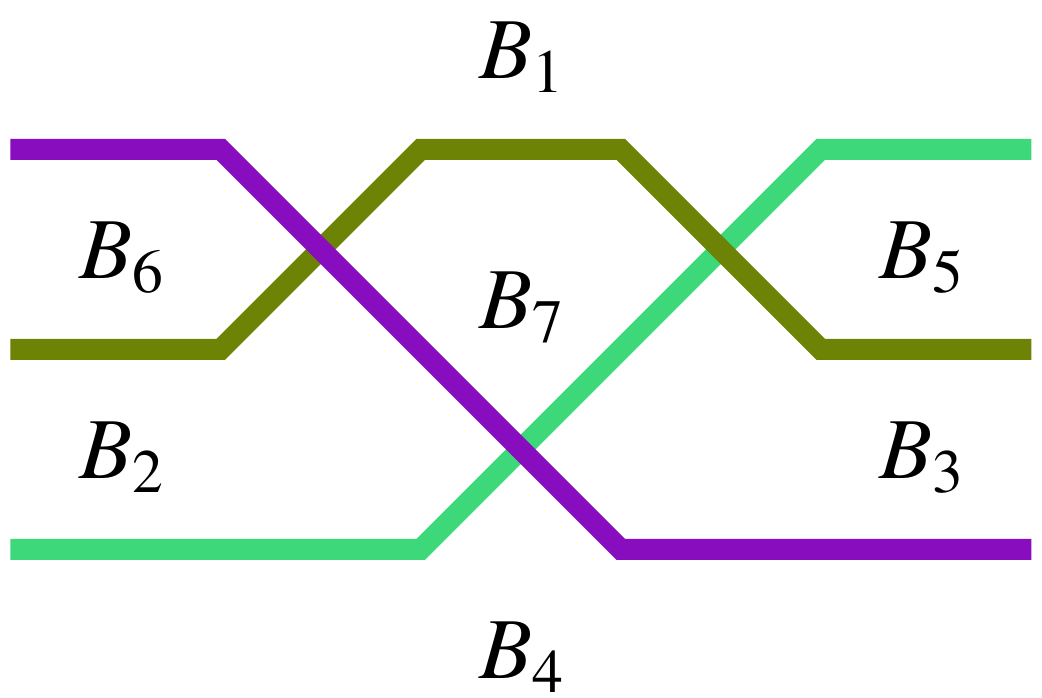}\hfill}
\caption{Transformation of $B$ variables under a Y-$\Delta$ transformation. We have $B_0 B_7=B_1 B_4+B_2 B_5+B_3 B_6$.\label{BYD}}
\end{figure}

For each cell of the minimal strand diagram (corresponding to a face or vertex of the network), consider the set of
crossings for which the cell is part of the left region (see for example Figure~\ref{fig:right-crossings}).  The
product of horizontal conductances on these crossings is the $B$ variable for that cell.
For the reconstruction example in Section~\ref{sec:reconstruction}, the $B$ variables
given by this construction are shown in Figure~\ref{fig:b-variables}.

\begin{proposition}
  With the construction in Figure~\ref{dyck-tiling-B},
  which is further illustrated in Figure~\ref{fig:b-variables},
  the $B$ variables satisfy the biratio formula \eqref{BBBB}.
\end{proposition}

\begin{proof}
  Consider the contribution of a particular crossing $e$ to the
  biratio of $B$ variables at any other crossing $e'$.  If $e'$ is in
  the interior of the shadow of $e$, then all four $B$ variables have
  a factor of $c_e$.  If $e'$ is in the interior of the complement of
  the shadow, none of the four $B$ variables have a factor of $c_e$.
  If $e'$ is on the boundary of the shadow but not at $e$, then two
  adjacent $B$ variables have the weight $c_{e}$ and this gets divided
  out in the biratio.  The only net contribution comes when $e'=e$.
\end{proof}

One advantageous property of the $B$ variables is that they transform in a simple manner under
Y-$\Delta$ moves: via the cube recurrence \cite{CS,GK}, see Figure~\ref{BYD}.

\section{Positivity and Laurent phenomenon}

The top-dimensional cell of $\Omega_n$ consists of networks which are called \textbf{well-connected}. There are a number of
equivalent definitions of well-connected networks, due to \cite{CdV}.
A network~$\G$ is \textit{well-connected\/} if for any pair of noninterlaced subsets $A,B\subset \No$ of the same
cardinality $k$, there is a pairwise vertex-disjoint set of $k$ paths in $\G$ connecting
in pairs the nodes in $A$ to those in~$B$.  A minimal well-connected network has exactly
$n(n-1)/2$ edges; it has strand matching $\pi = \{\{1,n+1\},\{2,n+2\},\dots,\{n,2n\}\}$. A
network is well-connected if and only if all noninterlaced minors of $L$ are strictly positive.

For well-connected networks, we show here that one can test positivity
using $\binom{n}{2}$ minors of $L$, rather than Pfaffians. We
conjecture that an analogous statement holds for general minimal
networks:

\begin{conjecture}
  For a minimal network on $m$ edges, one can test positivity using $m$ noninterlaced minors of $L$.
\end{conjecture}

\subsection{Contiguous and central minors}

We define a \textbf{contiguous minor} of an $n\times n$ matrix $M$ to be a minor of the form
\[\CM_{a,b,y}(M)=\det M_{a,a+1,\dots,a+y-1}^{b+y-1,\dots,b+1,b}\]
where the indices are interpreted modulo $n$.
In other words, the row and column indices are both contiguous modulo~$n$.
We define the \textbf{central (contiguous) minor} $\CM_{x,y}(M)$ (we use one less subscript)
to be the contiguous minor $\CM_{a,b,y}(M)$ with
\begin{align*}
a &= \left\lfloor\frac{x-y}{2}\right\rfloor&
b &= \left\lfloor\frac{x-y+n-(n-1\bmod 2)}{2}\right\rfloor\,.
\end{align*}
In other words, the central minor $\CM_{x,y}$ is a $y\times y$ minor
of~$M$ whose row and column indices are (cyclically) contiguous, and
when the matrix indices are arranged on a circle, the chords
connecting the row and column indices are generally as central as
possible, modulo some details about rounding.  Increasing $x$ by~$2$
is equivalent to cyclically shifting the indices, so $x$ is naturally
interpreted modulo $2n$.  The parameter $y$ naturally ranges
from~$0$ to~$n$.
Central minors were implicitly defined within a proof in \cite[Chapt.~5.3]{curtis-morrow}.

\renewcommand{\minorat}[3]{\begin{scope}[shift={#3},scale=0.30]
    \foreach \x in {1,...,#1} { \coordinate (\x) at ({1.*cos(\x*360/#1)},{1.*sin(\x*360/#1)});}
    \draw[fill=yellow,draw=none] (0,0) circle(1.0);
    \draw [thick] \foreach \x/\y in {#2} {(\x)--(\y)};
    \foreach \x in {1,...,#1} {\node [circle,fill=orange!60!yellow,inner sep=.5pt] at (\x) {\scalebox{1}{$\scriptstyle\x$}};}
    \foreach \x/\y in {#2} {
      \node [circle,fill=green!50!yellow,inner sep=.5pt,draw] at (\x) {\scalebox{1}{$\scriptstyle\x$}};
      \node [circle,fill=red!0!white,inner sep=.5pt,draw] at (\y) {\scalebox{1}{$\scriptstyle\y$}};
     };
\end{scope}}

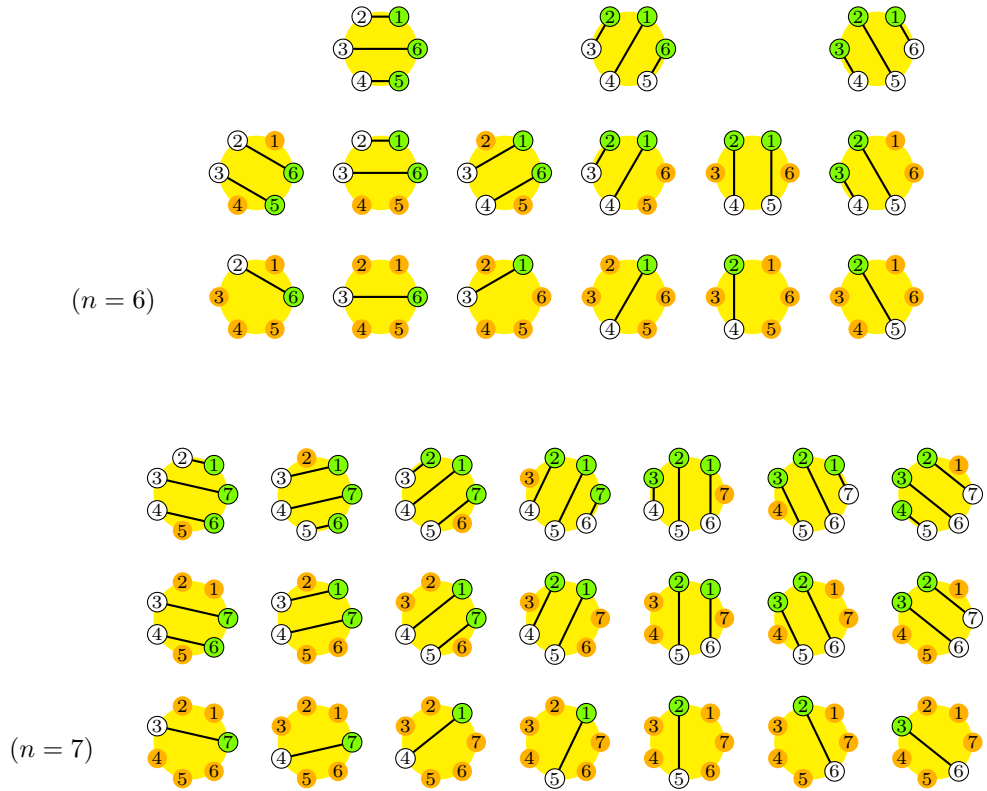
\begin{figure}[b!]
\begin{center}
\raisebox{12pt}{($n=6$)}\quad\quad
\begin{tikzpicture}
\begin{scope}[scale=1.65,baseline=1cm]
\minorat{6}{6/2}{(1,1)}
\minorat{6}{5/3,6/2}{(1,2)}
\minorat{6}{6/3}{(2,1)}
\minorat{6}{6/3,1/2}{(2,2)}
\minorat{6}{5/4,6/3,1/2}{(2,3)}
\minorat{6}{1/3}{(3,1)}
\minorat{6}{6/4,1/3}{(3,2)}
\minorat{6}{1/4}{(4,1)}
\minorat{6}{1/4,2/3}{(4,2)}
\minorat{6}{6/5,1/4,2/3}{(4,3)}
\minorat{6}{2/4}{(5,1)}
\minorat{6}{1/5,2/4}{(5,2)}
\minorat{6}{2/5}{(6,1)}
\minorat{6}{2/5,3/4}{(6,2)}
\minorat{6}{1/6,2/5,3/4}{(6,3)}
\end{scope}
\end{tikzpicture}
\\[40pt]
\raisebox{12pt}{($n=7$)}\quad\quad
\begin{tikzpicture}
\begin{scope}[scale=1.65,baseline=1cm]
\minorat{7}{7/3}{(1,1)}
\minorat{7}{6/4,7/3}{(1,2)}
\minorat{7}{6/4,7/3,1/2}{(1,3)}
\minorat{7}{7/4}{(2,1)}
\minorat{7}{7/4,1/3}{(2,2)}
\minorat{7}{6/5,7/4,1/3}{(2,3)}
\minorat{7}{1/4}{(3,1)}
\minorat{7}{7/5,1/4}{(3,2)}
\minorat{7}{7/5,1/4,2/3}{(3,3)}
\minorat{7}{1/5}{(4,1)}
\minorat{7}{1/5,2/4}{(4,2)}
\minorat{7}{7/6,1/5,2/4}{(4,3)}
\minorat{7}{2/5}{(5,1)}
\minorat{7}{1/6,2/5}{(5,2)}
\minorat{7}{1/6,2/5,3/4}{(5,3)}
\minorat{7}{2/6}{(6,1)}
\minorat{7}{2/6,3/5}{(6,2)}
\minorat{7}{1/7,2/6,3/5}{(6,3)}
\minorat{7}{3/6}{(7,1)}
\minorat{7}{2/7,3/6}{(7,2)}
\minorat{7}{2/7,3/6,4/5}{(7,3)}
\end{scope}
\end{tikzpicture}
\end{center}
\caption{\label{central-minors}
The $15$ small central minors of a (symmetric) $6\times 6$ matrix,
and the $21$ small central minors of a (symmetric) $7\times7$ matrix.
The row indices of the minor are shown in green, the column indices are shown in white.
The columns are indexed left-to-right $x=1,\dots,n$.
For non-symmetric matrices there are twice as many small central minors, and the range of $x$ is $1,\dots,2n$.
}
\end{figure}

\renewcommand{\minorat}[3]{\begin{scope}[shift={#3},scale=0.35]
    \foreach \x in {1,...,#1} { \coordinate (\x) at ({1.*cos(\x*360/#1)},{1.*sin(\x*360/#1)});}
    \draw[fill=yellow,draw=none] (0,0) circle(1.0);
    \draw [thick] \foreach \x/\y in {#2} {(\x)--(\y)};
    \foreach \x in {1,...,#1} {\node [circle,fill=orange!60!yellow,inner sep=.2pt] at (\x) {\scalebox{0.7}{$\scriptstyle\x$}};}
    \foreach \x/\y in {#2} {
      \node [circle,fill=green!50!yellow,inner sep=.2pt,draw] at (\x) {\scalebox{0.7}{$\scriptstyle\x$}};
      \node [circle,fill=red!0!white,inner sep=.2pt,draw] at (\y) {\scalebox{0.7}{$\scriptstyle\y$}};
     };
\end{scope}}

If the matrix $M$ is symmetric, then $\CM_{x+n,y}=\CM_{x,y}$ when $n$
is odd, and also when $n$ is even and $x+y$ is odd.  If $n$ is even
and $x+y$ is even, then $\CM_{x,y}$ is off-center by one and
$\CM_{x+n,y}$ is off-center by one in the other direction, so they are
different, but for some purposes they will turn out to be essentially
interchangeable.

The row and column indices of $\CM_{x,y}(M)$ are disjoint when either $y<n/2$ or $y=n/2$ and $x+y$ is odd.

We define the \textbf{small central minors} of an $n\times n$ matrix $M$ to be
\begin{equation}
 \Big\{\CM_{x,y}(M)\Big\}_{1\leq x\leq 2n}^{1\leq y<n/2\text{ or }y=n/2\text{ and }x+y\text{ odd}}\,.
\end{equation}
For symmetric matrices,
we define the small central minors as above
but with $1\leq x\leq n$.
See Figure~\ref{central-minors} for illustrations
of the small central minors when $n=6$ and $n=7$.

There are $\binom{n}{2}$ small central minors of a symmetric $n\times n$ matrix, whether $n$ is even or
odd.  The number of edges in a minimal well-connected network is also $\binom{n}{2}$.
We shall see that the small central minors of the response matrix are
all positive precisely when the network is well connected (Corollary~\ref{centralminorcor} below),
and that they can be used to reconstruct the response matrix
(Corollary~\ref{minor-Laurent}),
from which the network can be reconstructed as described in Section~\ref{sec:standard}.

We will show that a general contiguous minor $\CM_{a,b,y}(M)$ can be
expressed as a Laurent polynomial in the central minors, where the
terms in this Laurent polynomial are in bijective correspondence with
domino tilings of a certain region.

We let $\AD_{x_0,y_0,\ell}$ denote the ``Aztec diamond
region of order $\ell$'' centered at $(x_0,y_0)\in\Z^2$, which
consists of those unit squares of the square lattice centered at points
$(x,y)$ for which $x,y\in\Z+\frac12$ and $|x-x_0|+|y-y_0|\leq\ell$.
There are some examples of Aztec diamonds illustrated in the next few
pages.
There are $2^{\ell(\ell+1)/2}$ ways
to tile an order-$\ell$ Aztec diamond by $2\times 1$ dominos
\cite{EKLP1}.

We give weights to the dominos:
If $(x,y)\in\Z\times(\Z+\frac12)$, then $w[x,y]$ is the weight of the horizontal domino containing the squares at $(x+\frac12,y)$ and $(x-\frac12,y)$, and if $(x,y)\in(\Z+\frac12)\times\Z$, then
$w[x,y]$ is the weight of the vertical domino containing squares $(x,y+\frac12)$ and $(x,y-\frac12)$.
The weight of a domino tiling is the product
of the weights of its dominos.
Given a region $R$ tileable by dominos,
we define its weight $W(R)$ to be the sum of the weights of its domino tilings.
This weighted sum of domino tilings may be computed using Kasteleyn matrices,
or by other methods in the case of Aztec diamonds.

To define the Laurent polynomial of a region,
we set the domino weights to be expressions defined
in terms of variables $\{v_{x,y}\}_{x,y\in\Z}$.
For horizontal dominos, $(x,y)\in\Z\times(\Z+\frac12)$,
we set 
\begin{equation}\label{h-domino-wt}
w[x,y]=1/(v_{x,y-1/2} v_{x,y+1/2})\,,\end{equation}
and for vertical dominos, $(x,y)\in(\Z+\frac12)\times\Z$, we set
\begin{equation}\label{v-domino-wt}
w[x,y]=1/(v_{x-1/2,y} v_{x+1/2,y})\,.
\end{equation}
We define the Laurent polynomial $P(R)$ of a region $R$ to be $W(R)$,
with these domino weights, times a monomial factor
 which is
(usually) the product of the variables $v_{x,y}$
on which the domino weights in $R$ depend.

For the Aztec diamond $\AD_{x_0,y_0,\ell}$, the monomial
factor is \[\prod_{\substack{x,y\in\Z\\|x-x_0|+|y-y_0|\leq\ell}}
v_{x,y}\,.\]
For the Aztec diamond $\AD_{x_0,y_0,0}$, the monomial factor
is $v_{x_0,y_0}$ even though the weight $W(\AD_{x_0,y_0,0})=1$ does not
actually depend on $v_{x_0,y_0}$.  More precisely, the monomial factor is
a product of variables that include the variables on which $W(R)$ depends,
and which will be clear from context.  For each monomial in the Laurent
polynomial $P(R)$ of the region, the degree of any variable
$v_{x,y}$ is one of $0,-1,+1$.

Let the \textit{truncated Aztec diamond\/} $\TAD_{x_0,y_0,\ell,n}$ be Aztec diamond $\AD_{x_0,y_0,\ell}$
truncated to contain only squares centered at points with $y$-coordinate between $0$ and $n$,
and depend only on variables $v_{x,y}$ for which $|x-x_0|+|y-y_0|\leq\ell$ and $0\leq y\leq n$.

\begin{theorem}\label{domino}
Suppose $0\leq y\leq n$ and $\ell\geq 0$.
Consider the Laurent polynomial $P(\TAD_{x,y,\ell,n})$ evaluated
with each variable $v_{x',y'}$ set to the central minor 
$v_{x',y'}=\CM_{x',y'}(M)$ of an $n\times n$ matrix $M$.
The rightmost central minor of the truncated Aztec diamond
has row-set (or column-set respectively) which is different
from the corresponding set of indices of the central minor to its left,
and the leftmost central minor has column-set (or row-set respectively)
which is different from the one to its right.
This evaluation of $P(\TAD_{x,y,\ell,n})$ gives the contiguous minor of $M$
with row-set given by the row-set of the rightmost (respectively leftmost) central minor, and
column-set given by the column-set of the leftmost (respectively rightmost) central minor.
\end{theorem}

The following example illustrates most of the key ideas in Theorem~\ref{domino}.
Here the Aztec diamond has order-$1$, and the rightmost central minor gives the row-set and the leftmost central minor gives the column set in the contiguous minor.  There are $2$ domino tilings of the region, so the contiguous minor is a $2$-term Laurent polynomial in the central minors.

\begin{align*}
&\det\begin{bmatrix} M_{1,3}\end{bmatrix} =
\minors{7}{1/3} =
\begin{tikzpicture}[baseline=1cm-2.5pt]
\filldraw[blue!15!white,draw=black] (0,0) rectangle (1,1);
\filldraw[blue!15!white,draw=black] (-1,0) rectangle (0,1);
\filldraw[blue!15!white,draw=black] (-1,1) rectangle (0,2);
\filldraw[blue!15!white,draw=black] (0,1) rectangle (1,2);
\minorat{7}{7/4}{(0,1)}
\minorat{7}{7/3}{(-1,1)}
\minorat{7}{1/4}{(1,1)}
\minorat{7}{1/3,7/4}{(0,2)}
\minorat{7}{}{(0,0)}
\end{tikzpicture}
=
\begin{tikzpicture}[baseline=1cm-2.5pt]
\draw[thick] (0,2) -- (0,2.2);
\draw[thick] (0,0) -- (0,-.2);
\draw[thick] (1,1) -- (1.2,1);
\draw[thick] (-1,1) -- (-1.2,1);
\draw[thick] (1,2) -- (1.1414,2.1414);
\draw[thick] (1,0) -- (1.1414,-0.1414);
\draw[thick] (-1,2) -- (-1.1414,2.1414);
\draw[thick] (-1,0) -- (-1.1414,-0.1414);
\filldraw[blue!15!white,draw=black,thick] (-1,0) rectangle (1,1);
\filldraw[blue!15!white,draw=black,thick] (-1,1) rectangle (1,2);
\end{tikzpicture}
+
\begin{tikzpicture}[baseline=1cm-2pt]
\draw[thick] (0,2) -- (0,2.2);
\draw[thick] (0,0) -- (0,-.2);
\draw[thick] (1,1) -- (1.2,1);
\draw[thick] (-1,1) -- (-1.2,1);
\draw[thick] (1,2) -- (1.1414,2.1414);
\draw[thick] (1,0) -- (1.1414,-0.1414);
\draw[thick] (-1,2) -- (-1.1414,2.1414);
\draw[thick] (-1,0) -- (-1.1414,-0.1414);
\filldraw[blue!15!white,draw=black,thick] (-1,0) rectangle (0,2);
\filldraw[blue!15!white,draw=black,thick] (0,0) rectangle (1,2);
\end{tikzpicture}
\\
&=
\frac{\minors{7}{7/3}\minors{7}{1/4}}{\minors{7}{7/4}}
+
\frac{\minors{7}{1/3,7/4}\minors{7}{}}{\minors{7}{7/4}}
=
\frac{\det\!\begin{bmatrix} M_{7,3}\end{bmatrix}\det\!\begin{bmatrix} M_{1,4}\end{bmatrix}}{\det\!\begin{bmatrix} M_{7,4}\end{bmatrix}}
+
\frac{\det\!\begin{bmatrix} M_{7,4} & M_{7,3} \\ M_{1,4} & M_{1,3} \end{bmatrix}\overbrace{\det[]}^{=1}}{\det\!\begin{bmatrix} M_{7,4}\end{bmatrix}}
\end{align*}

In the following example the (order-$2$) Aztec diamond is truncated,
the rightmost central minor contributes the column-set
and the leftmost central minor contributes the row-set of the contiguous minor.
Here there are $6$ domino tilings, so the Laurent polynomial has $6$ terms,
which we do not list explicitly.
\[
\minors{7}{6/5} =
\begin{tikzpicture}[baseline=1.5cm-2.5pt]
\filldraw[blue!15!white,draw=black] (9,0) rectangle (10,1);
\filldraw[blue!15!white,draw=black] (9,1) rectangle (10,2);
\filldraw[blue!15!white,draw=black] (10,0) rectangle (11,1);
\filldraw[blue!15!white,draw=black] (10,1) rectangle (11,2);
\filldraw[blue!15!white,draw=black] (10,2) rectangle (11,3);
\filldraw[blue!15!white,draw=black] (11,0) rectangle (12,1);
\filldraw[blue!15!white,draw=black] (11,1) rectangle (12,2);
\filldraw[blue!15!white,draw=black] (11,2) rectangle (12,3);
\filldraw[blue!15!white,draw=black] (12,0) rectangle (13,1);
\filldraw[blue!15!white,draw=black] (12,1) rectangle (13,2);
\minorat{7}{6/3}{(9,1)}
\minorat{7}{}{(10,0)}
\minorat{7}{7/3}{(10,1)}
\minorat{7}{7/3,6/4}{(10,2)}
\minorat{7}{}{(11,0)}
\minorat{7}{7/4}{(11,1)}
\minorat{7}{1/3,7/4}{(11,2)}
\minorat{7}{1/3,7/4,6/5}{(11,3)}
\minorat{7}{}{(12,0)}
\minorat{7}{1/4}{(12,1)}
\minorat{7}{1/4,7/5}{(12,2)}
\minorat{7}{1/5}{(13,1)}
\end{tikzpicture}
\]

An example which better shows the Aztec diamond structure of the region is
\[
\minorsize{13}{3/4,2/5,1/6}{0.65} =
\begin{tikzpicture}[baseline=3cm-2.5pt]
\filldraw[blue!15!white,draw=black] (17,2) rectangle (18,3);
\filldraw[blue!15!white,draw=black] (17,3) rectangle (18,4);
\filldraw[blue!15!white,draw=black] (18,1) rectangle (19,2);
\filldraw[blue!15!white,draw=black] (18,2) rectangle (19,3);
\filldraw[blue!15!white,draw=black] (18,3) rectangle (19,4);
\filldraw[blue!15!white,draw=black] (18,4) rectangle (19,5);
\filldraw[blue!15!white,draw=black] (19,0) rectangle (20,1);
\filldraw[blue!15!white,draw=black] (19,1) rectangle (20,2);
\filldraw[blue!15!white,draw=black] (19,2) rectangle (20,3);
\filldraw[blue!15!white,draw=black] (19,3) rectangle (20,4);
\filldraw[blue!15!white,draw=black] (19,4) rectangle (20,5);
\filldraw[blue!15!white,draw=black] (19,5) rectangle (20,6);
\filldraw[blue!15!white,draw=black] (20,0) rectangle (21,1);
\filldraw[blue!15!white,draw=black] (20,1) rectangle (21,2);
\filldraw[blue!15!white,draw=black] (20,2) rectangle (21,3);
\filldraw[blue!15!white,draw=black] (20,3) rectangle (21,4);
\filldraw[blue!15!white,draw=black] (20,4) rectangle (21,5);
\filldraw[blue!15!white,draw=black] (20,5) rectangle (21,6);
\filldraw[blue!15!white,draw=black] (21,1) rectangle (22,2);
\filldraw[blue!15!white,draw=black] (21,2) rectangle (22,3);
\filldraw[blue!15!white,draw=black] (21,3) rectangle (22,4);
\filldraw[blue!15!white,draw=black] (21,4) rectangle (22,5);
\filldraw[blue!15!white,draw=black] (22,2) rectangle (23,3);
\filldraw[blue!15!white,draw=black] (22,3) rectangle (23,4);
\minorat{13}{13/4,12/5,11/6}{(17,3)}
\minorat{13}{13/5,12/6}{(18,2)}
\minorat{13}{13/5,12/6,11/7}{(18,3)}
\minorat{13}{1/4,13/5,12/6,11/7}{(18,4)}
\minorat{13}{13/6}{(19,1)}
\minorat{13}{13/6,12/7}{(19,2)}
\minorat{13}{1/5,13/6,12/7}{(19,3)}
\minorat{13}{1/5,13/6,12/7,11/8}{(19,4)}
\minorat{13}{2/4,1/5,13/6,12/7,11/8}{(19,5)}
\minorat{13}{}{(20,0)}
\minorat{13}{13/7}{(20,1)}
\minorat{13}{1/6,13/7}{(20,2)}
\minorat{13}{1/6,13/7,12/8}{(20,3)}
\minorat{13}{2/5,1/6,13/7,12/8}{(20,4)}
\minorat{13}{2/5,1/6,13/7,12/8,11/9}{(20,5)}
\minorat{13}{3/4,2/5,1/6,13/7,12/8,11/9}{(20,6)}
\minorat{13}{1/7}{(21,1)}
\minorat{13}{1/7,13/8}{(21,2)}
\minorat{13}{2/6,1/7,13/8}{(21,3)}
\minorat{13}{2/6,1/7,13/8,12/9}{(21,4)}
\minorat{13}{3/5,2/6,1/7,13/8,12/9}{(21,5)}
\minorat{13}{2/7,1/8}{(22,2)}
\minorat{13}{2/7,1/8,13/9}{(22,3)}
\minorat{13}{3/6,2/7,1/8,13/9}{(22,4)}
\minorat{13}{3/7,2/8,1/9}{(23,3)}
\end{tikzpicture}
\]

\begin{proof}[Proof of Theorem~\ref{domino}]
  We use the graphical condensation of Kuo \cite[Thm.~5.5]{Kuo} together
  with the Desnanot--Jacobi identity, which are both related to Dodgson
  condensation.

  For a matrix $M$, let $M_{\widehat{r_1,\dots,r_k}}^{\widehat{c_1,\dots,c_k}}$ denote the submatrix obtained
  by deleting rows $r_1,\dots,r_k$ and columns $c_1,\dots,c_k$.  The Desnanot--Jacobi identity
  is
\[ \det M_{\widehat{a}}^{\widehat{c}} \det M_{\widehat{b}}^{\widehat{d}} =
 \det M \det M_{\widehat{a,b}}^{\widehat{c,d}} +
 \det M_{\widehat{b}}^{\widehat{c}} \det M_{\widehat{a}}^{\widehat{d}}\,,\]
where row indices $a$ and $b$ are in sorted order, and column indices $c$ and $d$ are in sorted order.
For example,
\[
\minors{9}{8/1,7/2}  =
\frac{\minors{9}{8/1,7/2,6/3} \minors{9}{7/2}
  +
\minors{9}{7/1,6/2} \minors{9}{8/2,7/3}}{\minors{9}{7/2,6/3}}\,.
\]

Kuo's graphical condensation relates the weighted sum of domino tilings of
weighted Aztec diamonds.
Recall that $w$ denotes the domino weights.
Kuo's formula is
\begin{multline*}
  W(\AD_{x,y,\ell}) W(\AD_{x,y,\ell-2}) = \\
  w[x{-}\ell{+}\tfrac12,y]\,w[x{+}\ell{-}\tfrac12,y]\, \times
  W(\AD_{x,y+1,\ell-1})W(\AD_{x,y-1,\ell-1})\\
 +
  w[x,y{-}\ell{+}\tfrac12]\,w[x,y{+}\ell{-}\tfrac12]\, \times
  W(\AD_{x-1,y,\ell-1})W(\AD_{x+1,y,\ell-1})\,,
\end{multline*}
for any domino weights (not just the ones given by \eqref{h-domino-wt} and \eqref{v-domino-wt}).
Notice that this formula extends to truncated Aztec diamonds, simply by setting
the vertical domino weights to $0$ and horizontal domino weights to
$1$ for dominos whose $y$-coordinate is too low or too high.

When the weights are given by \eqref{h-domino-wt} and \eqref{v-domino-wt},
if we multiply both sides by the monomial factors of $\TAD_{x,y,\ell,n}$ and $\TAD_{x,y,\ell-2,n}$,
the individual domino weights drop out, and we obtain
\begin{multline*}
  P(\TAD_{x,y,\ell,n}) P(\TAD_{x,y,\ell-2,n}) =
  P(\TAD_{x,y+1,\ell-1,n}) P(\TAD_{x,y-1,\ell-1,n})\\
 +
  P(\TAD_{x-1,y,\ell-1,n})P(\TAD_{x+1,y,\ell-1,n})\,.
\end{multline*}
Here is an example (with the $v$'s set to the central minors):
\newcommand{\exsc}{0.395}
\[
\begin{tikzpicture}[scale=\exsc,baseline=0.98cm-2.5pt]
\filldraw[blue!15!white,draw=black] (11,1) rectangle (12,2);
\filldraw[blue!15!white,draw=black] (11,2) rectangle (12,3);
\filldraw[blue!15!white,draw=black] (12,0) rectangle (13,1);
\filldraw[blue!15!white,draw=black] (12,1) rectangle (13,2);
\filldraw[blue!15!white,draw=black] (12,2) rectangle (13,3);
\filldraw[blue!15!white,draw=black] (12,3) rectangle (13,4);
\filldraw[blue!15!white,draw=black] (13,0) rectangle (14,1);
\filldraw[blue!15!white,draw=black] (13,1) rectangle (14,2);
\filldraw[blue!15!white,draw=black] (13,2) rectangle (14,3);
\filldraw[blue!15!white,draw=black] (13,3) rectangle (14,4);
\filldraw[blue!15!white,draw=black] (13,4) rectangle (14,5);
\filldraw[blue!15!white,draw=black] (14,0) rectangle (15,1);
\filldraw[blue!15!white,draw=black] (14,1) rectangle (15,2);
\filldraw[blue!15!white,draw=black] (14,2) rectangle (15,3);
\filldraw[blue!15!white,draw=black] (14,3) rectangle (15,4);
\filldraw[blue!15!white,draw=black] (14,4) rectangle (15,5);
\filldraw[blue!15!white,draw=black] (15,0) rectangle (16,1);
\filldraw[blue!15!white,draw=black] (15,1) rectangle (16,2);
\filldraw[blue!15!white,draw=black] (15,2) rectangle (16,3);
\filldraw[blue!15!white,draw=black] (15,3) rectangle (16,4);
\filldraw[blue!15!white,draw=black] (16,1) rectangle (17,2);
\filldraw[blue!15!white,draw=black] (16,2) rectangle (17,3);
\minorats{11}{9/2,8/3}{(11,2)}
\minorats{11}{9/3}{(12,1)}
\minorats{11}{9/3,8/4}{(12,2)}
\minorats{11}{10/2,9/3,8/4}{(12,3)}
\minorats{11}{}{(13,0)}
\minorats{11}{9/4}{(13,1)}
\minorats{11}{10/3,9/4}{(13,2)}
\minorats{11}{10/3,9/4,8/5}{(13,3)}
\minorats{11}{11/2,10/3,9/4,8/5}{(13,4)}
\minorats{11}{}{(14,0)}
\minorats{11}{10/4}{(14,1)}
\minorats{11}{10/4,9/5}{(14,2)}
\minorats{11}{11/3,10/4,9/5}{(14,3)}
\minorats{11}{11/3,10/4,9/5,8/6}{(14,4)}
\minorats{11}{1/2,11/3,10/4,9/5,8/6}{(14,5)}
\minorats{11}{}{(15,0)}
\minorats{11}{10/5}{(15,1)}
\minorats{11}{11/4,10/5}{(15,2)}
\minorats{11}{11/4,10/5,9/6}{(15,3)}
\minorats{11}{1/3,11/4,10/5,9/6}{(15,4)}
\minorats{11}{11/5}{(16,1)}
\minorats{11}{11/5,10/6}{(16,2)}
\minorats{11}{1/4,11/5,10/6}{(16,3)}
\minorats{11}{1/5,11/6}{(17,2)}
\end{tikzpicture}
\times
\begin{tikzpicture}[scale=\exsc,baseline=0.98cm-2.5pt]
\filldraw[blue!15!white,draw=black] (13,1) rectangle (14,2);
\filldraw[blue!15!white,draw=black] (13,2) rectangle (14,3);
\filldraw[blue!15!white,draw=black] (14,1) rectangle (15,2);
\filldraw[blue!15!white,draw=black] (14,2) rectangle (15,3);
\minorats{11}{10/3,9/4}{(13,2)}
\minorats{11}{10/4}{(14,1)}
\minorats{11}{10/4,9/5}{(14,2)}
\minorats{11}{11/3,10/4,9/5}{(14,3)}
\minorats{11}{11/4,10/5}{(15,2)}
\end{tikzpicture}
=
\begin{tikzpicture}[scale=\exsc,baseline=0.98cm-2.5pt]
\filldraw[blue!15!white,draw=black] (12,2) rectangle (13,3);
\filldraw[blue!15!white,draw=black] (12,3) rectangle (13,4);
\filldraw[blue!15!white,draw=black] (13,1) rectangle (14,2);
\filldraw[blue!15!white,draw=black] (13,2) rectangle (14,3);
\filldraw[blue!15!white,draw=black] (13,3) rectangle (14,4);
\filldraw[blue!15!white,draw=black] (13,4) rectangle (14,5);
\filldraw[blue!15!white,draw=black] (14,1) rectangle (15,2);
\filldraw[blue!15!white,draw=black] (14,2) rectangle (15,3);
\filldraw[blue!15!white,draw=black] (14,3) rectangle (15,4);
\filldraw[blue!15!white,draw=black] (14,4) rectangle (15,5);
\filldraw[blue!15!white,draw=black] (15,2) rectangle (16,3);
\filldraw[blue!15!white,draw=black] (15,3) rectangle (16,4);
\minorats{11}{10/2,9/3,8/4}{(12,3)}
\minorats{11}{10/3,9/4}{(13,2)}
\minorats{11}{10/3,9/4,8/5}{(13,3)}
\minorats{11}{11/2,10/3,9/4,8/5}{(13,4)}
\minorats{11}{10/4}{(14,1)}
\minorats{11}{10/4,9/5}{(14,2)}
\minorats{11}{11/3,10/4,9/5}{(14,3)}
\minorats{11}{11/3,10/4,9/5,8/6}{(14,4)}
\minorats{11}{1/2,11/3,10/4,9/5,8/6}{(14,5)}
\minorats{11}{11/4,10/5}{(15,2)}
\minorats{11}{11/4,10/5,9/6}{(15,3)}
\minorats{11}{1/3,11/4,10/5,9/6}{(15,4)}
\minorats{11}{1/4,11/5,10/6}{(16,3)}
\end{tikzpicture}
\times
\begin{tikzpicture}[scale=\exsc,baseline=0.98cm-2.5pt]
\filldraw[blue!15!white,draw=black] (12,0) rectangle (13,1);
\filldraw[blue!15!white,draw=black] (12,1) rectangle (13,2);
\filldraw[blue!15!white,draw=black] (13,0) rectangle (14,1);
\filldraw[blue!15!white,draw=black] (13,1) rectangle (14,2);
\filldraw[blue!15!white,draw=black] (13,2) rectangle (14,3);
\filldraw[blue!15!white,draw=black] (14,0) rectangle (15,1);
\filldraw[blue!15!white,draw=black] (14,1) rectangle (15,2);
\filldraw[blue!15!white,draw=black] (14,2) rectangle (15,3);
\filldraw[blue!15!white,draw=black] (15,0) rectangle (16,1);
\filldraw[blue!15!white,draw=black] (15,1) rectangle (16,2);
\minorats{11}{9/3}{(12,1)}
\minorats{11}{}{(13,0)}
\minorats{11}{9/4}{(13,1)}
\minorats{11}{10/3,9/4}{(13,2)}
\minorats{11}{}{(14,0)}
\minorats{11}{10/4}{(14,1)}
\minorats{11}{10/4,9/5}{(14,2)}
\minorats{11}{11/3,10/4,9/5}{(14,3)}
\minorats{11}{}{(15,0)}
\minorats{11}{10/5}{(15,1)}
\minorats{11}{11/4,10/5}{(15,2)}
\minorats{11}{11/5}{(16,1)}
\end{tikzpicture}
+
\begin{tikzpicture}[scale=\exsc,baseline=0.98cm-2.5pt]
\filldraw[blue!15!white,draw=black] (11,1) rectangle (12,2);
\filldraw[blue!15!white,draw=black] (11,2) rectangle (12,3);
\filldraw[blue!15!white,draw=black] (12,0) rectangle (13,1);
\filldraw[blue!15!white,draw=black] (12,1) rectangle (13,2);
\filldraw[blue!15!white,draw=black] (12,2) rectangle (13,3);
\filldraw[blue!15!white,draw=black] (12,3) rectangle (13,4);
\filldraw[blue!15!white,draw=black] (13,0) rectangle (14,1);
\filldraw[blue!15!white,draw=black] (13,1) rectangle (14,2);
\filldraw[blue!15!white,draw=black] (13,2) rectangle (14,3);
\filldraw[blue!15!white,draw=black] (13,3) rectangle (14,4);
\filldraw[blue!15!white,draw=black] (14,1) rectangle (15,2);
\filldraw[blue!15!white,draw=black] (14,2) rectangle (15,3);
\minorats{11}{9/2,8/3}{(11,2)}
\minorats{11}{9/3}{(12,1)}
\minorats{11}{9/3,8/4}{(12,2)}
\minorats{11}{10/2,9/3,8/4}{(12,3)}
\minorats{11}{}{(13,0)}
\minorats{11}{9/4}{(13,1)}
\minorats{11}{10/3,9/4}{(13,2)}
\minorats{11}{10/3,9/4,8/5}{(13,3)}
\minorats{11}{11/2,10/3,9/4,8/5}{(13,4)}
\minorats{11}{10/4}{(14,1)}
\minorats{11}{10/4,9/5}{(14,2)}
\minorats{11}{11/3,10/4,9/5}{(14,3)}
\minorats{11}{11/4,10/5}{(15,2)}
\end{tikzpicture}
\times
\begin{tikzpicture}[scale=\exsc,baseline=0.98cm-2.5pt]
\filldraw[blue!15!white,draw=black] (13,1) rectangle (14,2);
\filldraw[blue!15!white,draw=black] (13,2) rectangle (14,3);
\filldraw[blue!15!white,draw=black] (14,0) rectangle (15,1);
\filldraw[blue!15!white,draw=black] (14,1) rectangle (15,2);
\filldraw[blue!15!white,draw=black] (14,2) rectangle (15,3);
\filldraw[blue!15!white,draw=black] (14,3) rectangle (15,4);
\filldraw[blue!15!white,draw=black] (15,0) rectangle (16,1);
\filldraw[blue!15!white,draw=black] (15,1) rectangle (16,2);
\filldraw[blue!15!white,draw=black] (15,2) rectangle (16,3);
\filldraw[blue!15!white,draw=black] (15,3) rectangle (16,4);
\filldraw[blue!15!white,draw=black] (16,1) rectangle (17,2);
\filldraw[blue!15!white,draw=black] (16,2) rectangle (17,3);
\minorats{11}{10/3,9/4}{(13,2)}
\minorats{11}{10/4}{(14,1)}
\minorats{11}{10/4,9/5}{(14,2)}
\minorats{11}{11/3,10/4,9/5}{(14,3)}
\minorats{11}{}{(15,0)}
\minorats{11}{10/5}{(15,1)}
\minorats{11}{11/4,10/5}{(15,2)}
\minorats{11}{11/4,10/5,9/6}{(15,3)}
\minorats{11}{1/3,11/4,10/5,9/6}{(15,4)}
\minorats{11}{11/5}{(16,1)}
\minorats{11}{11/5,10/6}{(16,2)}
\minorats{11}{1/4,11/5,10/6}{(16,3)}
\minorats{11}{1/5,11/6}{(17,2)}
\end{tikzpicture}
\]

We prove the theorem by induction on $\ell$.  The case $\ell=0$ is a tautology.  The case $\ell=1$
is straightforward to verify; this was the first example given after the theorem statement.  For $\ell\geq 2$
we observe that the contiguous minors and the truncated Aztec diamond Laurent polynomials satisfy the same recurrence.
\end{proof}

\begin{corollary}\label{entry-Laurent}
  Any matrix entry $M_{i,j}$ of an $n\times n$ matrix $M$ can be
  expressed as a positive-coefficient Laurent polynomial in the central minors of $M$.
\end{corollary}

\begin{proof}
  We use Theorem~\ref{domino} to express $M_{i,j}=P(\TAD_{x,y,\ell,n})$
  for $y=1$ and a suitable choice of $x$, and $\ell$.
\end{proof}

\begin{corollary}\label{contiguous-Laurent}
  For an $n\times n$ matrix $M$, any contiguous minor with disjoint
  row and column indices is a positive-coefficient Laurent polynomial
  in the small central minors of $M$.
\end{corollary}

\begin{proof}
  Let $R$ and $S$ denote the row and column indices of the contiguous minor.
  We use Theorem~\ref{domino} to express $\det M_R^S = P(\TAD_{x,y,\ell,n})$
  for $y=|R|$ and a suitable choice of $x$, and $\ell$.
  We can take the order $\ell$ of the truncated Aztec diamond to be at most $\lfloor n/2\rfloor-y$, with the extreme case being when $R$ and $S$ are adjacent.
The top of the truncated Aztec diamond has height at most $\lfloor n/2\rfloor$.  For odd $n$, the central minors up to that height are all small.  For even $n$, half the central minors at height $\lfloor n/2\rfloor$ are not small, but the non-small ones are part of truncated Aztec diamonds that give overlapping sets of row and column indices.
\end{proof}

\begin{corollary}\label{minor-Laurent}
  Any minor of a matrix with disjoint row and column indices
  is a Laurent polynomial in the small central minors of the matrix.
\end{corollary}
\begin{proof}
  Any minor is a polynomial in the matrix entries,
  and by Corollary~\ref{contiguous-Laurent},
  the off-diagonal matrix entries
  are Laurent polynomials in the small central minors.
\end{proof}

\subsection{Positivity}

A noninterlaced minor of a matrix is a Laurent polynomial in the
small central minors (by Corollary~\ref{minor-Laurent}), but we do not know
whether or not the coefficients are always positive.  But as the next
theorem shows, noninterlaced minors are positive rational functions
of the small central minors.  Our positivity results in this section were
discussed in the survey \cite{kenyon:surfaces}.

\begin{theorem}\label{centralminorthm}
  Any noninterlaced minor of a matrix is a positive rational function
  of its small central minors.
\end{theorem}

\begin{proof}
Let $A$ be a matrix, and let $a,b,c$ index some of its columns, and $z,d$ index some of its rows.  It is elementary that
\[
0=
\begin{vmatrix}
A_{z,a}&A_{z,b}&A_{z,c}\\
A_{z,a}&A_{z,b}&A_{z,c}\\
A_{d,a}&A_{d,b}&A_{d,c}\\
\end{vmatrix}
=
A_{z,a}
\begin{vmatrix}
A_{z,b}&A_{z,c}\\
A_{d,b}&A_{d,c}\\
\end{vmatrix}
-A_{z,b}
\begin{vmatrix}
A_{z,a}&A_{z,c}\\
A_{d,a}&A_{d,c}\\
\end{vmatrix}
+A_{z,c}
\begin{vmatrix}
A_{z,a}&A_{z,b}\\
A_{d,a}&A_{d,b}\\
\end{vmatrix}.
\]
Suppose $A$ is invertible, and let $M$ denote its inverse; $a,b,c$ index rows of $M$ and $z,d$ index columns of $M$.
Dividing through by $(\det A)^2$ and using Jacobi's identity, we obtain
\[ 0 = \det M_{\widehat{a}}^{\widehat{z}} \det M_{\widehat{b,c}}^{\widehat{z,d}} - \det M_{\widehat{b}}^{\widehat{z}} \det M_{\widehat{a,c}}^{\widehat{z,d}} + \det M_{\widehat{c}}^{\widehat{z}} \det M_{\widehat{a,b}}^{\widehat{z,d}}. \]
Since this is a polynomial equation that holds generically, it must always hold.
As column $z$ is always excluded, it need not index an actual column of $M$.
Dropping $z$ and rearranging terms we obtain
\[ \det M_{\widehat{b}} \det M_{\widehat{a,c}}^{\widehat{d}}  = \det M_{\widehat{a}} \det M_{\widehat{b,c}}^{\widehat{d}} + \det M_{\widehat{c}} \det M_{\widehat{a,b}}^{\widehat{d}}. \]
For example, if $M=M^{9,8,7,6,5}_{1,2,3,4}$ and $a,b,c,d=9,8,5,4$, we have
\[
\minors{9}{1/9,2/7,3/6,4/5} =
\frac{\minors{9}{1/8,2/7,3/6,4/5} \minors{9}{1/9,2/7,3/6}
  +
\minors{9}{1/9,2/8,3/7,4/6}  \minors{9}{1/7,2/6,3/5}}{\minors{9}{1/8,2/7,3/6}}.
\]

We call transformations of this type the ``jaw move''.
For a given noninterlaced minor interspersed with at least one isolated node, we can take $b$ to be one of the interspersed isolated nodes, and $a$ and $c$ to be the first and last of the nodes on the same side as $b$, and $d$ to be either first or last node on the other side as $b$.  With this choice of $a,b,c,d$, the jaw move expresses the original determinant as a positive rational function of ``simpler'' noninterlaced minors, where a determinant is simpler if it has fewer strands, or else the same number of strands but fewer interspersed isolated nodes.

By repeated application of the jaw move, any noninterlaced minor can be expressed as a positive rational function of noninterlaced contiguous minors.
Then we can use
Corollary~\ref{contiguous-Laurent}
to express the noninterlaced contiguous minors 
as positive Laurent polynomials in
the small central minors.
\end{proof}

\begin{theorem}\label{symmetric-centralminorthm}
  Any noninterlaced minor of a symmetric $n\times n$ matrix
  is a positive rational function
  of the small central minors for which $1\leq x\leq n$.
\end{theorem}

\begin{proof}
This essentially follows from Theorem~\ref{centralminorthm}.
If the matrix is symmetric and $n$ is odd, then $\CM_{x,y}=\CM_{x+n,y}$,
so we can restrict to using the central minors for which $1\leq x\leq n$.
The remaining case to check is for $n$ even, since the (minimally) off-center contiguous minors come in pairs, of which only one is a small central minor, even for symmetric matrices.  Using a condensation move we can express a minimally off-center contiguous minor
 in terms of its opposite minimally off-center contiguous
minor
and central contiguous minors
as shown below:
\[
\minors{8}{1/7,2/6} = \frac{\minors{8}{1/6,2/5} \ \minors{8}{2/7,3/6} + \minors{8}{2/6} \ \minors{8}{1/7,2/6,3/5}}{\minors{8}{2/6,3/5}}\,.\qedhere
\]
\end{proof}

\begin{corollary}\label{centralminorcor}
  If each of the $\binom{n}{2}$ small central minors of the response matrix of a network is positive then the network is well-connected.
\end{corollary}

We believe that these results have analogs for non-well-connected networks:
\begin{conjecture}
  For a minimal network with $m$ edges, there is a ``base set'' $S$ of $m$
  noninterlaced minors of the response matrix, analogous to the set of
  small central minors of a well-connected network, such that any
  nonzero noninterlaced minor of the response matrix is (1) a Laurent
  polynomial in minors from $S$, and (2) a positive function of minors
  from $S$.
\end{conjecture}

\subsection{Domino regions for semicontiguous minors}

We believe that Theorem~\ref{domino} can be extended to semicontiguous
minors, i.e., minors $M_A^B$ where only one of $A$ or $B$ is
contiguous.  
\begin{conjecture}
For every semicontiguous minor, there is an associated domino tiling region,
such that the semicontiguous minor is a Laurent polynomial central minors
with each Laurent monomial corresponding to a domino tiling of the region,
according to the same rule as in Theorem~\ref{domino}.
\end{conjecture}
We include here some example formulas which illustrate this conjecture.

\vspace{6pt}
\noindent\textbf{Note:} After submitting our article, this conjecture
was confirmed by Tri Lai \cite{lai:semicontiguous}.
It would be interesting if there were a further generalization
for when both sets of indices are noncontiguous.

\[
\minors{9}{2/3,1/5}
=
\begin{tikzpicture}[baseline=2cm-2.5pt]
\filldraw[blue!15!white,draw=black] (11,0) rectangle (12,1);
\filldraw[blue!15!white,draw=black] (11,1) rectangle (12,2);
\filldraw[blue!15!white,draw=black] (12,0) rectangle (13,1);
\filldraw[blue!15!white,draw=black] (12,1) rectangle (13,2);
\filldraw[blue!15!white,draw=black] (12,2) rectangle (13,3);
\filldraw[blue!15!white,draw=black] (13,0) rectangle (14,1);
\filldraw[blue!15!white,draw=black] (13,1) rectangle (14,2);
\filldraw[blue!15!white,draw=black] (13,2) rectangle (14,3);
\filldraw[blue!15!white,draw=black] (13,3) rectangle (14,4);
\filldraw[blue!15!white,draw=black] (14,1) rectangle (15,2);
\filldraw[blue!15!white,draw=black] (14,2) rectangle (15,3);
\filldraw[blue!15!white,draw=black] (14,3) rectangle (15,4);
\filldraw[blue!15!white,draw=black] (15,1) rectangle (16,2);
\filldraw[blue!15!white,draw=black] (15,2) rectangle (16,3);
\minorat{9}{8/3}{(11,1)}
\minorat{9}{}{(12,0)}
\minorat{9}{8/4}{(12,1)}
\minorat{9}{9/3,8/4}{(12,2)}
\minorat{9}{}{(13,0)}
\minorat{9}{9/4}{(13,1)}
\minorat{9}{9/4,8/5}{(13,2)}
\minorat{9}{1/3,9/4,8/5}{(13,3)}
\minorat{9}{9/5}{(14,1)}
\minorat{9}{1/4,9/5}{(14,2)}
\minorat{9}{1/4,9/5,8/6}{(14,3)}
\minorat{9}{2/3,1/4,9/5,8/6}{(14,4)}
\minorat{9}{1/5}{(15,1)}
\minorat{9}{1/5,9/6}{(15,2)}
\minorat{9}{2/4,1/5,9/6}{(15,3)}
\minorat{9}{2/5,1/6}{(16,2)}
\end{tikzpicture}
\]

\[
\minors{15}{4/7,3/8,2/9,1/11}
=
\begin{tikzpicture}[baseline=2.5cm-2.5pt]
\filldraw[blue!15!white,draw=black] (24,2) rectangle (25,3);
\filldraw[blue!15!white,draw=black] (24,3) rectangle (25,4);
\filldraw[blue!15!white,draw=black] (25,1) rectangle (26,2);
\filldraw[blue!15!white,draw=black] (25,2) rectangle (26,3);
\filldraw[blue!15!white,draw=black] (25,3) rectangle (26,4);
\filldraw[blue!15!white,draw=black] (25,4) rectangle (26,5);
\filldraw[blue!15!white,draw=black] (26,0) rectangle (27,1);
\filldraw[blue!15!white,draw=black] (26,1) rectangle (27,2);
\filldraw[blue!15!white,draw=black] (26,2) rectangle (27,3);
\filldraw[blue!15!white,draw=black] (26,3) rectangle (27,4);
\filldraw[blue!15!white,draw=black] (26,4) rectangle (27,5);
\filldraw[blue!15!white,draw=black] (27,0) rectangle (28,1);
\filldraw[blue!15!white,draw=black] (27,1) rectangle (28,2);
\filldraw[blue!15!white,draw=black] (27,2) rectangle (28,3);
\filldraw[blue!15!white,draw=black] (28,0) rectangle (29,1);
\filldraw[blue!15!white,draw=black] (28,1) rectangle (29,2);
\minorat{15}{2/7,1/8,15/9}{(24,3)}
\minorat{15}{2/8,1/9}{(25,2)}
\minorat{15}{2/8,1/9,15/10}{(25,3)}
\minorat{15}{3/7,2/8,1/9,15/10}{(25,4)}
\minorat{15}{2/9}{(26,1)}
\minorat{15}{2/9,1/10}{(26,2)}
\minorat{15}{3/8,2/9,1/10}{(26,3)}
\minorat{15}{3/8,2/9,1/10,15/11}{(26,4)}
\minorat{15}{4/7,3/8,2/9,1/10,15/11}{(26,5)}
\minorat{15}{}{(27,0)}
\minorat{15}{2/10}{(27,1)}
\minorat{15}{3/9,2/10}{(27,2)}
\minorat{15}{3/9,2/10,1/11}{(27,3)}
\minorat{15}{4/8,3/9,2/10,1/11}{(27,4)}
\minorat{15}{}{(28,0)}
\minorat{15}{3/10}{(28,1)}
\minorat{15}{3/10,2/11}{(28,2)}
\minorat{15}{3/11}{(29,1)}
\end{tikzpicture}
\]

\[
\minors{13}{3/7,2/9,1/10}
=
\begin{tikzpicture}[baseline=1.5cm-2.5pt]
\filldraw[blue!15!white,draw=black] (21,0) rectangle (22,1);
\filldraw[blue!15!white,draw=black] (21,1) rectangle (22,2);
\filldraw[blue!15!white,draw=black] (22,0) rectangle (23,1);
\filldraw[blue!15!white,draw=black] (22,1) rectangle (23,2);
\filldraw[blue!15!white,draw=black] (22,2) rectangle (23,3);
\filldraw[blue!15!white,draw=black] (23,0) rectangle (24,1);
\filldraw[blue!15!white,draw=black] (23,1) rectangle (24,2);
\filldraw[blue!15!white,draw=black] (23,2) rectangle (24,3);
\filldraw[blue!15!white,draw=black] (24,1) rectangle (25,2);
\filldraw[blue!15!white,draw=black] (24,2) rectangle (25,3);
\minorat{13}{1/7}{(21,1)}
\minorat{13}{}{(22,0)}
\minorat{13}{1/8}{(22,1)}
\minorat{13}{2/7,1/8}{(22,2)}
\minorat{13}{}{(23,0)}
\minorat{13}{2/8}{(23,1)}
\minorat{13}{2/8,1/9}{(23,2)}
\minorat{13}{3/7,2/8,1/9}{(23,3)}
\minorat{13}{2/9}{(24,1)}
\minorat{13}{3/8,2/9}{(24,2)}
\minorat{13}{3/8,2/9,1/10}{(24,3)}
\minorat{13}{3/9,2/10}{(25,2)}
\end{tikzpicture}
\]

\[
\minors{13}{4/5,3/7,2/8,1/9}
=
\begin{tikzpicture}[baseline=3cm-2.5pt]
\filldraw[blue!15!white,draw=black] (17,0) rectangle (18,1);
\filldraw[blue!15!white,draw=black] (17,1) rectangle (18,2);
\filldraw[blue!15!white,draw=black] (18,0) rectangle (19,1);
\filldraw[blue!15!white,draw=black] (18,1) rectangle (19,2);
\filldraw[blue!15!white,draw=black] (18,2) rectangle (19,3);
\filldraw[blue!15!white,draw=black] (19,0) rectangle (20,1);
\filldraw[blue!15!white,draw=black] (19,1) rectangle (20,2);
\filldraw[blue!15!white,draw=black] (19,2) rectangle (20,3);
\filldraw[blue!15!white,draw=black] (19,3) rectangle (20,4);
\filldraw[blue!15!white,draw=black] (20,1) rectangle (21,2);
\filldraw[blue!15!white,draw=black] (20,2) rectangle (21,3);
\filldraw[blue!15!white,draw=black] (20,3) rectangle (21,4);
\filldraw[blue!15!white,draw=black] (20,4) rectangle (21,5);
\filldraw[blue!15!white,draw=black] (21,2) rectangle (22,3);
\filldraw[blue!15!white,draw=black] (21,3) rectangle (22,4);
\filldraw[blue!15!white,draw=black] (21,4) rectangle (22,5);
\filldraw[blue!15!white,draw=black] (21,5) rectangle (22,6);
\filldraw[blue!15!white,draw=black] (22,3) rectangle (23,4);
\filldraw[blue!15!white,draw=black] (22,4) rectangle (23,5);
\filldraw[blue!15!white,draw=black] (22,5) rectangle (23,6);
\filldraw[blue!15!white,draw=black] (23,3) rectangle (24,4);
\filldraw[blue!15!white,draw=black] (23,4) rectangle (24,5);
\minorat{13}{12/5}{(17,1)}
\minorat{13}{}{(18,0)}
\minorat{13}{12/6}{(18,1)}
\minorat{13}{13/5,12/6}{(18,2)}
\minorat{13}{}{(19,0)}
\minorat{13}{13/6}{(19,1)}
\minorat{13}{13/6,12/7}{(19,2)}
\minorat{13}{1/5,13/6,12/7}{(19,3)}
\minorat{13}{13/7}{(20,1)}
\minorat{13}{1/6,13/7}{(20,2)}
\minorat{13}{1/6,13/7,12/8}{(20,3)}
\minorat{13}{2/5,1/6,13/7,12/8}{(20,4)}
\minorat{13}{1/7,13/8}{(21,2)}
\minorat{13}{2/6,1/7,13/8}{(21,3)}
\minorat{13}{2/6,1/7,13/8,12/9}{(21,4)}
\minorat{13}{3/5,2/6,1/7,13/8,12/9}{(21,5)}
\minorat{13}{2/7,1/8,13/9}{(22,3)}
\minorat{13}{3/6,2/7,1/8,13/9}{(22,4)}
\minorat{13}{3/6,2/7,1/8,13/9,12/10}{(22,5)}
\minorat{13}{4/5,3/6,2/7,1/8,13/9,12/10}{(22,6)}
\minorat{13}{3/7,2/8,1/9}{(23,3)}
\minorat{13}{3/7,2/8,1/9,13/10}{(23,4)}
\minorat{13}{4/6,3/7,2/8,1/9,13/10}{(23,5)}
\minorat{13}{4/7,3/8,2/9,1/10}{(24,4)}
\end{tikzpicture}
\]

\[
\minors{13}{4/5,3/6,2/7,1/9}
=
\begin{tikzpicture}[baseline=3cm-2.5pt]
\filldraw[blue!15!white,draw=black] (19,2) rectangle (20,3);
\filldraw[blue!15!white,draw=black] (19,3) rectangle (20,4);
\filldraw[blue!15!white,draw=black] (20,1) rectangle (21,2);
\filldraw[blue!15!white,draw=black] (20,2) rectangle (21,3);
\filldraw[blue!15!white,draw=black] (20,3) rectangle (21,4);
\filldraw[blue!15!white,draw=black] (20,4) rectangle (21,5);
\filldraw[blue!15!white,draw=black] (21,0) rectangle (22,1);
\filldraw[blue!15!white,draw=black] (21,1) rectangle (22,2);
\filldraw[blue!15!white,draw=black] (21,2) rectangle (22,3);
\filldraw[blue!15!white,draw=black] (21,3) rectangle (22,4);
\filldraw[blue!15!white,draw=black] (21,4) rectangle (22,5);
\filldraw[blue!15!white,draw=black] (21,5) rectangle (22,6);
\filldraw[blue!15!white,draw=black] (22,0) rectangle (23,1);
\filldraw[blue!15!white,draw=black] (22,1) rectangle (23,2);
\filldraw[blue!15!white,draw=black] (22,2) rectangle (23,3);
\filldraw[blue!15!white,draw=black] (22,3) rectangle (23,4);
\filldraw[blue!15!white,draw=black] (22,4) rectangle (23,5);
\filldraw[blue!15!white,draw=black] (22,5) rectangle (23,6);
\filldraw[blue!15!white,draw=black] (23,0) rectangle (24,1);
\filldraw[blue!15!white,draw=black] (23,1) rectangle (24,2);
\filldraw[blue!15!white,draw=black] (23,3) rectangle (24,4);
\filldraw[blue!15!white,draw=black] (23,4) rectangle (24,5);
\minorat{13}{1/5,13/6,12/7}{(19,3)}
\minorat{13}{1/6,13/7}{(20,2)}
\minorat{13}{1/6,13/7,12/8}{(20,3)}
\minorat{13}{2/5,1/6,13/7,12/8}{(20,4)}
\minorat{13}{1/7}{(21,1)}
\minorat{13}{1/7,13/8}{(21,2)}
\minorat{13}{2/6,1/7,13/8}{(21,3)}
\minorat{13}{2/6,1/7,13/8,12/9}{(21,4)}
\minorat{13}{3/5,2/6,1/7,13/8,12/9}{(21,5)}
\minorat{13}{}{(22,0)}
\minorat{13}{1/8}{(22,1)}
\minorat{13}{2/7,1/8}{(22,2)}
\minorat{13}{2/7,1/8,13/9}{(22,3)}
\minorat{13}{3/6,2/7,1/8,13/9}{(22,4)}
\minorat{13}{3/6,2/7,1/8,13/9,12/10}{(22,5)}
\minorat{13}{4/5,3/6,2/7,1/8,13/9,12/10}{(22,6)}
\minorat{13}{}{(23,0)}
\minorat{13}{2/8}{(23,1)}
\minorat{13}{2/8,1/9}{(23,2)}
\minorat{13}{3/7,2/8,1/9}{(23,3)}
\minorat{13}{3/7,2/8,1/9,13/10}{(23,4)}
\minorat{13}{4/6,3/7,2/8,1/9,13/10}{(23,5)}
\minorat{13}{2/9}{(24,1)}
\minorat{13}{4/7,3/8,2/9,1/10}{(24,4)}
\end{tikzpicture}
\]

\[
\minors{15}{5/6,4/7,3/8,2/9,1/11}
=
\begin{tikzpicture}[baseline=4cm-2.5pt]
\filldraw[blue!15!white,draw=black] (23,3) rectangle (24,4);
\filldraw[blue!15!white,draw=black] (23,4) rectangle (24,5);
\filldraw[blue!15!white,draw=black] (24,2) rectangle (25,3);
\filldraw[blue!15!white,draw=black] (24,3) rectangle (25,4);
\filldraw[blue!15!white,draw=black] (24,4) rectangle (25,5);
\filldraw[blue!15!white,draw=black] (24,5) rectangle (25,6);
\filldraw[blue!15!white,draw=black] (25,1) rectangle (26,2);
\filldraw[blue!15!white,draw=black] (25,2) rectangle (26,3);
\filldraw[blue!15!white,draw=black] (25,3) rectangle (26,4);
\filldraw[blue!15!white,draw=black] (25,4) rectangle (26,5);
\filldraw[blue!15!white,draw=black] (25,5) rectangle (26,6);
\filldraw[blue!15!white,draw=black] (25,6) rectangle (26,7);
\filldraw[blue!15!white,draw=black] (26,0) rectangle (27,1);
\filldraw[blue!15!white,draw=black] (26,1) rectangle (27,2);
\filldraw[blue!15!white,draw=black] (26,2) rectangle (27,3);
\filldraw[blue!15!white,draw=black] (26,3) rectangle (27,4);
\filldraw[blue!15!white,draw=black] (26,4) rectangle (27,5);
\filldraw[blue!15!white,draw=black] (26,5) rectangle (27,6);
\filldraw[blue!15!white,draw=black] (26,6) rectangle (27,7);
\filldraw[blue!15!white,draw=black] (27,0) rectangle (28,1);
\filldraw[blue!15!white,draw=black] (27,1) rectangle (28,2);
\filldraw[blue!15!white,draw=black] (27,2) rectangle (28,3);
\filldraw[blue!15!white,draw=black] (27,4) rectangle (28,5);
\filldraw[blue!15!white,draw=black] (27,5) rectangle (28,6);
\filldraw[blue!15!white,draw=black] (28,0) rectangle (29,1);
\filldraw[blue!15!white,draw=black] (28,1) rectangle (29,2);
\minorat{15}{2/6,1/7,15/8,14/9}{(23,4)}
\minorat{15}{2/7,1/8,15/9}{(24,3)}
\minorat{15}{2/7,1/8,15/9,14/10}{(24,4)}
\minorat{15}{3/6,2/7,1/8,15/9,14/10}{(24,5)}
\minorat{15}{2/8,1/9}{(25,2)}
\minorat{15}{2/8,1/9,15/10}{(25,3)}
\minorat{15}{3/7,2/8,1/9,15/10}{(25,4)}
\minorat{15}{3/7,2/8,1/9,15/10,14/11}{(25,5)}
\minorat{15}{4/6,3/7,2/8,1/9,15/10,14/11}{(25,6)}
\minorat{15}{2/9}{(26,1)}
\minorat{15}{2/9,1/10}{(26,2)}
\minorat{15}{3/8,2/9,1/10}{(26,3)}
\minorat{15}{3/8,2/9,1/10,15/11}{(26,4)}
\minorat{15}{4/7,3/8,2/9,1/10,15/11}{(26,5)}
\minorat{15}{4/7,3/8,2/9,1/10,15/11,14/12}{(26,6)}
\minorat{15}{5/6,4/7,3/8,2/9,1/10,15/11,14/12}{(26,7)}
\minorat{15}{}{(27,0)}
\minorat{15}{2/10}{(27,1)}
\minorat{15}{3/9,2/10}{(27,2)}
\minorat{15}{3/9,2/10,1/11}{(27,3)}
\minorat{15}{4/8,3/9,2/10,1/11}{(27,4)}
\minorat{15}{4/8,3/9,2/10,1/11,15/12}{(27,5)}
\minorat{15}{5/7,4/8,3/9,2/10,1/11,15/12}{(27,6)}
\minorat{15}{}{(28,0)}
\minorat{15}{3/10}{(28,1)}
\minorat{15}{3/10,2/11}{(28,2)}
\minorat{15}{5/8,4/9,3/10,2/11,1/12}{(28,5)}
\minorat{15}{3/11}{(29,1)}
\end{tikzpicture}
\]

\[
\minors{11}{4/5,3/6,2/10,1/11,1/11}
=
\begin{tikzpicture}[baseline=2.5cm-2.5pt]
\filldraw[blue!15!white,draw=black] (0,0) rectangle (1,1);
\filldraw[blue!15!white,draw=black] (0,1) rectangle (1,2);
\filldraw[blue!15!white,draw=black] (0,2) rectangle (1,3);
\filldraw[blue!15!white,draw=black] (0,3) rectangle (1,4);
\filldraw[blue!15!white,draw=black] (0,4) rectangle (1,5);
\filldraw[blue!15!white,draw=black] (1,0) rectangle (2,1);
\filldraw[blue!15!white,draw=black] (1,1) rectangle (2,2);
\filldraw[blue!15!white,draw=black] (1,2) rectangle (2,3);
\filldraw[blue!15!white,draw=black] (1,3) rectangle (2,4);
\filldraw[blue!15!white,draw=black] (1,4) rectangle (2,5);
\filldraw[blue!15!white,draw=black] (2,0) rectangle (3,1);
\filldraw[blue!15!white,draw=black] (2,1) rectangle (3,2);
\filldraw[blue!15!white,draw=black] (2,2) rectangle (3,3);
\filldraw[blue!15!white,draw=black] (2,3) rectangle (3,4);
\filldraw[blue!15!white,draw=black] (3,1) rectangle (4,2);
\filldraw[blue!15!white,draw=black] (3,2) rectangle (4,3);
\filldraw[blue!15!white,draw=black] (-5,1) rectangle (-4,2);
\filldraw[blue!15!white,draw=black] (-5,2) rectangle (-4,3);
\filldraw[blue!15!white,draw=black] (-4,0) rectangle (-3,1);
\filldraw[blue!15!white,draw=black] (-4,1) rectangle (-3,2);
\filldraw[blue!15!white,draw=black] (-4,2) rectangle (-3,3);
\filldraw[blue!15!white,draw=black] (-4,3) rectangle (-3,4);
\filldraw[blue!15!white,draw=black] (-3,0) rectangle (-2,1);
\filldraw[blue!15!white,draw=black] (-3,1) rectangle (-2,2);
\filldraw[blue!15!white,draw=black] (-3,2) rectangle (-2,3);
\filldraw[blue!15!white,draw=black] (-3,3) rectangle (-2,4);
\filldraw[blue!15!white,draw=black] (-3,4) rectangle (-2,5);
\filldraw[blue!15!white,draw=black] (-2,0) rectangle (-1,1);
\filldraw[blue!15!white,draw=black] (-2,1) rectangle (-1,2);
\filldraw[blue!15!white,draw=black] (-2,2) rectangle (-1,3);
\filldraw[blue!15!white,draw=black] (-2,3) rectangle (-1,4);
\filldraw[blue!15!white,draw=black] (-2,4) rectangle (-1,5);
\filldraw[blue!15!white,draw=black] (-1,0) rectangle (0,1);
\filldraw[blue!15!white,draw=black] (-1,1) rectangle (0,2);
\filldraw[blue!15!white,draw=black] (-1,2) rectangle (0,3);
\filldraw[blue!15!white,draw=black] (-1,3) rectangle (0,4);
\filldraw[blue!15!white,draw=black] (0,0) rectangle (1,1);
\filldraw[blue!15!white,draw=black] (0,1) rectangle (1,2);
\filldraw[blue!15!white,draw=black] (0,2) rectangle (1,3);
\filldraw[blue!15!white,draw=black] (0,3) rectangle (1,4);
\minorat{11}{}{(1,0)}
\minorat{11}{3/9}{(1,1)}
\minorat{11}{4/8,3/9}{(1,2)}
\minorat{11}{4/8,3/9,2/10}{(1,3)}
\minorat{11}{5/7,4/8,3/9,2/10}{(1,4)}
\minorat{11}{5/7,4/8,3/9,2/10,1/11}{(1,5)}
\minorat{11}{}{(2,0)}
\minorat{11}{4/9}{(2,1)}
\minorat{11}{4/9,3/10}{(2,2)}
\minorat{11}{5/8,4/9,3/10}{(2,3)}
\minorat{11}{5/8,4/9,3/10,2/11}{(2,4)}
\minorat{11}{4/10}{(3,1)}
\minorat{11}{5/9,4/10}{(3,2)}
\minorat{11}{5/9,4/10,3/11}{(3,3)}
\minorat{11}{5/10,4/11}{(4,2)}
\minorat{11}{1/5,11/6}{(-5,2)}
\minorat{11}{1/6}{(-4,1)}
\minorat{11}{1/6,11/7}{(-4,2)}
\minorat{11}{2/5,1/6,11/7}{(-4,3)}
\minorat{11}{}{(-3,0)}
\minorat{11}{1/7}{(-3,1)}
\minorat{11}{2/6,1/7}{(-3,2)}
\minorat{11}{2/6,1/7,11/8}{(-3,3)}
\minorat{11}{3/5,2/6,1/7,11/8}{(-3,4)}
\minorat{11}{}{(-2,0)}
\minorat{11}{2/7}{(-2,1)}
\minorat{11}{2/7,1/8}{(-2,2)}
\minorat{11}{3/6,2/7,1/8}{(-2,3)}
\minorat{11}{3/6,2/7,1/8,11/9}{(-2,4)}
\minorat{11}{4/5,3/6,2/7,1/8,11/9}{(-2,5)}
\minorat{11}{}{(-1,0)}
\minorat{11}{2/8}{(-1,1)}
\minorat{11}{3/7,2/8}{(-1,2)}
\minorat{11}{3/7,2/8,1/9}{(-1,3)}
\minorat{11}{4/6,3/7,2/8,1/9}{(-1,4)}
\minorat{11}{}{(0,0)}
\minorat{11}{3/8}{(0,1)}
\minorat{11}{3/8,2/9}{(0,2)}
\minorat{11}{4/7,3/8,2/9}{(0,3)}
\minorat{11}{4/7,3/8,2/9,1/10}{(0,4)}
\end{tikzpicture}
\]

\section{Question}

Is there a decision-procedure which on any response matrix makes at most
$\binom{n}{2}$ positivity tests and determines the strand
matching for which all the conductances are positive?

\newcommand{\MRhref}[2]{\href{http://www.ams.org/mathscinet-getitem?mr=#1}{MR#1}}
\def\@rst #1 #2other{#1}
\newcommand\MR[1]{\relax\ifhmode\unskip\spacefactor3000 \space\fi
  \MRhref{\expandafter\@rst #1 other}{#1}}

\newcommand{\arXiv}[1]{\href{http://arxiv.org/abs/#1}{arXiv:#1}}
\newcommand{\arxiv}[1]{\href{http://arxiv.org/abs/#1}{#1}}

%\phantomsection
%\pdfbookmark[1]{References}{bib}
%\bibliographystyle{hmralphaabbrv}
%\small
\bibliographystyle{siamplain}
\bibliography{dets}

\end{document}